\pdfoutput=1
\documentclass{article}
\usepackage[style=alphabetic, sorting=nyt , backend=biber, maxnames=10, minnames=10, maxalphanames=10, minalphanames=10, giveninits=false]{biblatex}

\renewbibmacro{in:}{}
\DeclareFieldFormat{pages}{#1}
\DeclareFieldFormat{postnote}{#1}
\DeclareFieldFormat{multipostnote}{#1}

\DeclareFieldFormat[article, book, thesis, online]{title}{\emph{#1}}
\DeclareFieldFormat[article, online]{journaltitle}{#1}

\renewbibmacro*{volume+number+eid}{%
  \printfield{volume}%
  \setunit{\addcomma\space}%
  \printfield{eid} }     

\renewbibmacro*{note+pages}{
  \printfield{note}%
 \setunit{\bibpagespunct}%
  \printfield{pages}%
  \newunit}

\usepackage{graphicx}
\usepackage{amsmath}
\usepackage{amssymb}
\usepackage{amsfonts}
\usepackage{IEEEtrantools}
\usepackage{amsthm}
\usepackage{fancyhdr}
\usepackage{xcolor} \usepackage[dvipsnames]{xcolor}
\usepackage{bm}
\usepackage{extarrows}
\usepackage{enumitem}   \setlist[enumerate]{itemsep=1pt, topsep=1pt, label=(\arabic*)}
\usepackage{stmaryrd} 
\usepackage{comment}
\usepackage[titles]{tocloft}
\usepackage{tikz-cd}

\usepackage{float}
\usepackage{mathtools}

\usepackage{thmtools, thm-restate}
\usepackage{adjustbox}

\usepackage{titlesec}

\newcommand{\nlb}{\binoppenalty=10000  \relpenalty=10000 }

\newcommand{\Q}{\mathbb{Q}}
\newcommand{\R}{\mathbb{R}}
\newcommand{\C}{\mathbb{C}}
\newcommand{\Z}{\mathbb{Z}}
\newcommand{\dps}{\displaystyle}
\newcommand{\ess}{\mathrm{ed}}
\newcommand{\ed}{\mathrm{ed}}
\newcommand{\trdeg}{\mathrm{tr.deg}}

\newcommand{\Lie}{\mathrm{Lie}\,}
\newcommand{\fadele}{\mathbb{A}_f}
\newcommand{\Id}{\mathrm{Id}}
\newcommand{\id}{\mathrm{id}}
\newcommand{\Stab}{\mathrm{Stab}}
\newcommand{\Cent}{\mathrm{Cent}}
\newcommand{\Spec}{\mathrm{Spec}\,}
\newcommand{\DS}{\mathbb{S}} 
\newcommand{\Hom}{\mathrm{Hom}} 
\newcommand{\pr}{\mathrm{pr}}
\newcommand{\mA}{\mathcal{A}}
\newcommand{\mX}{\mathcal{X}}
\newcommand{\mC}{\mathcal{C}}
\newcommand{\mU}{\mathcal{U}}
\newcommand{\mH}{\mathcal{H}}
\newcommand{\mY}{\mathcal{Y}}
\newcommand{\mZ}{\mathcal{Z}}
\newcommand{\mUbar}{\overbar{\mathcal{U}}} 
\newcommand{\Int}{\mathrm{Int}}
\newcommand{\Cl}{\mathrm{Cl}}
\newcommand{\im}{\mathrm{im}}
\newcommand{\imbar}{\overbar{\mathrm{im}}}
\newcommand{\bs}{\backslash}
\newcommand{\ord}{\mathrm{ord}}
\newcommand{\rank}{\mathrm{rank}}
\newcommand{\bG}{\mathbb{G}}
\newcommand{\bS}{\mathbb{S}}

\newcommand{\hP}{{\widetilde{P}}}
\newcommand{\hU}{{\widetilde{U}}}
\newcommand{\hQ}{{ \widetilde{Q}}}
\newcommand{\hT}{{\widetilde{T}}}
\newcommand{\hX}{{\widetilde{\mathcal{X}}}}
\newcommand{\hGamma}{{\widetilde{\Gamma}}}
\newcommand{\hSigma}{{\widetilde{\Sigma}}}
\newcommand{\him}{{\widetilde{\mathrm{im}}}}

\newcommand{\overbar}[1]{\mkern 1.5mu\overline{\mkern-1.5mu#1\mkern-1.5mu}\mkern 1.5mu}

\theoremstyle{plain}
\newtheorem{sectheorem}{Theorem}[section]

\newtheorem{secproposition}[sectheorem]{Proposition}

\newtheorem{subsectheorem}{Theorem}[subsection]
\newtheorem{subseccorollary}[subsectheorem]{Corollary}
\newtheorem{subseclemma}[subsectheorem]{Lemma}
\newtheorem{subsecproposition}[subsectheorem]{Proposition}
\theoremstyle{definition}
\newtheorem{secdefinition}[sectheorem]{Definition}
\newtheorem{secexample}[sectheorem]{Example}

\newtheorem{secremark}[sectheorem]{Remark}

\newtheorem{subsecdefinition}[subsectheorem]{Definition}
\newtheorem{subsecexample}[subsectheorem]{Example}

\newtheorem{subsecremark}[subsectheorem]{Remark}

\theoremstyle{remark}

\usepackage[hypertexnames=false ]{hyperref}
\hypersetup{bookmarks=true,%
                         colorlinks=true, %
                         linktocpage=true,%
                        linkcolor=red,%
                         citecolor=blue,%
                        }

\title{\vspace{-12pt}\large \textbf{Lower Bounds on Essential Dimension for Congruence Covers of Mixed Shimura Varieties}}
\author{Qi'An Chen}
\date{}

\def\temp{&} \catcode`&=\active \let&=\temp
\begin{document}
\maketitle

\renewcommand{\abstractname}{\vspace{-2.5pt} \centering \textbf{Abstract}}

\begin{abstract}
\noindent
We use the fixed point method and toroidal compactifications to establish general lower bounds for the essential dimension of congruence covers $\Gamma' \backslash \mathcal{X}^0 \rightarrow \Gamma \backslash \mathcal{X}^0$ of mixed Shimura varieties.\ Our main result shows that $\mathrm{ed}_{\mathbb{C}}(\Gamma' \backslash \mathcal{X}^0 \rightarrow \Gamma \backslash \mathcal{X}^0; p)$ is bounded from below by the dimension of certain unipotent subgroups associated with the rational boundary components of the given mixed Shi\-mu\-ra datum.\ This generalizes theorems of Brosnan and Fakhruddin to the case of an arbitrary mixed Shimura datum.\ As~a con\-se\-quence, we obtain incompressibility results for congruence covers of universal families of principally polarized abelian varieties.\ We also describe explicit fixed points for $(GL_2, \mathcal{H}_2)$ and $(V \rtimes GL_2, \mathcal{Y}_2)$.
\end{abstract}

\tableofcontents

\section{Introduction}\label{sec: introduction}

Fix a base field $k$. Let $K$ be an extension field of $k$, and $A$ a finite \'etale $K$-algebra. As in \cite{Reichstein22}, we say that $A/K$ descends to $A'/K'$ if $K'$ is a subfield of $K$ and $A'$ is a $K'$-algebra such that 
$K\otimes_{K'}A'$ and $A$ are isomorphic as $K$-algebras. Define the \emph{essential dimension} of $A/K$ to be
\[
\ess_k(A/K) \coloneq \min_{A'/K'}\trdeg_kK'.
\] 

More generally, in the language of schemes, let $f\colon  X\rightarrow Y$ be a generically \'etale morphism of schemes of finite type over $k$ with $Y$ integral. We say that $f$ can be compressed to another morphism $f' \colon X'\rightarrow Y'$ of the same type if the following conditions hold.
\begin{enumerate}[label=(\arabic*)]
\item There exists a dominant rational morphism $g\colon Y \dashrightarrow Y'$. 
\item $X\rightarrow Y$ and $Y\times_{Y'}X'\rightarrow Y$ are isomorphic over the generic point of $Y$.
\end{enumerate} 

Equivalently, we have a diagram 
\[
\begin{tikzcd}
X \arrow[d, "f"] \arrow[r, dashed, "g'"] & X' \arrow[d, " f' "]\\
Y \arrow[r, "g", dashed]                                     &  Y'
\end{tikzcd}
\]
and an open dense subscheme $V\subset Y$ on which $g$ is defined and over which $f$ is the pullback of $f'$ along $g$. We define the essential dimension $\ess_k(X \rightarrow Y)$ of $f$ as the minimum dimension of $Y'$ for which such a compression exists.

If $p$ is a prime, define the \emph{essential $p$-dimension} $\ess_k(X \rightarrow Y; p)$ of $f$ as the minimum of $\ess_k(Y'' \times_Y X \rightarrow Y'')$, taken over all generically finite morphisms $Y'' \rightarrow Y$ of degree prime to $p$ with $Y''$ integral. By convention, the essential dimension at $p=0$ is the usual one without base change. We also say that $X\rightarrow Y$ is incompressible (resp. $p$-incompressible) if $\ess_k(X\rightarrow Y)$ (resp. $\nlb \ess_k(X\rightarrow Y; p)$) equals $\dim Y$.

Suppose the above $f \colon X \rightarrow Y$ is a $G$-torsor with $G$ a finite constant group scheme over $k$. Then $f$ is finite \'etale. The essential dimension $\ess_k(G)$ (resp. essential $p$-dimension $\ess_k(G ; p)$) of $G$ is defined as the maximum of $\ess_k(X\rightarrow Y)$ (resp. $\ess_k(X\rightarrow Y; p)$) for all $G$-torsors $X\rightarrow Y$, where we require all $X'\rightarrow Y'$ to be $G$-torsors and $g' \colon X\dashrightarrow X'$ to be $G$-equivariant.

The notion of essential dimension can be traced back to Klein and Kronecker and is closely related to Hilbert's 13th problem (which involves another invariant called resolvent degree); see \cite{FarbWolfson19} and \cite{Reichstein22}. It was first formally defined for finite groups by Buhler and Reichstein in \cite{BuhlerReichstein} and then extended to algebraic groups in \cite{Reichstein00} and to functors $\mathcal{F} \colon Fields_k\rightarrow Sets$ in \cite{BerhuyFavi}. Intuitively, the essential dimension of an algebraic group measures the complexity of its torsors. For a detailed survey, see \cite{Merkurjev}. For more historical remarks, see, for example, \cite{BuhlerReichstein}, \cite{FarbWolfson19}, and \cite{FarbKisinWolfson21}. We will present a collection of results needed on essential dimension in Section \ref{sec: essential dimension}.

Shimura varieties are objects of both algebro-geometric and arithmetic interest. Recall that a (pure) \emph{Shimura datum} is a pair $(G, \mathcal{X})$. Here $G$ is a connected reductive $\mathbb{Q}$-group and $\mathcal{X}$ is a $G(\mathbb{R})$-conjugacy class of morphisms $\mathbb{S}\rightarrow G_\mathbb{R}$ with $\mathbb{S}=\mathrm{Res}_{\mathbb{C}/\mathbb{R}}(\mathbb{G}_{m})$ the Deligne torus. They are required to satisfy certain axioms (cf.\ \cite{Milne05}) which guarantee that $\mathcal{X}$ is a union of Hermitian symmetric domains parametrizing pure Hodge structures on $\Lie G$.

Let $G(\Q)$ act diagonally on $\dps \mathcal{X} \times G(\mathbb{A} _f)$ from the left, and let an open compact subgroup $K_f\subset G(\mathbb{A}_f)$ act on $G(\mathbb{A}_f)$ from the right (with $\fadele$ the ring of finite adeles). The double quotient 
\[
M^{K_f}(G, \mathcal{X})(\C) \coloneq G(\Q) \backslash \mathcal{X}  \times  G(\mathbb{A}_f)/K_f
\]
has a natural structure of a quasi-projective complex algebraic variety. If $\mathcal{X}^0$ is a connected component of $\mathcal{X}$ and $p_f\in G(\fadele)$, then $\Gamma \coloneq  \mathrm{Stab}_{G(\mathbb{Q})}(\mathcal{X}^0) \cap p_f K_f p_f^{-1}$ is an arithmetic subgroup of $G(\Q)$ that acts on $\mathcal{X}^0$. The inclusion map $\mathcal{X}^0\times \{p_f\} \rightarrowtail \mathcal{X}\times G(\fadele)$ induces an embedding $\Gamma \backslash \mathcal{X}^0 \rightarrowtail M^{K_f}(G, \mathcal{X})(\C)$, realizing $\Gamma \backslash \mathcal{X}^0$ as a connected component. Such a $\Gamma \backslash \mathcal{X}^0$ is called a locally symmetric variety. Given two open compact subgroups $K'_f \subset K_f \subset G(\fadele)$, we get a covering $\Gamma ' \backslash \mathcal{X}^0 \rightarrow \Gamma \backslash \mathcal{X}^{0}$. When $K_f'$ is a normal subgroup of $K_f$, the action of the finite group $\Gamma / \Gamma '$ is transitive on the fiber.

Now it makes sense to study $\dps \ed_\mathbb{C}(\Gamma ' \backslash \mathcal{X}^0 \rightarrow \Gamma \backslash \mathcal{X}^{0})$ or $\dps \ed_\mathbb{C}(\Gamma ' \backslash \mathcal{X}^0 \rightarrow \Gamma \backslash \mathcal{X}^{0}; p)$. Results of this kind first appeared in the paper \cite{FarbKisinWolfson21} of Farb, Kisin, and Wolfson, where they computed the essential dimension for many congruence covers of Shimura varieties using arithmetic methods. For example, let $\mathcal{A}_{g, N}$ be the moduli space of principally polarized abelian varieties with level-$N$ structure. They proved that $ \ed_{\mathbb{C}}(\mA_{g, nN}\rightarrow \mA_{g,N};p)=\dim \mathcal{A}_g=\binom{g+1}{2}$ for $g\geq 2$, $\gcd(n, N)=1$, and $ p\, |\, n$. Hence the cover $\mA_{g, nN}\rightarrow \mA_{g,N}$ is $p$-incompressible. They also obtained incompressibility results on the ``principal $p$-level coverings" $\Gamma' \backslash \mathcal{X} \rightarrow \Gamma \backslash \mathcal{X}$, including cases in which the locally symmetric varieties are proper. Their results are strengthened in many respects in \cite{FarbKisinWolfson24} via non-vanishing characteristic classes in prismatic cohomology and toroidal compactifications. Fakhruddin and Saini also proved incompressibility of congruence covers for certain pure Shimura varieties of PEL type. See \cite[5]{FakhruddinSaini22}.

On the other hand, in \cite{BrosnanFakhruddin}, Brosnan and Fakhruddin obtained lower bounds on the essential dimension of congruence covers of locally symmetric varieties by ``the fixed point method". This leads to new incompressibility theorems while reproving many of those in \cite{FarbKisinWolfson21} such as the above one concerning $\mathcal{A}_{g,N}$. Roughly speaking, the $p$-rank of a finite abelian group $H$ contained in $\Gamma/\Gamma'$ bounds the essential $p$-dimension from below, provided that this $H$ has a smooth fixed point in some $H$-equivariant (partial) compactification of $\Gamma ' \backslash \mathcal{X}^0$. They proved the existence of smooth fixed points for some $H$-equivariant compactifications of the cover $\Gamma'  \backslash  \mathcal{X}^0 \rightarrow \Gamma  \backslash  \mathcal{X}^0$ and applied the fixed point method to get their results. 

The approaches of \cite{FarbKisinWolfson21} and \cite{BrosnanFakhruddin} complement each other. Each applies to a class of objects inaccessible to the other. For a more refined comparison, we refer the reader to the introduction of \cite{BrosnanFakhruddin}, and some remarks in Section 5 thereof.

In this paper, we generalize Theorems 33 and 34 in \cite{BrosnanFakhruddin} to a broader class of covers, while recovering some of the main results in \cite{FarbKisinWolfson21} and \cite{BrosnanFakhruddin}. Namely, with the fixed point method and toroidal compactifications, we obtain general lower bounds on essential dimension for congruence covers of mixed Shimura varieties.

A \emph{mixed Shimura datum} $(P,\mathcal{X})$ consists of a connected algebraic group $P$ over $\Q$ and a finite covering $\mathcal{X}$ of a conjugacy class of morphisms $h_x \colon \mathbb{S}_\mathbb{C}\rightarrow P_\mathbb{C}$, subject to some axioms formulated by Pink in his thesis \cite{Pink}. The group $P$ may be non-reductive with unipotent radical $W$. For a mixed Shimura datum $(P, \mathcal{X})$, each  $h_x$ induces a rational mixed Hodge structure of weight $\{0, -1, -2\}$ on $\Lie P_\mathbb{C}$, with $\Lie W$ corresponding to the subspace of weight $\leq -1$. Moreover, there is a $\Q$-subgroup $U$, central in $W$ and normal in $P$, whose Lie algebra corresponds to the subspace of weight $-2$. The group $P(\R) {\cdot} U(\C)\subset P(\C)$ acts transitively on $\mathcal{X}$. As in the pure case, the double quotient space
\[
M^{K_f}(P, \mX)(\C) \coloneq P(\Q) \backslash \mathcal{X}\times P(\fadele) / K_f
\]
carries the structure of a quasi-projective complex algebraic variety.

Mixed Shimura data arise naturally. The rational boundary components of a pure or mixed Shimura datum are mixed Shimura data. (The definition of a rational boundary component will be given in Subsection \ref{subsec: rational boundary components}.) Thus in this sense they are more natural objects of study. One can also construct mixed Shimura data as unipotent extensions from pure ones, as explained in Subsection \ref{subsec: mixed shimura data}. Despite some notational sophistication, the general theory of mixed Shimura varieties resembles that of pure Shimura varieties. (Toroidal compactifications are also closely related to the Baily-Borel compactification. See for example Proposition \ref{subsecproposition: toroidal compactifications, map to baily-borel}.) The systematic treatment of the subject is \cite{Pink}. See also the introduction therein for more insightful remarks. There is also a brief introduction to mixed Shimura varieties in \cite{Milne88}.

We illustrate our main theorem with the following example.  Let $\mY_{g,d} $ be the universal family of principally polarized abelian varieties with full level-$d$ structure. We have a canonical projection map $\mY_{g,d} \rightarrow \mA_{g,d}$ onto the moduli space $\mA_{g,d}$ mentioned above. In fact, $\mY_{g,d} \rightarrow \mA_{g,d}$ is an abelian scheme for $d\geq 3$; see \cite[7.9]{GIT}. ($\mY_{g,d}$ is an example of \emph{Kuga varieties} and can be constructed explicitly from the vector space $V$ of dimension $2g$ on which $Sp_{2g}$ acts. A Kuga datum $(P, \mY)$ is a mixed Shimura datum with the group $U$ trivial.)   For a prime $p$ not dividing $d$, we get a canonical cover $\varphi \colon \mY_{g, pd} \rightarrow \mY_{g, d}$ defined as the composition $\mY_{g, pd} \overset{p}{\rightarrow} \mY_{g, pd} \rightarrow \mY_{g,d}$. Here the first map is multiplication by $p$ on each fiber and the second is the pullback of $\mA_{g, pd} \rightarrow \mA_{g, d}$ along $\mY_{g,d} \rightarrow \mA_{g,d}$, as shown in the diagram.
\begin{equation*}\adjustbox{center}{
\begin{tikzcd}
 \mY_{g,pd} \arrow[r, "p"] \arrow[dr] &
\mY_{g,pd} \arrow[r]  \arrow[d] &
\mY_{g,d} \arrow[d]\\
 &
\mA_{g, pd} \arrow[r, "\mathrm{pullback\,}" yshift=12 pt] &
\mA_{g,d}
\end{tikzcd}
}
\end{equation*}

It follows from our main theorem that $\ed_{\C}(\mY_{g, pd} \overset{\varphi}{\rightarrow} \mY_{g,d};p) = g+\binom{g+1}{2}$ for $p \nmid d\geq 3$. Thus the cover $\varphi \colon \mY_{g, pd} \rightarrow \mY_{g,d}$ is $p$-incompressible. This is part of Theorem~\ref{subsectheorem: kuga varieties, incompressibility}. It is a conjecture of Brosnan that the multiplication-by-$p$ map $[p] \colon  A  \rightarrow A$ is $p$-incompressible for an abelian variety $A$ over a field of characteristic $0$ (cf.\ \cite[6.1,\ 6.7]{FakhruddinSaini22}, \cite[2.3.5]{FarbKisinWolfson24}, and \cite{KollarZhuang25}).  The $p$-incompressibility of $\varphi \colon \mY_{g, pd} \rightarrow \mY_{g,d}$ may be viewed as a combination of this conjecture with the $p$-incompressibility of $\mA_{g, pd}\rightarrow \mA_{g,d}$. We refer the reader to Subsection \ref{subsec: kuga varieties} for more details on this example. See also Remark \ref{subsecremark: examples, brosnan's conjecture} for a discussion of recent progress on the conjecture.

In Subsection \ref{subsec: universal families over siegel modular varieties}, we obtain another such result by deducing from the main theorem the $p$-incompressibility of $\mZ_{g, pd}\rightarrow \mZ_{g, d}$, where $\mZ_{g, d}$ is the universal family over $\mY_{g,d}$ of certain line bundles on principally polarized abelian varieties with level structures. See Theorem~\ref{subsectheorem: universal families over siegel modular varieties, incompressibility} for the precise statement.

Finally, we state our main theorem. For a mixed Shimura datum $(P,\mathcal{X})$, let $Z(P)$  be the center of $P$ and define 
\[
\dps  \Id_{Z(P)(\Q)}(\mathcal{X}) \coloneq \left \{z\in Z(P)(\Q) ~\lvert ~z|_\mathcal{X}=\mathrm{id}   \right \}.
\] This is the subgroup of elements in $Z(P)(\Q)$ acting trivially on $\mathcal{X}$.

Let $K_f' \subset K_f \subset P(\fadele)$ be open compact subgroups. We define the following groups.
\begin{IEEEeqnarray*}{rCl}
\Gamma           &\coloneq&\Stab_{P(\Q)}(\mathcal{X}^0)\cap \bigl (p_f K_fp_f^{-1}\bigr)
\end{IEEEeqnarray*}
$\Gamma$ is the stabilizer in $P(\Q)$ of $\mX^0 \times \{p_f K_f\} \subset \mX \times P(\fadele)/K_f$. Hence $\Gamma \bs \mX^0$ is a connected component of $M^{K_f}(P, \mX)(\C)$.
\begin{IEEEeqnarray*}{rCl}
\Gamma_Z      &\coloneq&   \Id_{Z(P)(\Q)}(\mathcal{X}) \cap \bigl (  p_f K_fp_f^{-1} \bigr )\\
\Delta&\coloneq&\Gamma/\Gamma_Z
\end{IEEEeqnarray*}
$\Gamma_Z$ lies in the kernel of the action of $\Gamma$ on $\mX^0$. Indeed, as we shall see in Lemma \ref{subseclemma: mixed shimura varieties, centralizer lemma}, when $K_f$ is neat, it is also equal to the stabilizer in $\Gamma$ of any point $x\in \mX^0$. Therefore, $\Delta = \Gamma / \Gamma_Z$ acts freely on $\mX^0$ in the neat case.

Any rational boundary component $(P_1, \mathcal{X}_1)$ of $(P,\mathcal{X})$ is also a mixed Shimura datum and thus has a corresponding commutative subgroup $U_1 \unlhd P_1$. Since $U_1$ is a unipotent group, $U_1(\R)$ stabilizes any connected component $\mathcal{X}^0$ of $\mathcal{X}$. 
\begin{IEEEeqnarray*}{rCl}
\Gamma_{ZU_1}&\coloneq& \bigl ( \Id_{Z(P)(\Q)} (\mathcal{X}) \times U_1(\Q) \bigr ) \cap \bigl (p_f K_fp_f^{-1} \bigr)
\end{IEEEeqnarray*}
$\Gamma_{ZU_1}$ is introduced mainly to define the group $\Gamma_{U_1}$ below, which is the image of $\Gamma_{ZU_1}$ under the projection $\dps \Id_{Z(P)(\Q)} (\mathcal{X}) \times U_1(\Q) \rightarrow U_1(\Q)$.
\begin{IEEEeqnarray*}{rCl}
\Gamma_{U_1}&\coloneq&\frac{\Gamma_{ZU_1}}{\Gamma_Z} \subset \frac{\Gamma}{\Gamma_Z}=\Delta
\end{IEEEeqnarray*}
By definition, $\Gamma_{U_1}$ is a lattice in $U_1(\Q)$.  It acts on $U_1(\C)$ by translation. Also, $\Gamma_{ZU_1}$ acts on $U_1(\C)$ through $\Gamma_{U_1}$.
In Subsection \ref{subsec: proof of main theorem}, we will also define analogous groups $\Omega_{Z_1U_1}$ and $\Omega_{U_1}$, which are closely related to the structure of the unipotent fiber (cf.\ \cite[3.13]{Pink} or Proposition \ref{subsecproposition: mixed shimura varieties, unipotent fiber}).

The groups $\Gamma'$, $\Gamma_Z'$, $\Delta'$, $\Gamma_{ZU_1}'$, and $\Gamma_{U_1}'$ are defined similarly.

\begin{sectheorem}
\label{sectheorem: introduction, main theorem}
Let $(P,\mathcal{X})$ be a mixed Shimura datum and $\mX^0 \subset \mX$ any connected component.\ Suppose $(P_1, \mathcal{X}_1)$ is any rational boundary component of $(P, \mathcal{X})$ such that $\mX^0 \subset \mX^+_{(P_1, \mX_1)}$ (as in Proposition \ref{subsecproposition: rational boundary components, rational boundary components}).\ Given neat open compact subgroups $K_f' \unlhd K_f \subset P(\fadele)$ and $p_f \in P(\fadele)$, define $\Gamma$, $\Gamma_Z$, $\Delta$, and $\Gamma_{U_1}$ as above. Let $\Sigma$ be a $K_f$-admissible complete cone decomposition for $(P, \mathcal{X})$ that defines a smooth projective toroidal compactification $M^{K_f}(P,\mathcal{X},\Sigma)(\C)$ (as in Theorem \ref{subsectheorem: toroidal compactifications, projective structure}) and denote by $\left (\Gamma \backslash \mathcal{X}^0 \right )_\Sigma \subset M^{K_f}(P, \mX, \Sigma)(\C)$ the connected component containing $\Gamma \backslash \mathcal{X}^0$. Then 
\begin{enumerate}[label=(\arabic*)]
\item The cover $\dps \Gamma' \backslash \mathcal{X}^0 \rightarrow \Gamma \backslash \mathcal{X}^0$ is a $\Delta/\Delta'$-torsor.

\item Any $\Delta/\Delta'$-equivariant partial compactification $Y$ containing $\Gamma' \backslash \mathcal{X}^0$ with a proper morphism $Y \rightarrow   (\Gamma \backslash \mathcal{X}^0 )_\Sigma$ extending $ \Gamma' \backslash \mathcal{X}^0 \rightarrow \Gamma \backslash \mathcal{X}^0$ admits a $\Gamma_{U_1}/\Gamma_{U_1}'$-fixed point.

\item For any prime $p>0$, we have
\begin{equation*}
\dps \ed_\C(\Gamma' \backslash \mathcal{X}^0 \rightarrow \Gamma \backslash \mathcal{X}^0 ; p)\geq \mathrm{rank}_p (\Gamma_{U_1}/\Gamma_{U_1}') (\leq \dim U_1).
\end{equation*}

\item The optimal lower bound can be attained: For any neat open compact subgroup $K_f$ and any prime $p>0$, there exists a neat open compact subgroup $K_f' \unlhd K_f$ in $P(\fadele)$ such that $\dps \mathrm{rank}_p (\Gamma_{U_1}/\Gamma_{U_1}')=\dim U_1$.

In particular, $\dps \ed_\C(\Gamma' \backslash \mathcal{X}^0 \rightarrow \Gamma \backslash \mathcal{X}^0 ; p)\geq \dim {U_1}$. 

\end{enumerate}

\end{sectheorem}

\begin{secremark}
We write $\rank_p(A)$ for the $p$-rank of a finite abelian group $A$. For a variety $X$, a partial compactification of $X$ is any variety $Y$ that contains $X$ as a Zariski-open dense subvariety.
\end{secremark}

\begin{secremark}
The condition $\mX^0 \subset \mX^+_{(P_1, \mX_1)}$ is not really a restriction. Indeed, given any $\mX^0$ and any $\Q$-admissible parabolic subgroup $Q \subset P$ (to be defined in Subsection \ref{subsec: rational boundary components}), there exists a $(P_1, \mX_1)$ associated with $Q$ such that $\mX^0 \subset \mX^+_{(P_1, \mX_1)}$. As $\dim U_1$ only depends on $Q$, we get equally good lower bounds for every $\mX^0$. 
\end{secremark}

\begin{secremark}
The results of \cite[33, 34]{BrosnanFakhruddin} are stated for semisimple algebraic groups over $\Q$ of Hermitian type and for those $\mX^0$ that are tube domains, respectively. The proofs require Borel's extension theorem, which is unavailable in the mixed setting. Our proof of Theorem \ref{sectheorem: introduction, main theorem} is based on a careful examination of the construction of toroidal compactifications. Although this process is more involved, it allows us to remove all the technical assumptions on algebraic groups and spaces, thereby enabling $(P,\mX)$ to be an arbitrary mixed Shimura datum. 
\end{secremark}

\begin{secremark}
Suppose $(P,\mX)$ is a pure Shimura datum and for some $K_f$ the Shimura variety $M^{K_f}(P,\mX)(\C)$ is compact. In this case one can show that $(P,\mX)$ has no proper rational boundary components and therefore $\dim U_1=0$. Thus Theorem \ref{sectheorem: introduction, main theorem} does not subsume results like \cite[7]{FarbKisinWolfson21} or \cite[3.3.19]{FarbKisinWolfson24}.

Let $(P, \mX)$ be a mixed Shimura datum with $W\subset P$ the unipotent radical. Suppose the quotient (pure) Shimura datum $(P,\mX)/W$ (as defined in Proposition \ref{subsecproposition: mixed shimura data, properties}) has no proper rational boundary components. Then the only rational boundary components of $(P,\mX)$ are the improper ones $(P_1, \mX_1)$ with $U_1=U$. If $U\neq 1$, we still get a non-trivial lower bound from $\dim U$, but sometimes this may be relatively small compared with $\dim_\C \mX^0$.
\end{secremark}

The terminology will be clarified in the subsequent sections. We will review the definition of a rational boundary component in \ref{subsec: rational boundary components} and the construction of  toroidal compactifications in \ref{subsec: toroidal compactifications}.

We conclude the introduction with an outline of the paper. In Section \ref{sec: essential dimension} we recall several elementary properties of, and key results on, essential dimension. Section \ref{sec: the fixed point method} contains the precise statement of the fixed point method and the theorem in \cite{BrosnanFakhruddin} on the existence of fixed points, needed in the proof of Theorem \ref{sectheorem: introduction, main theorem}. The delicate machinery of mixed Shimura varieties is introduced in Section \ref{sec: basics}, leading to the construction of their toroidal compactifications, as developed in \cite{Pink}. Finally, Section \ref{sec: proof of main theorem} completes the proof of Theorem \ref{sectheorem: introduction, main theorem} after the establishment of some technical lemmas. In Section \ref{sec: examples}, we present a few examples of lower bounds to illustrate our main theorem, including the incompressibility of $\varphi \colon \mY_{g, pd}\rightarrow \mY_{g,d}$ (in Subsection \ref{subsec: kuga varieties}). Section \ref{sec: examples of fixed points} is devoted to describing  two explicit fixed points on suitable toroidal compactifications for the data $(GL_{2,\Q}, \mH_2)$ and $(V\rtimes GL_{2,\Q}, \mY_2)$ respectively.


\section*{Acknowledgments}\label{sec: acknowledgments}
I sincerely thank my advisor Professor Patrick Brosnan for introducing me to the notion of essential dimension and for directing me toward this research topic. The numerous discussions I had with him and his consistent support have been integral to this project.

\section{Essential dimension} \label{sec: essential dimension}

In the introduction, we have defined $\ed_k(G)$ for a finite group $G$. In fact, essential dimension can be defined for any algebraic group $G$. We include the definition and some basic results below. Our main references for this section are \cite{Merkurjev} and \cite{BrosnanFakhruddin}.

Let $G$ be an algebraic group over a field $k$. A \emph{$G$-scheme} is a scheme of finite type over $k$ with a left $G$-action. A morphism of $G$-schemes $f\colon X\rightarrow Y$ is a \emph{$G$-torsor} if 
\begin{enumerate}
\item $Y$ is integral and $G$ acts trivially on $Y$;
\item $f$ is faithfully flat;
\item The morphism $G\times_kX\rightarrow X\times_YX \colon (g,x)\mapsto (g\cdot x, x)$ is an isomorphism. 
\end{enumerate}

A $G$-scheme is generically free if there exists an open dense $G$-invariant subscheme $U$ such that $U$ is the total space of a $G$-torsor $U\rightarrow Y$. (In some literature, our generically free $G$-scheme is called primitive, since $G$ acts transitively on the set of irreducible components of $X$.) Following \cite{Merkurjev}, we use $k(X)^G$ to denote the function field of $Y$.  Note that $k(X)^G$ does not depend on the choice of $U$. We denote the $G$-torsor coming from the generic fiber by $T_X\rightarrow \Spec k(X)^G$.

Given a generically free $G$-scheme $X$, a \emph{$G$-compression} is a $G$-equivariant dominant rational map $X\dashrightarrow X'$ to another generically free $G$-scheme $X'$. The essential dimension of $X$ (resp. of $G$) is defined to be 
\begin{IEEEeqnarray*}{rCl}
\ed_k(X,G)&\coloneq& 
\min_{X\dashrightarrow X'} \left \{\dim X' - \dim G \right \},\\
\ed_k(G)&\coloneq& \max_{X} \{\ed_k(X,G)\}.
\end{IEEEeqnarray*}
Similarly, the essential $p$-dimension of $X$ is defined to be 
\begin{IEEEeqnarray*}{rCl}
\ed_k(X,G; p)&\coloneq&\min_{Y''\dashrightarrow Y}\left \{\ed_k(Y''\times_Y X, G) \right \},\\
\ed_k(G; p)&\coloneq& \max_{X} \{\ed_k(X,G; p)\}.
\end{IEEEeqnarray*}
where the minimum is taken over all generically finite dominant rational maps $Y'' \dashrightarrow Y$ of degree prime to $p$.

A generically free $G$-scheme $X$ is \emph{incompressible} (resp. $p$-incompressible) if $\ed_k(X,G)=\dim X- \dim G$ (resp. $\ed_k(X,G;p)=\dim X- \dim G$). In other words, $X$ cannot be compressed to a generically free $G$-scheme of lower dimension. 

\begin{secproposition}[\cite{BerhuyFavi}]
\begin{IEEEeqnarray*}{rCl}
\ed_k(X,G)&=&\ed_k(T_X\rightarrow \Spec  k(X)^G), \\ \ed_k(G)&=&\max_{T\rightarrow \Spec K} \left \{\ed_k(T\rightarrow \Spec K)\right \}.
\end{IEEEeqnarray*}
Here $\ed_k(T\rightarrow \Spec K)$ is the minimum of the transcendence degree of a field of definition for the $G$-torsor $T\rightarrow \Spec K$.
\end{secproposition}

Among all generically free $G$-schemes, some have the highest essential dimension. This is reflected in the notion of versality. A $G$-scheme $X$ is \emph{versal} if every open dense $G$-invariant subscheme $U\subset X$ satisfies the condition that every generically free $G$-scheme $X'$ with $k(X')^G$ infinite admits a $G$-equivariant rational map $X' \dashrightarrow U$. Intuitively, this means that whenever $X$ can be compressed to some extent, $X'$ can be compressed at least to the same extent ``through $X$". Therefore, versal generically free $G$-schemes should have the highest essential dimension. By definition, any open dense $G$-invariant subscheme of $X$ is also versal. Further, if $X\dashrightarrow X'$ is a $G$-compression of generically free $G$-schemes with $X$ versal, then $X'$ is versal. 

By analogy, there is a notion of $p$-versality. A $G$-scheme $X$ is \emph{$p$-versal} if, for every $G$-invariant dense open $U\subset X$ and every generically free $G$-scheme $X'$ with $k(X')^G$ infinite, there exists a $G$-equivariant dominant rational map $X''\dashrightarrow X'$ with $k(X'')^G\leftarrow k(X')^G$ finite of degree prime to $p$ such that $X''$ admits a $G$-equivariant rational map to $U$. The observations in the preceding paragraph carry over to $p$-versality.

In summary, we have
\begin{secproposition} [{\cite[3.11]{Merkurjev}} and {\cite[6]{BrosnanFakhruddin}}]
~
\begin{enumerate}
\item If $X$ is a versal generically free $G$-scheme, then 
\[ \dps \ed_k(G)=\ed_k(X,G)=\min_{\mathrm{versal}\,X'} \{ \dim X' -\dim G \}.
\]
\item If $X$ is a $p$-versal generically free $G$-scheme, then 
\[ \dps \ed_k(G;p)=\ed_k(X,G;p)=\min_{p\text{-}\mathrm{versal}\,X'} \{ \dim X' -\dim G \}.
\]
\end{enumerate}
\end{secproposition}

Versal $G$-schemes abound. For every $k$-vector space $V$, we say that a representation $G \rightarrow \mathrm{GL}_V$ is generically free if $V$ as a $G$-scheme is generically free. In particular, if $G $ is a finite group, a representation is generically free if and only if it is faithful. It can be shown that every $G$-representation $V$ is $G$-versal (cf. {\cite[3.10]{Merkurjev}}).

Consider an algebraic subgroup $H\subset G$. Every generically free $G$-scheme remains generically free when regarded as an $H$-scheme and every $G$-compression is also an $H$-compression. Thus we have obtained that $\dps \ed_k(X,G; p) +\dim G \geq \ed_k(X,H;p) +\dim H$ for all $p\geq 0$.
\begin{secproposition} [{\cite[2.2]{BrosnanReichsteinVistoli}}]
 For every $p\geq 0$,
\begin{IEEEeqnarray*}{rCl} 
\ed_k(G;p)+\dim G&\geq& \ed_k(H;p)+\dim H.
\end{IEEEeqnarray*}
\end{secproposition}

In particular, if $H\subset G$ are finite groups, then $\ed_k(H;p) \leq \ed_k(G;p)$ for all $p\geq 0$.

We conclude this section by listing several results on the essential dimension of finite $p$-groups.

\begin{sectheorem}\label{sectheorem: essential dimension, essential dimension of finite groups}
~

\begin{enumerate}  \label{sectheorem: essential dimension, KarpenkoMerkurjev}
\item (\cite[{4.1, 5.1, 5.2}]{KarpenkoMerkurjev}) Let $G$ be a finite $p$-group and $k$ a field of characteristic not equal to $p$ containing a primitive $p$-th root of unity. Then $\ed_k(G;p)=\ed_k(G)$ is equal to the least dimension of a faithful representation of $G$ over $k$. Moreover
\begin{IEEEeqnarray*}{rCl}
\ed_k(G_1\times G_2)&=&  \ed_k(G_1)+\ed_k(G_2),\\
\ed_k(\Z/p^n\Z)            &=&   [k(\xi_{p^n}):k],
\end{IEEEeqnarray*}
where $\xi_{p^n}$ is a primitive $p^n$-th root of unity.

\item If $k$ is as above, then in particular, $\dps \ed_k\bigl ( \left (\Z/p\Z \right)^d \bigr )=d$. 
\end{enumerate}
\end{sectheorem}

\begin{secremark}
As indicated in the assumptions of the theorem, $\ed_k(G)$ or $\ed_k(G;p)$ depends on the ground field $k$.
\end{secremark}

\begin{secremark} \label{secremark: essential dimension, essential dimension of product groups}
The equality $\ed_k(G_1\times G_2)=\ed_k(G_1)+\ed_k(G_2)$ holds only for $p$-groups. In general, $\dps \ed_k(G_1\times G_2;p) \leq \ed_k(G_1;p)+\ed_k(G_2;p)$ for all $p\geq 0$  as a $G_1\times G_2$-torsor is naturally the product of a $G_1$-torsor with a $G_2$-torsor.
\end{secremark}


\section{The fixed point method} \label{sec: the fixed point method}

In this section, we briefly explain the fixed point method. Our starting point is the following theorem, which relates the existence of a $G$-fixed smooth $k$-point to $p$-versality.

\begin{sectheorem}[{\cite[8.6]{DuncanReichstein}}] \label{sectheorem: fixed point method, DuncanReichstein}
Let $G$ be a smooth affine algebraic group over a field $k$ and $X$ a generically free $G$-variety that is geometrically integral. If $G$ fixes a smooth point $x\in X(k)$, then $X$ is $p$-versal.
\end{sectheorem}

\begin{secremark} In practice, $G$ will be a finite group and $k$ will be the field of complex numbers.\end{secremark}

Combining Theorem \ref{sectheorem: essential dimension, essential dimension of finite groups} and Theorem \ref{sectheorem: fixed point method, DuncanReichstein}, while noting that the essential $p$-dimension is a birational invariant, we are led to 

\begin{sectheorem}[{\cite[10]{BrosnanFakhruddin}}]\label{sectheorem: fixed point method, fixed point method}
Let $G$ be a finite group and $X$ an irreducible generically free $G$-variety over a base field $k$ of characteristic not equal to $p$ containing a primitive $p$-th root of unity. Suppose $G$ has a subgroup $H$ such that $H$ is a $p$-group and has a smooth fixed point in $X$. Then
\begin{enumerate}[label=(\arabic*)]
\item $\dps \ed_k(X_0, G; p) \geq \ed_k(X_0, H; p)=\ed_k(H; p)$ for any open dense $G$-invariant subvariety $X_0\subset X$.
\item If $H=(\Z/p\Z)^r$, then $\ed_k(X_0, G; p) \geq \ed_k(X_0, H; p)=r$.
\end{enumerate}
\end{sectheorem}

To state the next theorem, we recall a few definitions from \cite{BrosnanFakhruddin}. We assume here that $k$ is an algebraically closed field.

\begin{secdefinition}
A \emph{toroidal singularity} is a pair $(S,\varphi)$, consisting of the following data.

\begin{enumerate}
\item $T$ is a torus over $k$ and $T_\sigma$ is the torus embedding with respect to a strongly convex rational polyhedral cone $\sigma\subset Y_{*}(T)_\R$ with $\dim \sigma =\dim T$. ($Y_*(T)$ is the cocharacter group of $T$.)

\item $S$ is a scheme over $k$ and $\dps \varphi \colon S\overset{\sim}{\rightarrow} \Spec \widehat{\mathcal{O}}_{T_\sigma,x}$ is an isomorphism of $S$ with the completion of the local ring of $T_\sigma$ at its $\sigma$-stratum (which is a torus fixed point). 
\end{enumerate}

\begin{secremark}
For more on toroidal singularities, see \cite{BrosnanFakhruddin}. For more preliminaries on torus embeddings, see Subsection \ref{subsec: torus embeddings}.
\end{secremark}

Given $(S,\varphi)$ as above, we define $S^\circ\subset S$ to be the inverse image of $T\subset T_\sigma$ under the morphism $S\overset{\sim}{\rightarrow} \Spec \widehat{\mathcal{O}}_{T_\sigma, x} \rightarrow T_\sigma$. (Thus we have a pullback diagram below, where $\Spec \widehat{O}_{T_\sigma,x} \cap T$ is the preimage of $T$.) Each ray $\ell \leq \sigma$ defines a stratum of codimension $1$, whose closure is a boundary divisor.

\begin{equation*}
\begin{tikzcd}
S^\circ \arrow[r,"\sim"']\arrow[d, tail]&\Spec \widehat{O}_{T_\sigma,x} \cap T \arrow[r]                     \arrow[d,tail] &T\arrow[d,tail]\\
S\arrow[r,"\varphi", "\sim" ']&\Spec \widehat{O}_{T_\sigma,x} \arrow[r]&T_\sigma
\end{tikzcd}
\end{equation*}

\begin{secexample}
Let $T=\mathbb{G}_m\times \mathbb{G}_m$. Then $Y_*(T)=\Z e_1 \oplus \Z e_2$ with $e_1 \colon \mathbb{G}_m \rightarrowtail \mathbb{G}_m\times \{1\}$ and $e_2 \colon \mathbb{G}_m \rightarrowtail \{1\} \times \mathbb{G}_m$. 
We take $\dps \sigma=\R_{\geq 0}\cdot e_1+\R_{\geq 0}\cdot e_2$. Then $T_\sigma= \Spec \C[x,y]=\mathbb{A}^2_\C$ with $x, y$ the dual generators of $e_1, e_2$ in $X^*(T)$, the character group of $T$. Then the $\sigma$-stratum is defined by the ideal $I=(x,y)$. The pair $(S,\varphi)=(\Spec \C [[x,y]], \mathrm{id})$ is a toroidal singularity.

Since $T=\Spec \C[x,y]_{xy}$, we have $S^\circ=\C[[x,y]]_{xy}$, the Laurent series ring in two variables.
\end{secexample}
\end{secdefinition}

\begin{sectheorem}[{\cite[16]{BrosnanFakhruddin}}] \label{sectheorem: fixed point method, existence of fixed points}
Let $Y$ be an irreducible variety over $k$ and let $f\colon X \rightarrow Y$ be an irreducible finite \'etale Galois cover with Galois group $G$. Let $S$ be a simplicial toroidal singularity and let $S^\circ$ be the complement of the boundary divisor. Suppose there exists a morphism $g \colon S^\circ \rightarrow Y$ such that the image of the composite $\dps \pi_1^\mathrm{et}(S^\circ, s)\overset{g_*}{\rightarrow} \pi_1^\mathrm{et}(Y, g(s)) \twoheadrightarrow G$ is a finite (abelian) group $A$ of order not divisible by $\mathrm{char}\,k$ (for $s\in S^\circ$ any geometric point). Assume that $g$ extends to a morphism $\overbar{g} \colon S\rightarrow \overbar{Y}$, where $\overbar{Y}$ is a partial compactification of $Y$. Then 
\begin{enumerate}[label=(\arabic*)]
\item Any $G$-equivariant partial compactification $\overbar{X}\supseteq X$ admitting a proper morphism $\overbar{f}\colon \overbar{X} \rightarrow \overbar{Y}$ extending $f$ has an $A$-fixed point.
\item Any smooth proper variety $X'$ with a $G$-action which is equivariantly birational to $X$ has an $A$-fixed point.
\end{enumerate}  
\end{sectheorem}

\begin{secremark}
As pointed out in {\cite[17]{BrosnanFakhruddin}}, when $k=\C$, 
it suffices to find holomorphic maps $g$ and $\overbar{g}$ on a sufficiently small open neighborhood of the torus fixed point $x\in T_\sigma$. Then we have the desired maps $S\rightarrow \overbar{Y}$ and $S^\circ \rightarrow Y$ since we can complete the induced homomorphism between analytic local rings.
\end{secremark}

\begin{secremark}
Suppose our $X$ is known to be smooth and there exists a $G$-equivariant proper $\overbar{f} \colon \overbar{X}\rightarrow \overbar{Y}$ as above. In order to apply Theorem \ref{sectheorem: fixed point method, fixed point method}, we need a smooth $\overbar{X}'$. This can be achieved by applying an equivariant resolution of singularities to $\overbar{X}$. See for example \cite{EncinasVillamayor98}.
\end{secremark}

\section{Basics on mixed Shimura varieties} \label{sec: basics}
          
         This section is mainly an exposition of the theory of mixed Shimura varieties. Our primary reference is \cite{Pink}. (Some lemmas are also deduced in Subsection \ref{subsec: mixed shimura varieties}.) The proof of our main theorem will make use of many elements in this section, especially those on rational boundary components and the construction of toroidal compactifications, as developed in Subsections \ref{subsec: rational boundary components} and \ref{subsec: toroidal compactifications}. Hopefully, this section can provide readers with more background on mixed Shimura varieties and help them navigate through \cite{Pink} on their own if necessary. Readers familiar with mixed Shimura varieties and eager to know the proof may proceed directly to Section \ref{sec: proof of main theorem}.

          \subsection{Mixed Hodge structures} \label{subsec: mixed hodge structures}

We introduce the notion of mixed Hodge structures in this subsection. 

Let $V$ be a finite-dimensional $\Q$-vector space. 

\begin{subsecdefinition}~

\begin{enumerate}
\item A pure rational Hodge structure on $V$ of weight $n\in \Z$ is a decomposition $\dps V_\C= \oplus_{(p,q)\in \Z \oplus \Z}V^{p,q}$ such that $V^{p,q}=0$ if $p+q\neq n$ and $\overbar{V^{p,q}}=V^{q,p}$. The Hodge filtration of this pure Hodge structure is the decreasing sequence of subspaces $\dps \{ F^{p}V=\oplus_{p \leq p'}V^{p',q} \}_{p \in \Z}$. Note that $F^pV \cap \overbar{F^qV}=V^{p,q}$. Thus the decomposition and the Hodge filtration determine each other.

\item A mixed rational Hodge structure on $V$ consists of an increasing, separated and exhaustive filtration of $\Q$-subspaces $\{W_nV\}_{n\in\Z}$ and a decreasing, separated and exhaustive filtration $\{F^pV\}_{p\in \Z}$ such that on each $Gr_nV\coloneq W_nV/W_{n-1}V$, $\{F^pV\}_{p\in \Z}$ induces a pure Hodge structure of weight $n$.
\end{enumerate}
\end{subsecdefinition} 

Given a mixed rational Hodge structure on $V$, we define the Hodge numbers $\dps h^{p,q} \coloneq \dim_\C (Gr_nV)^{p,q}$. By definition, $h^{p,q}=h^{q,p}$.
If $A\subset \Z \oplus \Z$ is a subset, a Hodge structure on $V$ has type $A$ if $h^{p,q}=0, ~(p,q)\notin A$. 

For example, a pure Hodge structure of type $\{(-1,0), (0,-1)\}$ is just a complex structure on $V$.

Morphisms of Hodge structures are $\Q$-linear maps $V\rightarrow V'$ preserving the weight and Hodge filtrations. With morphisms as arrows and Hodge-structured vector spaces as objects, the category of mixed rational Hodge structures is abelian. 
Given rational mixed Hodge structures on $V$,  $V_1$, and $V_2$, we have induced structures on $V_1\oplus V_2$, $V_1\otimes_\Q V_2$ and $V^{\vee}$ as follows.
\begin{IEEEeqnarray*}{rCl}
(V_1\oplus V_2)^{p,q}      &  =  &V_1^{p,q}\oplus V_2^{p, q}\\
(V_1 \otimes_\Q V_2)^{p,q} &  =  & 
\bigoplus_{\begin{subarray}{c}(p_1, q_1)+(p_2, q_2)\\=(p,q) \end{subarray}} 
      V_1^{p_1, q_1} \otimes_\Q V_2^{p_2, q_2}\\
(V^{\vee})^{p,q}                 &  = & V^{-p,-q}
\end{IEEEeqnarray*}
Thus there are also induced Hodge structures on $\Hom(V_1, V_2)=V_1^\vee \otimes_\Q V_2$. 

A mixed Hodge structure on $V$ is said to split over $\R$ if there exists a decomposition $V_\C= \bigoplus_{p,q} V^{p,q}$ such that $\dps W_nV=\oplus_{p+q\leq n}V^{p,q}$, $\dps F^pV=\oplus_{p \leq p'} V^{p',q}$  and $\overbar{V^{p,q}}=V^{q,p}$. Thus a pure Hodge structure splits over $\R$ by definition. Although not every mixed Hodge structure splits over $\R$, we have the following 

\begin{subsecproposition} [{\cite[1.2]{Pink}}]
Let $V$ be a finite-dimensional vector space equipped with a rational mixed Hodge structure.  
\begin{enumerate}
\item \hspace{-1pt}There exists a decomposition $V_\C\,{=}\, \bigoplus_{p,q} V^{p,q}$ such that $ W_nV\hspace{-2.5pt}=\hspace{-1pt}\bigoplus_{p+q\leq n}\hspace{-2pt}V^{p,q}$ and $ F^pV=\bigoplus_{p \leq p'} V^{p',q}$. 

\item \hspace{-1pt}There exists a unique decomposition as in (a) which satisfies $ \overbar{V^{p,q}} \equiv V^{q,p}\, \mathrm{mod} (\bigoplus_{p'<p, q'<q} V^{p', q'})$.
\end{enumerate}
\end{subsecproposition}

The above definitions are in the classical setting. Now we reformulate them in terms of representations of algebraic groups.

Let $\DS = \mathrm{Res}_{\C/\R}(\mathbb{G}_{m})$ be the Deligne torus. Then $\DS_\C \overset{\sim}{\rightarrow} \mathbb{G}_{m, \C} \times \mathbb{G}_{m, \C}$ and we can choose the isomorphism so that $\C^\times = \DS(\R) \rightarrow \DS(\C)\overset{\sim}{\rightarrow} \C^{\times} \times \C ^{\times}\colon  z \mapsto (z, \overbar{z})$. Complex conjugation acts on $\DS(\C)$ as $(z_1, z_2) \mapsto (\overbar{z_2}, \overbar{z_1})$. There are also the following frequently used morphisms of $\DS$:
\begin{IEEEeqnarray*}{rClll}
 \omega & \colon &\mathbb{G}_{m, \R}\rightarrow \DS, &~&z \mapsto (z^{-1}, z^{-1});\\
 \iota& \colon &S^1 \rightarrow \DS, &~&z \mapsto (z, z^{-1});\\
\mu_0 & \colon &   \mathbb{G}_{m,\C} \rightarrow \DS_\C, &~& z \mapsto (z, 1);\\
   N& \colon & \DS \rightarrow \mathbb{G}_{m, \R}, &~& (z_1, z_2) \mapsto z_1 z_2.
\end{IEEEeqnarray*}
 (Note that a morphism  over $\R$ is given in terms of its expression over $\C$. And one can easily check they are defined over $\R$.) By definition, we have a short exact sequence $1 \rightarrow S^1 \rightarrow \DS \overset{N}{\rightarrow } \mathbb{G}_{m,\R}\rightarrow 1$.

Since $\DS(\C)=\C^{\times}\times \C^{\times}$, a linear representation $h\colon  \DS_\C \rightarrow GL_{V,\C}$ is equivalent to a decomposition $\dps V_\C= \oplus_{p,q}V^{p,q}$, where $V^{p,q}$ is the eigenspace for the character $(z_1, z_2) \mapsto z_1^{p}z_2^{q}$ (and $GL_{V,\C}$ is the base change of $GL_V$ to $\C$). 

From now on, we will adopt the convention, to denote by $V^{p,q}$ the eigenspace for the character $(z_1, z_2) \mapsto z_1^{-p}z_2^{-q}$. Therefore through $h \circ \omega\colon \mathbb{G}_{m, \C} \rightarrow GL_{V,\C}$, the action on $V^{p,q}$ is through $t \mapsto t^{p+q}$. As before, $W^h_nV \coloneq \oplus_{p+q \leq n}V^{p,q}$ and $F^p_h V  \coloneq \oplus_{p \leq p'} V^{p',q}$. The superscript or subscript $h$ indicates the representation. We will often drop it if no confusion arises. 

Suppose the morphism $h\colon \DS_\C \rightarrow GL_{V,\C}$ defines a rational mixed Hodge structure on $V$. Since $\overbar{V^{p,q}}=V^{q,p}$ if and only if $h$ is defined over $\R$, the Hodge structure splits over $\R$ if and only if $h$ is defined over $\R$. And if $h\colon \DS \rightarrow GL_{V, \R}$ and $h'\colon \DS \rightarrow GL_{V', \R}$ induce pure Hodge structures on $V$ and $V'$ respectively, a morphism of pure Hodge structures is a $\Q$-linear map $V\rightarrow V' $ that is $\DS$-equivariant.

With the above perspective, let us make two more definitions. The \emph{Tate Hodge structure} $\Q(n)$ is defined to be the $\Q$-subspace $(2\pi i)^{n}\cdot \Q \subset \C$ with the pure Hodge structure of type $(-n, -n)$ on $\Q_\C=\C$. Suppose $h\colon \DS \rightarrow GL_{V, \R}$ defines a pure Hodge structure of weight $n$ on $V$. A polarization of the pure Hodge structure on $V$ is a morphism $\psi \colon V\otimes_\Q V  \rightarrow \Q(-n)$ of Hodge structures (of weight $2n$!) such that the pairing
\begin{IEEEeqnarray*}{rCccc}
(2\pi i)^n (\psi_\R)_{h(i)}&\colon & V_\R \times V_\R & \overset{\psi}{\longrightarrow} &\R(-n) \overset{(2\pi i)^n}{\longrightarrow} \R\\
&& (u,v)& \longmapsto& (2\pi i)^n \psi(u, h(i)v)
\end{IEEEeqnarray*}
is symmetric and positive definite. If $\psi \colon V\otimes_\Q V  \rightarrow \Q(-n)$ is known to be a morphism of Hodge structures, then $(2\pi i)^n (\psi_\R)_{h(i)}$ is symmetric if and only if (1) $n$ is even and $\psi$ is symmetric or (2) $n$ is odd and $\psi$ is alternating.

The next proposition provides some motivation for the definition of a mixed Shimura datum.

\begin{subsecproposition}[{\cite[1.4]{Pink}}]\label{subsecproposition: mixedhodge}
Let $P$ be a connected affine algebraic group over $\Q$ with unipotent radical $W$ and $\pi \colon P \rightarrow P/W \eqcolon G$. Suppose $h \colon \DS_\C \rightarrow P_\C$ is a morphism satisfying the following conditions.
\begin{enumerate}
\item $\pi \circ h \colon \DS_\C \rightarrow G_\C$ is defined over $\R$.
\item $\pi \circ h \circ \omega \colon \mathbb{G}_{m, \R} \rightarrow G_{\R}$ is a cocharacter defined over $\Q$ into the center of $G$.
\item In the decomposition induced on $\Lie P $ via the adjoint representation of $P$, $W^{\mathrm{Ad}_P\circ h}_{-1}(\Lie P)= \Lie W$.
\end{enumerate}

Then we have
\begin{enumerate}
\item For every $\Q$-representation $\rho \colon P \rightarrow GL_{V}$,  $\rho\circ h \colon \DS_\C \rightarrow P_\C \rightarrow GL_{V, \C}$ induces a rational mixed Hodge structure on $V$.
\item The weight filtration $\{ W^{\rho\circ h}_n V\}_{n\in \Z}$ is $P$-invariant.
\item For any $p \in P(\R)\cdot W(\C)$, the same conclusions as in (a) and (b) hold for the conjugate $\mathrm{int}(p) \circ h$. Moreover, $\mathrm{int}(p) \circ h$ induces the same weight filtration and Hodge numbers.
\end{enumerate}
\end{subsecproposition}

\begin{subseccorollary} \label{subseccorollary: mixed hodge}
Let $P$, $W$ and $\pi \colon P\rightarrow P/W=G$ be as above. Suppose $\lambda \colon \mathbb{G}_{m, \C} \rightarrow P_\C$ is a cocharacter such that $\pi \circ \lambda$ maps into $Z(G)_\C$ and is defined over $\Q$. Then for any $\Q$-representation $\rho \colon P \rightarrow GL_V$, the weight filtration $\{W^{\rho\circ \lambda}_nV\}_{n\in\Z}$ (defined by $\mathbb{G}_{m,\C}\overset{\lambda}{\rightarrow}P_\C \overset{\rho_{\C}}{\rightarrow}GL_{V, \C}$) is defined over $\Q$ and stable under $P$.
\end{subseccorollary}
\begin{proof}
This follows from the proof of Proposition 1.4 in \cite{Pink}, by applying the same argument only to the weight filtration.
\end{proof}

\begin{subsecremark}
In fact, in the second condition of Proposition \ref{subsecproposition: mixedhodge}, we can replace $\Q$ with any subfield $K$ between $\Q$ and $\R$ and the conclusion becomes that every representation of $P_K \rightarrow GL_{V_K}$ induces a mixed Hodge structure on $V_K$ with weight filtration consisting of $K$-subspaces.

A similar remark holds for Corollary  \ref{subseccorollary: mixed hodge} as well.
\end{subsecremark}


          \subsection{Mixed Shimura data} \label{subsec: mixed shimura data}

In this subsection, we introduce the notion of a mixed Shimura datum. The number of axioms may seem somewhat overwhelming at first. However, they become less heavy once the reader recognizes that they can be roughly divided into two parts: the ones concerning the mixed Hodge structure and the ones concerning the algebraic group.

\begin{subsecdefinition}[mixed Shimura datum, {\cite[2.1]{Pink}}] \label{subsecdefinition: mixed shimura data, mixed shimura datum}
Let $P$ be a connected affine algebraic group over $\Q$ with unipotent radical $W$ and a subgroup $U \subset W$ that is normal in $P$. Denote by $\pi \colon P \rightarrow G\coloneq P/W$ and $\pi' \colon P \rightarrow P/U$ the canonical projections. A \emph{mixed Shimura datum} is a pair $(P, \mathcal{X})$ equipped with a map $h \colon \mathcal{X} \rightarrow \Hom(\DS_\C , P_\C)$ satisfying the following axioms.

\begin{enumerate}[label=(A\arabic*)]
\item $\mathcal{X}$ is a (left) $P(\R)\cdot U(\C)$-homogeneous space and the map $h$ is $P(\R)\cdot U(\C)$-equivariant with finite fibers. 
\end{enumerate}
For some (and hence for all) $x \in \mathcal{X}$,
\begin{enumerate}[label=(A\arabic*)] \setcounter{enumi}{1}
\item $\pi' \circ h_x \colon \DS_\C \rightarrow P_\C  \rightarrow (P/U)_\C$ is defined over $\R$.

\item $\pi \circ h_x \circ \omega \colon \mathbb{G}_{m, \R} \rightarrow G$ is a cocharacter into the center $Z(G)$ of $G$.

\item $\mathrm{Ad}_{P}\circ h_x \colon \DS_\C \rightarrow GL_{\Lie P, \C}$ defines a rational mixed Hodge structure on $\Lie P$ of type
\[
\{(-1,1), (0,0), (1,-1)\} \cup \{(-1,0), (0, -1)\}  \cup \{ (-1, -1) \}.
\]

\item The weight filtration on $\Lie P$ is given as follows: $W_{-2}(\Lie P)=\Lie U$, $W_{-1}(\Lie P)= \Lie W$ and $W_{0}(\Lie P)= \Lie P$.

\item $\mathrm{int}(\pi \circ h_x(i))$ induces a Cartan involution on $G^{\mathrm{ad}}_\R$.

\item $G^{\mathrm{ad}}$ has no simple $\Q$-factors of compact type.

\item The connected center $Z(G)^0$ of $G$ acts on $U$ and $V=W/U$ through a torus which is an almost direct product of a $\Q$-split torus with a $\Q$-torus of compact type.
\end{enumerate} 
\end{subsecdefinition}

\begin{subsecremark}A few clarifications on the definition. 

\begin{enumerate}[label=(\roman*)]
\item For a variety $X$ over $\Q$ (not assumed to be irreducible), we say $X$ is of compact type if and only if $X(\R)$ is compact in the Euclidean topology.

\item We will often denote the image of $h$ by $\mathcal{H}$. Thus $\mathcal{H}$ is a $P(\R)\cdot U(\C)$-conjugacy class of morphisms $\DS_\C \rightarrow P_\C$.

\item There is a canonical $P(\R)\cdot U(\C)$-invariant complex structure on $\mathcal{H}$ and $\mathcal{X}$ ({\cite[2.1(c)]{Pink}}).

\item The fourth axiom above implies Griffiths transversality. It guarantees that there is a variation of mixed Hodge structures over the image of $\mathcal{H}$ in a suitable Grassmannian. See {\cite[1.10]{Pink}}.

\item For the motivation behind axiom (8), we refer the reader to {\cite[1.13]{Pink}}. 

\item As we will see, $U$ is central in $W$. Thus the statement of (8) makes sense.

\item If $W=1$, $(P, \mathcal{X})$ is called a pure Shimura datum. The definition given in \cite{Deligne79} is the special case where $W=1$ and $\mathcal{X}\cong h(\mathcal{X})$. The definition here is slightly more general in that we allow finite coverings.

\item ({\cite[1.19]{Pink}}) For the quotient Shimura datum $(P, \mathcal{X})/Z(P)$ (defined in Proposition \ref{subsecproposition: mixed shimura data, properties}), stronger axioms hold: 
    \begin{itemize}
        \item In Axiom (A3), $\pi \circ h_x \circ \omega \colon \mathbb{G}_{m,\C} \rightarrow G_\C$ is a cocharacter into $Z(G)_\C$ that is defined over $\Q$. 
        \item In Axiom (A8),  the identity component $Z(G)^{0}$ is an almost direct product of a $\Q$-split torus with a $\Q$-torus of compact type.
     \end{itemize}

\end{enumerate}
\end{subsecremark}

Let $(P, \mathcal{X})$ and $(P', \mathcal{X}')$ be two mixed Shimura data. A morphism between them is a pair $(\varphi, \xi)$, where $\varphi \colon P \rightarrow P'$ is a morphism of algebraic groups over $\Q$ and $\xi \colon \mathcal{X} \rightarrow \mathcal{X}'$ is a $P(\R)\cdot U(\C)$-equivariant map. They are required to fit into the following commutative diagram. (The bottom map comes from $\Hom(\DS_\C, P_\C) \rightarrow \Hom(\DS_\C, P'_\C)$ induced by $\varphi$.)
\begin{equation*}
\begin{tikzcd}
\mathcal{X} \arrow[r,"\xi"] \arrow[d, "h"']     &  \mathcal{X'} \arrow[d, " h' "]\\
\mathcal{H} \arrow[r, "\varphi \circ -"]                        &  \mathcal{H'}
\end{tikzcd}
\end{equation*}

Here are a few basic examples.
\begin{subsecexample}
Let $T$ be a torus over $\Q$. A mixed Shimura datum $(T, \mathcal{X})$ consists of a finite set $ \mathcal{X}$ on which $T(\R)$ acts transitively and a morphism $h \colon  \DS \rightarrow T_\R$, where $\mathcal{X} \rightarrow \{h\}$ is the constant map. 
\end{subsecexample}

\begin{subsecexample}
For a more specific example on torus, consider the group $\mathbb{G}_{m, \Q}$ together with the norm morphism $N \colon \DS \rightarrow \mathbb{G}_{m, \R},  z \mapsto z\overbar{z} $. Define $\mathcal{H}_0$ to be the set of isomorphisms $\Z \overset{\sim}{\rightarrow} 2\pi i \Z$. So $\mathcal{H}_0$ contains two points $n \mapsto \pm 2 \pi i n$. We let $\mathbb{G}_{m}(\R)=\R^\times$ act on $\mathcal{H}_0$ via $\R^\times \rightarrow \pi_0(\R^\times)=\{\pm 1\}$, where $-1$ acts on $\mathcal{H}_0$ by interchanging the two points. With the constant projection map $\mathcal{H}_0 \rightarrow \{N\}$, $(\mathbb{G}_{m, \Q}, \mathcal{H}_0)$ is a pure Shimura datum.
\end{subsecexample}

\begin{subsecexample} \label{subsecexample: mixed shimura data, siegel modular variety}
Another classical example comes from the Siegel modular variety. Fix $\psi \colon V \times V \rightarrow \Q$, a non-degenerate alternating bilinear form on a $\Q$-vector space $V$ of dimension $2n>0$. Define $CSp_V$ to be the $\Q$-subgroup of $GL_V$ consisting of linear automorphisms that preserve $\psi$ up to a scalar, i.e. $CSp_V(R)=\{g \in GL_V(R) \,\lvert \,\psi(g\cdot v, g\cdot v')=\mu(g)\psi(v, v')\text{ for some } \mu(g) \in R^\times \}$ for all (small) $\Q$-algebras $R$. We get a scalar character $\mu \colon CSp_V \rightarrow \mathbb{G}_{m, \Q}$ and a short exact sequence $1 \rightarrow Sp_V \rightarrow CSp_V \overset{\mu}{\rightarrow}\mathbb{G}_{m, \Q}\rightarrow 1$, where $Sp_V \cong Sp_{2n}$. 

Let $\mathcal{H}_V$ be the set of morphisms $h\colon \DS \rightarrow GL_{V, \R}$ with the following properties.
\begin{enumerate}
\item $h$ induces a Hodge structure of type $\{(-1,0), (0,-1)\}$ on $V_\C$.
\item $\psi_\R (h(i)u, h(i)v) =\psi_\R(u,v)$ for all $u, v\in V_\R$. 

(This is equivalent to $\psi_\R (h(z)u, h(z)v) =z\overbar{z}\,\psi_\R(u,v)$. Therefore, $\mu \circ h=N$ for all such $h$.)
\item The (symmetric) bilinear form $(u,v) \mapsto \psi_\R(u, h(i)v)$ is either positive or negative definite.
\end{enumerate} 
By construction, every $h\in \mathcal{H}_V$ factors through $CSp_V$. It is well known that $CSp_V$ is a connected reductive algebraic group and $CSp_V(\R)$ acts transitively on $\mathcal{H}_V$. ($Sp_V(\R)$ is connected and acts transitively on either of the two connected components of $\mathcal{H}_V$.) Then $(CSp_V, \mathcal{H}_V)$ is a pure Shimura datum. If $V=\Q^{2n}$ with $\psi $ given by the matrix $\begin{bmatrix}0 & I_n\\ -I_n & 0\end{bmatrix}$, we will write $(CSp_{2n}, \mathcal{H}_{2n})$. 
\end{subsecexample}

\begin{subsecexample}
There is a morphism $(CSp_V, \mathcal{H}_V) \rightarrow (\mathbb{G}_{m, \Q}, \mathcal{H}_0)$. For any $h\in \mathcal{H}_V$, there is a unique $\lambda \in \mathcal{H}_0$ such that $(u, v) \mapsto \lambda \circ \psi_\R(u, h(i)v)$ is a polarization of the Hodge structure on $V$ defined by $h$. (If $\psi_\R(-, h(i)\cdot -)$ is positive  definite, take $\lambda \colon 1 \mapsto 2\pi i $; if it is negative definite, take $\lambda \colon 1\mapsto -2\pi i$.) This way we get a diagram
\begin{equation*}
\begin{tikzcd}
\mathcal{H}_V \arrow[r] \arrow[d,"\wr", "\mathrm{id}" ' ]  &  \mathcal{H}_0 \arrow[d]\\
\mathcal{H}_V \arrow[r]                      &  \{N\}
\end{tikzcd}
\end{equation*}
compatible with the action of $CSp_V(\R) \rightarrow  \mathbb{G}_{m,\R}(\R)$. 
\end{subsecexample}

We summarize some consequences of the axioms.

\begin{subsecproposition}[{\cite[2.4, 2.5, 2.9, 2.12]{Pink}}]\label{subsecproposition: mixed shimura data, properties}~

\begin{enumerate}
\item Every connected component $\mathcal{X}^0 \subset \mathcal{X}$ is mapped isomorphically onto a connected component $\mathcal{H}^0=h(\mathcal{X}^0)\subset \mathcal{H}$.

\item For any morphism $(\varphi, \xi)$ of mixed Shimura data, $\xi$ is a holomorphic map. If both $\varphi$ and $\xi$ are injective (whence we call $(\varphi, \xi)$ an embedding), then $\xi$ is a closed embedding.

\item We can define the product $(P, \mathcal{X}) \times (P', \mathcal{X}')=(P \times P', \mathcal{X} \times \mathcal{X}')$ with $h\times h' \colon \mathcal{X} \times \mathcal{X}' \rightarrow \mathcal{H} \times \mathcal{H}' $ having the desired universal property and projection maps.

\item Let $N\trianglelefteq P$ be a normal $\Q$-subgroup. There exists a quotient mixed Shimura datum $(P, \mathcal{X})/N=(P/N, \mathcal{X}/N)$ with a projection morphism $(\pi, \zeta) \colon (P, \mathcal{X}) \longrightarrow (P, \mathcal{X})/N$. It is defined up to isomorphism with the universal property that, if in a given morphism $(\varphi, \xi) \colon (P, \mathcal{X}) \rightarrow (P', \mathcal{X}')$, $\varphi$ factors through $P/N$, then $(\varphi, \xi)$ factors uniquely through $(P, \mathcal{X})/N$. In other words, we have the diagram below.
\begin{equation*}
\begin{tikzcd}[column sep=small]
                                  &(P,\mathcal{X}) \arrow[dr, "{(\varphi,\xi)}"] \arrow[dl, "{(\pi, \zeta)}"'] & \\
(P, \mathcal{X})/N    \arrow[rr, "\exists!"]           &        & (P', \mathcal{X}')
\end{tikzcd}
\end{equation*}

Let $P_1=P/N $ and $U_1$ be the image of $U$ under $\pi \colon P \rightarrow P_1$. Fix one $x\in\mathcal{X}$. Then $\dps \mathcal{X}/N$ is defined as $\dps \frac{ P_1(\R) \cdot U_1(\C) }{\pi(\Stab_{P(\R)\cdot U(\C)}(x))}$, where $\mathcal{H}/N$ is the $P_1(\R)\cdot U_1(\C)$-conjugacy class of $\pi \circ  h_x$.

\end{enumerate}
\end{subsecproposition}
 
A mixed Shimura datum $(P, \mathcal{X})$ is called \emph{irreducible} if $P$ is the smallest $\Q$-normal subgroup containing the image of all $h_x$. (Again, it suffices to check this for a single $h_x$.) Given an arbitrary $(P, \mathcal{X})$, Let $P_1$ be the smallest $\Q$-normal subgroup of $P$ through which some $h_x$ factors ($P_1$ thus must be connected!). Since $h_x(\bG_{m,\C})$ acts trivially on $P/P_1$, we must have $W\subset P_1$ by (A5). Then one checks that (see {\cite[2.13]{Pink}}) that the $P_1(\R) \cdot U(\C)$-orbit of $x\in\mathcal{X}$ is a union of connected components of $\mathcal{X}$. (We just have to show that $P_1(\R)\cdot U(\C)\cdot \Stab_{P(\R)\cdot U(\C)}(x)$ is a finite-index subgroup of $P(\R)\cdot U(\C)$.) Denote by $\mathcal{X}_1$ this orbit. Then $(P_1, \mathcal{X}_1) \rightarrow (P, \mathcal{X})$ is an embedding. Dividing $\mathcal{X}$ into a disjoint union of $P_1(\R) \cdot U(\C)$-orbits, we see that every mixed Shimura datum $(P, \mathcal{X})$ is canonically a disjoint union of irreducible ones.

Irreducible mixed Shimura data have the following nice property.

\begin{subsecproposition}[{\cite[2.14]{Pink}}]\label{subsecproposition: irreducible data}
Let  $(P, \mathcal{X})$ be an irreducible mixed Shimura datum. Then the representation $P \rightarrow GL_{U}$ factors through the scalar cocharacter $\mathbb{G}_{m, \Q} \rightarrowtail GL_U$. That is, $P$ acts on $U$ through scalars.
\end{subsecproposition}

We can also extract some useful information about the structure of the unipotent radical $W$. As $\Lie U$ and $\Lie W$ are of weight $-2$ and $<0$, respectively, $[\Lie U, \Lie W]=0$. Similarly, the derived group $W^{\mathrm{der}}\subset U$. It follows that $U$ is central in $W$ and $V$ is abelian. Moreover, we have a well-defined commutator pairing
\begin{IEEEeqnarray*}{rrCl}
\Psi \colon &V \times V &\longrightarrow& U\\
&(v, v')&\longmapsto& [v, v'].
\end{IEEEeqnarray*}
Here (in terms of $\Q$-points) $[v, v']=[w, w']$ for any $w, w'\in W(\Q)$ lifting $v, v'\in V(\Q)$. This pairing determines the group structure on W as follows: Since we are over characteristic $0$, $W$ is isomorphic to $U\times V$ as a scheme, and $(u,v) \cdot (u', v') \coloneq (u+u'+\frac{1}{2}\Psi(v, v'), v+v')$ defines the group structure on $U\times V$. Conversely, any such $\Psi$ determines a unipotent extension $1 \rightarrow U \rightarrow W \rightarrow V \rightarrow 1$ with $\Psi$ as its commutator pairing.\\

The last construction to be introduced is the unipotent extension. Let $(P, \mathcal{X})$ be a mixed Shimura datum. Suppose we have a short exact sequence $1\rightarrow W_0 \rightarrow P_1 \rightarrow P \rightarrow 1$ with $W_0$ a unipotent group. We seek to extend $(P, \mathcal{X})$ to $(P_1, \mathcal{X}_1)$ such that $(P_1, \mathcal{X}_1)/W_0 \rightarrow (P, \mathcal{X})$ is an isomorphism. 

Let $V_0$ be a subquotient of $W_0$ such that $\Lie V_0$ is irreducible as a $P_1$-representation. Then $V_0$ must be abelian as $\Lie \left ( V_0^{\mathrm{der}}\right )$ is a $P_1$-invariant subspace (not equal to $\Lie V_0$ as $V_0$ is unipotent, hence solvable). Therefore, the action of $P_1$ factors through $P$ and then further through $G$, as $W\subset P$ must act trivially on $V_0$. With this observation, we state 

\begin{subsecproposition}[{\cite[2.17]{Pink}}]\label{subsecproposition: mixed shimura data, unipotent extension}
 Let $(P, \mathcal{X})$, $P_1$, $W_0$ be as above. Suppose that for every irreducible subquotient $\Lie V_0$ of $\Lie W_0$,
\begin{enumerate}
\item  The Hodge structure on $\Lie V_0$ as a representation of $G$ induced by $\pi\circ h_x$ is of type $\{(-1,-1)\}$ or $\{(-1,0), (0,-1)\}$.
\item The connected center of $G$ acts on $V_0$  through a torus that is an almost direct product of a $\Q$-split torus with a $\Q$-torus of compact type.
\end{enumerate}
Then 

\begin{enumerate}
\item the projection map $P_1\rightarrow P$ extends to a morphism of mixed Shimura data $(P_1, \mathcal{X}_1) \rightarrow (P,\mathcal{X})$ such that $(P_1, \mathcal{X}_1)/W_0\rightarrow(P, \mathcal{X})$ is an isomorphism. The $(P_1, \mathcal{X}_1)$ is defined up to canonical isomorphism.

\item Given a morphism of mixed Shimura data $(\varphi, \xi) \colon (P', \mathcal{X}') \rightarrow (P, \mathcal{X})$, if $\varphi \colon P' \rightarrow P $ factors through $P_1$, then $(\varphi, \xi)$ factors uniquely through $(P_1, \mathcal{X}_1)$.
\end{enumerate}
Furthermore, for any unipotent extension $(P_1, \mathcal{X}_1) \rightarrow (P,\mathcal{X})$, the diagram 
\begin{equation*}
\begin{tikzcd}
\mathcal{X}_1 \arrow[r]\arrow[d,"h_1"']    &    \mathcal{X}\arrow[d, "h"]\\
\mathcal{H}_1 \arrow[r]                                    &     \mathcal{H}
\end{tikzcd}
\end{equation*}
is a pullback.
\end{subsecproposition}  

Now let $(\varphi, \xi) \colon (P_1, \mathcal{X}_1) \rightarrow (P,\mathcal{X})$ be a unipotent extension for $1\rightarrow W_0 \rightarrow P_1 \rightarrow P \rightarrow 1$. By (A5), it is not hard to deduce that the fiber $\xi^{-1}(x)$ over a point $x\in \mathcal{X}$ is an orbit under $W_0(\R)\cdot U_0(\C)$ where $U_0 \coloneq W_0\cap U_1\subset P_1$. Choose any $x_1\in \xi^{-1}(x)$. Then 
\[
\dps \Lie \left(W_0(\R)\cdot U_0(\C) \right)=(\Lie W_0)_\R+ (\Lie U_0)_\C \overset{\sim}{\rightarrow}\frac{ (\Lie W_0)_\C}{F_{h_{x_1}}^0(\Lie W_0)_\C} 
\] 
with $\{F_{h_{x_1}}^{p}(\Lie W_0)_\C \}_{p}$ the Hodge filtration on $\Lie W_0$ induced by $h_{x_1}$. When the short exact sequence splits, it can be shown that $\mathcal{X}_1 \rightarrow \mathcal{X}$ is isomorphic to a holomorphic complex vector bundle over $\mathcal{X}$.

Since $1 \rightarrow W \rightarrow P \rightarrow G\rightarrow 1$ splits, we obtain a structural result on $\mathcal{X}$.
\begin{subsecproposition}[{\cite[2.19]{Pink}}]
For a mixed Shimura datum $(P, \mathcal{X})$, any connected component $\mathcal{X}^0$ of $\mathcal{X}$ is a holomorphic vector bundle over a Hermitian symmetric domain $\mathcal{H}^0\subset \mathcal{H}$.
\end{subsecproposition}

Finally, we list a few examples of unipotent extension which will be useful.
\begin{subsecexample} \label{subsecexample: mixed shimura data, simplest example}
Let $\mathbb{G}_{m,\Q}$  act on $U_0=\Q$ through multiplication. Then under the norm morphism $N \colon \DS \rightarrow \mathbb{G}_{m, \Q}$, the induced Hodge structure on $\Lie U_0$ is of type $\{(-1,-1)\}$. Let $P_0\coloneq U_0 \rtimes \bG_{m,\Q}$ and let $(P_0, \mathcal{X}_0) \rightarrow (\mathbb{G}_{m, \Q}, \mathcal{H}_0)$ be the unipotent extension corresponding to $1 \rightarrow U_0 \rightarrow U_0 \rtimes \mathbb{G}_{m, \Q}  \rightarrow \mathbb{G}_{m,\Q}\rightarrow 1$.

The Hodge structure on $\Lie P_0$ is of type $\{(0,0), (-1,-1)\}$. $(P_0, \mathcal{X}_0)$ may be the simplest ``genuine" mixed Shimura datum.
\end{subsecexample}

\begin{subsecexample} \label{subsecexample: mixed shimura data, extension of symplectic group}
Let $V=\Q^{2n}$ and $\psi \colon V_{2n} \times V_{2n} \rightarrow U_{2n}=\Q$ the standard symplectic form on $V_{2n}$. Let $(CSp_{2n}, \mathcal{H}_{2n})$ be defined as in Example \ref{subsecexample: mixed shimura data, siegel modular variety} and $1 \rightarrow U_{2n} \rightarrow W_{2n} \rightarrow V_{2n}\rightarrow 1$ the central extension with $\psi$ as the commutator pairing. Let $CSp_{2n}$ act on $U_0$ through the multiplier map $CSp_{2n} \rightarrow \mathbb{G}_{m, \Q}$. Then $\psi$ becomes $CSp_{2n}$-equivariant. Through unipotent extensions we obtain the following morphisms of mixed Shimura data:
\begin{IEEEeqnarray*}{c}
(W\rtimes CSp_{2n}, \mathcal{X}_{2n}) \rightarrow (V\rtimes CSp_{2n}, \mathcal{Y}_{2n}) \rightarrow (CSp_{2n}, \mathcal{H}_{2n}).
\end{IEEEeqnarray*}
We will denote the first mixed Shimura datum as $(P_{2n}, \mathcal{X}_{2n})$.
\end{subsecexample}

          \subsection{Mixed Shimura varieties} \label{subsec: mixed shimura varieties}

We now turn to our main objects of study. Mixed Shimura varieties have a definition akin to their pure counterparts and possess morphisms and functoriality formulated in a similar way. For the purpose of constructing their toroidal compactifications, however, it is necessary to examine more closely the fibers of a unipotent extension. Since they are more general objects, the notation will become more complicated as we move on. But most of them are introduced in order to avoid further confusion. This subsection runs parallel to Chapter 3 of \cite{Pink}.

\begin{subsecdefinition}[mixed Shimura varieties]\label{subsecdefinition: mixed shimura varieties, mixed shimura varieties}
Let $(P, \mathcal{X})$ be a mixed Shimura datum and $K_f\subset P(\fadele)$ an open compact subgroup. For the product space $\mathcal{X}\times P(\fadele)$, we let $P(\Q)$ act diagonally from the left and let $K_f$ act from the right on $P(\fadele)$. The associated \emph{mixed Shimura variety} is defined to be the double coset space 
\[
M^{K_f}(P, \mathcal{X})(\C)=P(\Q) \backslash \mathcal{X} \times P(\fadele) / K_f.
\] 
\end{subsecdefinition}

We equip $M^{K_f}(P, \mathcal{X})(\C)$ with the quotient topology of $\mathcal{X} \times P(\fadele) / K_f$ (with $P(\fadele) / K_f$ a discrete space). The set of connected components of $M^{K_f}(P, \mathcal{X})(\C)$ is finite, since it is mapped onto by the finite set $\pi_0(P(\Q) \backslash P(\R)\cdot U(\C) \times P(\fadele) / K_f) \overset{\sim}{\leftarrow}\pi_0(P(\Q) \backslash P(\R) \times P(\fadele) / K_f)=\pi_0(P(\Q)\backslash P(\mathbb{A}) / K_f)$ (cf. {\cite[3.2]{Pink}}). This set also maps onto $\pi_0(P(\Q) \backslash P(\fadele)/K_f)=P(\Q) \backslash P(\fadele)/K_f$. Therefore, $P(\Q) \backslash P(\fadele)/K_f$ is also finite. 

Since $\mathcal{X}$ has only finitely many connected components, for any connected component $\mathcal{X}^0$, the subgroup $\Stab_{P(\Q)}(\mathcal{X}^0) \subset P(\Q)$ is of finite index. Hence $\Stab_{P(\Q)}(\mathcal{X}^0) \backslash P(\fadele)/K_f$ is a finite set. Choose a set of representatives $\{p_f\}\subset P(\fadele)$ for it. It follows that the connected components of $M^{K_f}(P, \mathcal{X})(\C)$ are precisely $\Gamma_{p_f}\backslash \mathcal{X}^0$ where $\Gamma_{p_f}=\Stab_{P(\Q)}(\mathcal{X}^0)\cap p_fK_fp_f^{-1}$. So we have a decomposition
\begin{IEEEeqnarray*}{rCl}
\bigsqcup_{p_f}\Gamma_{p_f}\backslash \mathcal{X}^0&\overset{\sim}{\longrightarrow}& M^{K_f}(P, \mathcal{X})(\C).
\end{IEEEeqnarray*}

The following proposition tells us that $M^{K_f}(P, \mathcal{X})(\C)$ is relatively well behaved as a topological space.

\begin{subsecproposition}[{\cite[3.3]{Pink}}]\label{subsecproposition: mixed shimura varieties, topology}
Fix $\mathcal{X}^0\subset \mathcal{X}$, $p_f \in P(\fadele)$, and let $\Gamma_{p_f}$ be defined as above.

\begin{enumerate}
\item $\Gamma_{p_f}$ acts properly discontinuously on $\mathcal{X}^0$. 

\item $M^{K_f}(P, \mathcal{X})(\C)$ inherits the structure of a normal complex space from the $P(\R)\cdot U(\C)$-invariant complex structure on $\mathcal{X}$. Its singularities are at most quotient singularities by finite groups.

\item If $K_f\subset P(\fadele)$ is a neat open compact subgroup, then $M^{K_f}(P, \mathcal{X})(\C)$ is a complex manifold.
\end{enumerate}
\end{subsecproposition}

\begin{subsecremark}~

\begin{enumerate}[label=(\roman*)]
\item $\Gamma_{p_f}$ is generally an infinite group. To say it acts properly discontinuously on some space $Y$ means that some finite quotient $\Gamma_{p_f}/\Gamma_0$ acts on $Y$ properly discontinuously in the usual sense.

\item For the generalized notion of neat open compact subgroups $K_f\subset P(\fadele)$, see (the last part of) Chapter 0 of \cite{Pink}. Generalized neatness retains the most commonly used  properties. For example, it is inherited by subgroups, images, and preserved under automorphisms of $P$ as well as inner automorphisms of $P(\fadele)$.
\end{enumerate}
\end{subsecremark}

We define some morphisms of mixed Shimura varieties.

Let $(P,\mathcal{X})$ be a mixed Shimura datum and $K_f, K'_f\subset P(\fadele)$ be open compact subgroups. Suppose for $p_f\in P(\fadele)$, we have $p_f^{-1}K_fp_f \subset K'_f$. Then multiplying on the right by $p_f$ induces a well-defined morphism 
\begin{IEEEeqnarray*}{rCrcl}
[\cdot p_f]=[\cdot p_f]_{K'_f, K_f}&\colon& M^{K_f}(P, \mathcal{X})(\C) &~\longrightarrow ~& M^{K'_f}(P, \mathcal{X})(\C)\\
&&  [x, p'_f]&~\longmapsto~&[x, p'_f \cdot p_f]
\end{IEEEeqnarray*}
of mixed Shimura varieties. $[\cdot p_f]$ is holomorphic and surjective with finite fibers (thus open and closed). If $p_f^{-1}K_fp_f=K'_f$ it is an isomorphism. This is the case, for example, if $K'_f=K_f$ and $p_f$ normalizes $K_f$.

For another type of morphism, let $ \varphi \colon (P, \mathcal{X}) \rightarrow (P', \mathcal{X}')$ be a morphism of mixed Shimura data (where we abuse the notation by writing $(\varphi, \xi)$ as $\varphi$).  If $K_f \subset P(\fadele)$, $K'_f \subset P'(\fadele)$ are open compact subgroups such that $\varphi(K_f)\subset K'_f$, then $\varphi$ induces a well-defined morphism of mixed Shimura varieties
\begin{IEEEeqnarray*}{rCrcl}
[\varphi]=[\varphi]_{K'_f, K_f}&\colon& M^{K_f}(P, \mathcal{X})(\C) &~\longrightarrow~& M^{K'_f}(P', \mathcal{X}')(\C)\\
&&  [x, p_f]&~\longmapsto~&[\varphi(x) , \varphi(p_f)].
\end{IEEEeqnarray*}
This morphism is again holomorphic. If $\varphi \colon P \rightarrow P'$ is an isomorphism of groups and $\varphi(K_f)=K'_f$, $[\varphi]$ is an isomorphism.

\begin{subsecremark}~

\begin{enumerate}[label=(\roman*)]
\item For any $k_f \in K_f$, $[\cdot k_f]$ is the identity morphism $[\id]\colon  M^{K_f}(P, \mathcal{X})(\C) \rightarrow M^{K_f}(P, \mathcal{X})(\C)$.

\item If $K'_f$ is a normal open compact subgroup of $K_f$, then $K_f/K'_f$ is finite and acts on $ M^{K'_f}(P, \mathcal{X})(\C)$. From this perspective, $ M^{K_f}(P, \mathcal{X})(\C)$ becomes the quotient space  under this finite group action. (However, as we shall see, this is in general not a faithful action.)
\end{enumerate}
\end{subsecremark}

For more properties of mixed Shimura varieties, we refer the reader to \cite[3.7-3.12]{Pink}. It ought to be mentioned that fiber products do not exist in general in the category of mixed Shimura varieties. They do exist for the pullback of a unipotent extension $(P_1, \mathcal{X}_1)\rightarrow (P, \mathcal{X})$ along an arbitrary morphism $\varphi \colon (P_2, \mathcal{X}_2) \rightarrow (P, \mathcal{X})$. Even this holds only under a mild assumption on $\varphi$.

We next prove a technical ``centralizer" lemma that will be useful later on. 
Basically, what it means is that if an element of a neat arithmetic subgroup of $P(\Q)$ commutes with the image of $h_x$ for some $x\in \mathcal{X}$, then it lies in the center of $P$ and thus commutes with all of $h(\mX)$. 

\begin{subseclemma}
\label{subseclemma: mixed shimura varieties, centralizer lemma}
Let $(P, \mathcal{X})$ be a mixed Shimura datum and $x \in \mathcal{X}$. Then for any neat arithmetic subgroup $K \subset P(\Q)$,  we have $\dps K\cap\,\Cent_{P(\Q)}(h_x) \subset Z(P)(\Q)$.
\end{subseclemma}
\begin{proof}
Let $K_0\coloneq K\cap \Cent_{P(\Q)}(h_x)$. It suffices to show that the image of $K_0$ in $\left (P/Z(P)\right )(\Q)$ is trivial. Thus we may assume that the connected center $Z(G)^0$ of $G$ is an almost direct product of a $\Q$-split torus with a $\Q$-torus of compact type. Since $P/P^{\mathrm{der}}\overset{\vspace{-2pt}{\sim}}{\rightarrow} G/G^{\mathrm{der}}$ and $Z(G)^0$ projects onto this, $P/P^{\mathrm{der}}$ is also an almost direct product of a $\Q$-split torus with a $\Q$-torus of compact type. Any neat arithmetic subgroup in such a group is trivial. Therefore, 
\begin{IEEEeqnarray*}{rCl}
K_0& \subset&  \Cent_{P^{\mathrm{der}}(\Q) }(h_x) \\
       & \subset &  \Cent_{P^{\mathrm{der}}(\R)\cdot U(\C)}(h_x)\\
        &\overset{\pi}{\cong} &  \Cent_{G^{\mathrm{der}}(\R)}(\pi \circ h_x)  \text{ (cf. \cite[1.17]{Pink})}\\
&\subset&\Cent_{G^{\mathrm{der}}(\R)}(\pi \circ h_x(i)) \\
  &\subset& \left \{  g\in G^{\mathrm{der}}(\C)  ~\lvert ~\mathrm{int}(\pi\circ h_x(i)) (g)=\overbar{g}  \right \}.
\end{IEEEeqnarray*}
But the last set is compact as $\mathrm{int}(\pi\circ h_x(i)) $ induces a Cartan involution on $G^{\mathrm{der}}$. As $K_0$ is discrete and torsion-free, it must be trivial.
\end{proof}

\begin{subseccorollary}
Let $(P, \mathcal{X})$ be a mixed Shimura datum and $x \in \mathcal{X}$. Let $K_f \subset P(\fadele)$ be a neat open compact subgroup. Then $\Stab_{P(\Q)}(x) \cap K_f \subset \Id_{Z(P)(\Q)}(\mathcal{X})$, where $\Id_{Z(P)(\Q)}(\mathcal{X})$ is the subgroup of $Z(P)(\Q)$ that acts trivially on $\mathcal{X}$. 
\end{subseccorollary}
\begin{proof}
We have $\Stab_{P(\Q)}(x) \subset \Cent_{P(\Q)}(h_x)$. If an element $z\in Z(P)(\Q)$ fixes a point $x\in \mathcal{X}$, it fixes the whole of $\mathcal{X}$ as $P(\R)\cdot U(\C)$ acts transitively on this space.
\end{proof}

The next proposition is also used in the study of unipotent fibers.

\begin{subsecproposition}\label{subsecproposition: mixed shimura varieties, surjection of open compact central intersection}

Let $(P, \mathcal{X})$ be a mixed Shimura datum with $G=P/W$.  Denote by $\pi \colon (P, \mathcal{X}) \rightarrow (G, \mathcal{X}_G)=(P, \mathcal{X})/W$  the morphism constructed from the canonical projection.  Suppose $K_f \subset G(\fadele)$ is a neat open compact subgroup. Then $\dps \Id_{Z(G)}(\mathcal{X}_G)\cap K_f \subset \pi ( \Id_{Z(P)(\Q)}(\mathcal{X}) )$. In particular, 
\[
\left (\pi \lvert_{Z(P)} \right )^{-1}\left (   \Id_{Z(G)}(\mathcal{X}_G)\cap K_f     \right ) \overset{\sim}{\rightarrow}  \Id_{Z(G)}(\mathcal{X}_G)\cap K_f.
\]
\end{subsecproposition}

\begin{proof}
Choose a splitting $P= W \rtimes  G$. As $A\coloneq \Id_{Z(G)}(\mathcal{X}_G)\cap K_f$ is a neat arithmetic subgroup of $Z(G)(\Q)$ acting on $U$ and $V$ through an almost direct product of a $\Q$-split torus with a $\Q$-torus of compact type, $A$ must act trivially on $U$ and $V$ and hence also on $W$. This means that, under the inclusion $G \rightarrowtail W \rtimes G$, $A$ is mapped into the center of $W\rtimes G$. But $A$ fixes points in $\mathcal{X}_G$ and hence must fix $\mathcal{X}$ as well. (By Proposition  \ref{subsecproposition: mixed shimura data, properties},  the splitting of groups $G \rightarrowtail W \rtimes G \rightarrow G$ extends to a splitting of mixed Shimura data $ (G, \mathcal{X}_G) \rightarrowtail (P, \mathcal{X}) \rightarrow (G, \mathcal{X}_G)$.) Therefore $A$ maps into $\Id_{Z(P)(\Q)}(\mathcal{X})$.
\end{proof}

\begin{subseccorollary} \label{subseccorollary: mixed shimura varieties, surjection of open compact central intersection}
Let $\varphi \colon (P_1, \mathcal{X}_1) \rightarrow (P, \mathcal{X})$ be a unipotent extension of mixed Shimura data and $K_f \subset P(\fadele)$ any neat open compact subgroup. Then under the map $\varphi \lvert_{Z(P_1)} \colon \Id_{Z(P_1)(\Q)}(\mathcal{X}_1) \rightarrowtail \Id_{Z(P)(\Q)}(\mathcal{X})$, we have 
\[
 \left ( \varphi \lvert_{Z(P_1)} \right )^{-1} \left ( \Id_{Z(P)(\Q)}(\mathcal{X}) \cap K_f \right) \overset{\sim}{\rightarrow} \Id_{Z(P)(\Q)}(\mathcal{X}) \cap K_f. 
\]
\end{subseccorollary}

\begin{proof} 
Let $\pi \colon (P, \mathcal{X}) \rightarrow (G, \mathcal{X}_G)=(P, \mathcal{X})/W$ be as above and $K^G_f \coloneq \pi(K_f)$. Then we have the diagram below (with $\pi_1 =\pi \circ \varphi$).

\begin{equation*}  
\begin{tikzcd}  
\left ( \varphi   \right )^{-1}(\Id_{Z(P)(\Q)}(\mathcal{X}) \cap K_f  ) 
      \arrow[r,"\sim" yshift= {-1pt}]  
      \arrow[d, rightarrowtail ]
& \Id_{Z(P)}(\mathcal{X}) \cap K_f 
     \arrow[d, rightarrowtail] \\
\left ( \pi_1  \right )^{-1}(\Id_{Z(G)(\Q)}(\mathcal{X}_G) \cap K^G_f  ) 
      \arrow[r, "\sim"',"\varphi"]  
      \arrow[dr, "\sim" yshift=-2pt, "\pi_1" ' pos=0.57 ]
& \left ( \pi   \right)^{-1} (\Id_{Z(G)(\Q)}(\mathcal{X}_G) \cap K^G_f) 
       \arrow[d, "\sim"' {yshift=0.5pt, sloped}, "\pi" xshift=-0.5pt]    \\
&\Id_{Z(G)(\Q)}(\mathcal{X}_G) \cap K^G_f
\end{tikzcd}
\end{equation*}
 
Since $\pi_1$ induces an isomorphism with $\varphi$ and $\pi$ is injective, both $\varphi$ and $\pi$ are isomorphisms. The corollary follows immediately as $\dps \left ( \varphi   \right )^{-1}(\Id_{Z(P)(\Q)}(\mathcal{X}) \cap K_f  )$ is contained in $\left ( \pi_1  \right )^{-1}(\Id_{Z(G)(\Q)}(\mathcal{X}_G) \cap K^G_f  )$.
\end{proof}

With the above results, we can analyze the structure of unipotent fibers in more detail. Let $(\varphi, \xi) \colon (P', \mathcal{X}')\rightarrow (P, \mathcal{X})$ be a unipotent extension with kernel $W_0 \subset P'$. Let $K_f' \subset P'(\fadele)$ be an open compact subgroup and define $K_f =\varphi(K'_f)\subset P(\fadele)$ to be its image. Fix $x' \in \mathcal{X}', p'_f\in P'(\fadele)$; let $x$, $p_f$ be their images in $\mathcal{X}$, $P(\fadele)$, respectively.

As $W_0$ is unipotent, we have a short exact sequence $1 \rightarrow W_0(\Q) \rightarrow P'(\Q) \rightarrow P(\Q) \rightarrow 1$. Thus the preimage of the orbit $P(\Q)\cdot (x, p_f)\cdot K_f \subset \mathcal{X} \times P(\fadele)$ is $P'(\Q) \cdot \Bigl ( \xi^{-1}(x) \times \bigl (p'_f\cdot W_0(\fadele)\cdot K'_f\bigr ) \Bigr )$. Then the fiber over the point $[x, p_f]$ of the morphism $[\varphi] \colon M^{K'_f}(P', \mathcal{X}')(\C) \rightarrow M^{K_f}(P, \mathcal{X})(\C)$ is
\begin{IEEEeqnarray*}{rCl}
[\varphi]^{-1}([x,p_f])&=&P'(\Q)\backslash \Bigl ( P'(\Q) \cdot \bigl ( \xi^{-1}(x) \times \frac{p'_f\cdot W_0(\fadele) \cdot K'_f}{K'_f} \bigr) \Bigr )\\
&\overset{\sim}{\leftarrow}&  \Stab_{P'(\Q)}(x) \backslash \bigl (\xi^{-1}(x)  \times \frac{p'_f\cdot W_0(\fadele)\cdot K'_f}{K'_f} \bigr ) \\
&\overset{\sim}{\leftarrow}&   \Stab_{P'(\Q)}(x) \backslash \bigl (\xi^{-1}(x) \times \frac{ W_0(\Q)\cdot p'_f\cdot K'_f}{K'_f} \bigr ) \\
&\overset{\sim}{\leftarrow}&   \Stab_{P'(\Q)}(x)\cap p'_f K'_f p'^{-1}_f  \backslash   \bigl ( \xi^{-1}(x)\times \frac{ p'_f\cdot K'_f}{K'_f} \bigr ) \\
&=& \Stab_{P'(\Q)}(x)\cap p'_f K'_f p'^{-1}_f  \backslash  \xi^{-1}(x).
\end{IEEEeqnarray*}
(The arrows indicate the direction of natural inclusions.)

As $\xi^{-1}(x)$ is isomorphic to $W_0(\R)\cdot U_0(\C)$ (with $U_0= W_0\cap U_1$), we see that the fiber is connected.

From now on assume that $K'_f$ is neat. Then $\dps \Stab_{P'(\Q)}(x)\cap p'_f K'_f p'^{-1}_f $ projects into $\Stab_{P(\Q)}(x)\cap p_f K_f p^{-1}_f $, which is contained in $\Id_{Z(P)(\Q)}(\mathcal{X})\cap p_f K_f p_f^{-1}$ by Lemma \ref{subseclemma: mixed shimura varieties, centralizer lemma}. Therefore, by Corollary \ref{subseccorollary: mixed shimura varieties, surjection of open compact central intersection} above, 
\begin{IEEEeqnarray*}{rClCl}
   && \Stab_{P'(\Q)}(x)\cap p'_f K'_f p'^{-1}_f &\subset &\Id_{Z(P')(\Q)}(\mathcal{X}')\times W_0(\Q)\\
&\implies& \Stab_{P'(\Q)}(x)\cap p'_f K'_f p'^{-1}_f & =& \bigl ( \Id_{Z(P')(\Q)}(\mathcal{X}')\times W_0(\Q) \bigr ) \cap  p'_f K'_f p'^{-1}_f .
\end{IEEEeqnarray*}
Define 
\[ \dps \Gamma' \coloneq \frac{ \left ( \Id_{Z(P')(\Q)}(\mathcal{X}')\times W_0(\Q) \right ) \cap p'_f K'_f p'^{-1}_f }{  \Id_{Z(P')(\Q)}(\mathcal{X}') \cap  p'_f K'_f p'^{-1}_f }.
\] 
This is the projection of $\dps  \bigl ( \Id_{Z(P')(\Q)}(\mathcal{X}')\times W_0(\Q) \bigr )   \cap    p'_f K'_f p'^{-1}_f $ into $W_0(\Q)$ under $Z(P') \times W_0 \rightarrow W_0$. Hence the fiber is isomorphic to $\Gamma' \backslash W_0(\R)\cdot U_0(\C)$. The situation can be illustrated with the diagram below.

\begin{equation*}
\begin{tikzcd}
W_0(\R)\cdot U_0(\C) \arrow[r, "\sim"] \arrow[d, twoheadrightarrow] & \xi^{-1}(x) \arrow[d, twoheadrightarrow]\\
\Gamma' \backslash W_0(\R)\cdot U_0(\C) \arrow[r, "\sim" ]      &\dps \frac {\xi^{-1}(x) }{\left (\Id_{Z(P')(\Q)}(\mathcal{X}')\times W_0(\Q) \right ) \cap  p'_f K'_f p'^{-1}_f } 
\end{tikzcd}
\end{equation*}

The bottom isomorphism is given by $[w] \mapsto   [w\cdot x', p'_f] $. It is in general non-canonical. But it is so when the unipotent extension splits. 

\begin{subsecproposition} \label{subsecproposition: mixed shimura varieties, unipotent fiber}
Let $(P', \mX') \rightarrow (P, \mX)$ be a unipotent extension by $W_0$. Given $[x', p'_f] \in M^{K'_f}(P',\mX')(\C)$ and its image $[x, p_f] \in M^{K_f}(P, \mX)(\C)$, the fiber over $[x, p_f]$ is isomorphic to $\Gamma' \bs W_0(\R) \cdot U_0(\C)$ via the map $[w] \mapsto   [w\cdot x', p'_f] $ as shown above.
\end{subsecproposition}

Finally, there are certain conditions under which $\Gamma'$ has a simpler description. We summarize them below.
\begin{subsecproposition}
In the following two situations, $\dps \Gamma' = W_0(\Q) \cap  p'_f K'_f p'^{-1}_f$.
\begin{enumerate} 
\item $Z(P)^0$ is an almost direct product of a $\Q$-split torus with a $\Q$-torus of compact type.

\item The unipotent extension $1 \rightarrow W_0 \rightarrow P' \rightarrow P \rightarrow 1 $ splits. $W_0$ is abelian and $K'_f = K^{W_0}_f \ltimes K_f$ with $K^{W_0}_f \subset W_0(\fadele)$ any (neat) open compact subgroup. 
\end{enumerate}
\end{subsecproposition}

\begin{subsecremark}
We will mostly apply the above result to the unipotent extension $1 \rightarrow U \rightarrow P \rightarrow P/U \rightarrow 1$. Then $\Gamma \subset U(\Q)$ is an arithmetic lattice and $\Gamma \backslash U(\C)$ is an algebraic torus. Thus $M^{K_f}(P, \mathcal{X})(\C) \rightarrow M^{\pi_U(K_f)}(P/U, \mathcal{X}/U)(\C)$ is locally an analytic $\mathbb{G}_{m, \C}^{\dim U}$-torsor. This will be revisited in the construction of toroidal compactifications.
\end{subsecremark}


          \subsection{Rational boundary components} \label{subsec: rational boundary components}

The rational boundary components of a given mixed Shimura datum $(P, \mathcal{X})$ are mixed Shimura data obtained from certain $\Q$-parabolic subgroups of $P$. When $(P, \mathcal{X})$ is irreducible, it is itself an (improper) rational boundary component. Rational boundary components  play a key role in the construction of toroidal compactifications. In this subsection, we present the process of defining rational boundary components and summarize some of their properties which will be useful later on. Of special importance is the associated open convex cone  $C(\mathcal{X}^0, P_1) \subset U_1(\R)(-1)$.

The first step in defining a rational boundary component is to obtain a new ``more mixed" Hodge structure on $\Lie P$ from the old one defined by some $h_x$. To achieve this, we introduce a few definitions.

Let $H_0\coloneq \DS\times_{\mathbb{G}_{m, \R}}GL_{2, \R}$ be the pullback along the two morphisms $N\colon \DS \rightarrow \mathbb{G}_{m, \R}$ and $\mathrm{det}\colon GL_{2, \R} \rightarrow \mathbb{G}_{m, \R}$. $H_0$ is called the reference group. It fits into the following diagram.

\[
\begin{tikzcd}
1 \arrow[r] &\mathbb{S}^1 \arrow[d, "\id"'] \arrow{r}   & H_0=\mathbb{S}\times_{\mathbb{G}_{m,\R}}GL_{2, \R}  \arrow[d, twoheadrightarrow, "\mathrm{pr}_{\DS}" ']\arrow[r, "\mathrm{pr}_{GL_{2}}~"]& GL_{2,\R}\arrow[d,twoheadrightarrow, "\mathrm{det}"]\arrow[r] &1\\
1 \arrow[r]     &  \mathbb{S}^1 \arrow[r]& \mathbb{S} \arrow[r, twoheadrightarrow, "N"] &\mathbb{G}_{m, \R} \arrow[r] & 1\\
\end{tikzcd}
\]

Let $h_0\colon \DS \rightarrow H_0$ be the morphism with two components $\id\colon \DS \rightarrow \DS$ and $\DS \rightarrow GL_{2,\R}, z=a+bi \mapsto \begin{bmatrix} a  &  -b\\ b &  a\end{bmatrix}$. (Over $\C$, this is $(z_1, z_2) \mapsto \begin{bmatrix} \frac{z_1+z_2}{2}  &  -\frac{z_1-z_2}{2i}\\  \frac{z_1-z_2}{2i} &   \frac{z_1+z_2}{2}\end{bmatrix}$). Also, define $h_\infty \colon \DS_\C \rightarrow H_{0, \C}$ to be the morphism with two components $\id_\C \colon \DS_\C \rightarrow \DS_\C$ and $\DS_\C \rightarrow GL_{2,\C}, (z_1, z_2) \mapsto \begin{bmatrix}z_1z_2 & i(z_1z_2-1)\\0&1\end{bmatrix}$. It is readily checked that these are well-defined morphisms to $H_0$ and $H_{0, \C}$. Note that $h_0$ is a morphism over $\R$ whereas $h_\infty$ is a morphism over $\C$.    

We observe that $\pr_{GL_2} \circ h_0 \colon \DS \rightarrow GL_{2, \R} $ induces on $M_0=\R^2$ a Hodge structure of type $\{(-1, 0), (0, -1)\}$ and $\pr_{GL_2}\circ h_\infty \colon \DS_\C \rightarrow GL_{2, \C} $ induces a Hodge structure of type $\{(-1, -1), (0,0)\}$. The two Hodge structures have the same Hodge filtration. For $M_0$, the only non-trivial step in the Hodge filtration is $F^0M_0=\C \cdot \begin{bmatrix} -i \\1 \end{bmatrix} $. The weight spaces are as follows:
\begin{IEEEeqnarray*}{rCl}
(M_0)^{0, -1}_{\pr_{GL_2} \circ h_0}&=&\C \cdot \begin{bmatrix} -i \\ 1 \end{bmatrix}, \\
(M_0)^{-1, 0}_{\pr_{GL_2} \circ h_0}&=&\C \cdot \begin{bmatrix} i \\ 1 \end{bmatrix} \text{ (the complex conjugate of $(M_0)^{0, -1}_{\pr_{GL_2} \circ h_0}$)},\\
(M_0)^{-1, -1}_{\pr_{GL_2}\circ h_\infty}&=&\C \cdot \begin{bmatrix} 1 \\ 0 \end{bmatrix} \text{ (follows from the definition)},\\
(M_0)^{0, 0}_{\pr_{GL_2} \circ h_\infty}&=&\C \cdot \begin{bmatrix} -i \\ 1 \end{bmatrix}.
\end{IEEEeqnarray*}

We can view the Hodge structure induced by  $\pr_{GL_2}\circ h_\infty$ as a modification of that induced by $\pr_{GL_2}\circ h_0$. If we require the $(-1, -1)$-weight space of $\pr_{GL_2}\circ h_\infty$ to be spanned by $(1,0)$ and that the two Hodge structures have the same Hodge filtration, then $\pr_{GL_2}\circ h_\infty$ is uniquely determined by $h_0$.

Let $B\subset GL_{2}$ be the Borel subgroup of invertible upper-triangular matrices and $B_0 \subset H_0$ be the preimage of  $B$ under $\pr_{GL_2}$. Define $\lambda_0\coloneq (h_0\circ \omega)^{-1}\cdot (h_\infty \circ \omega)$. $\lambda_0$ is a cocharacter $ \mathbb{ G}_{m,\C} \rightarrow H_0^{\mathrm{der}}\cap B_0$. Thus one has
\[
h_\infty \circ \omega = (h_0\circ \omega )\cdot \lambda_0 \colon  z \mapsto \begin{bmatrix}z^{-2}& i(z^{-2}-1)\\0&1\end{bmatrix}=\begin{bmatrix}z^{-1} & 0\\0&z^{-1}\end{bmatrix} \begin{bmatrix}z^{-1}& i(z^{-1}-z)\\0&z\end{bmatrix}.
\]
(Note that our convention on $\omega$ inverts $z$). 

In the next proposition, $N_0=\R^2$ is the representation of $H_0$ via $H_0 \overset{\pr_{\DS}}{\rightarrow }\DS \rightarrow GL_{2, \R}$ where the second arrow induces the standard complex structure on $\R^2=\C$ (i.e.  $ \pr_{GL_2} \circ h_0$). We view $M_0$ as a representation of $H_0$ via $\pr_{GL_2}$.

\begin{subsecproposition}[{\cite[4.4]{Pink}}]\label{subsecproposition: rational boundary components, two hodge structures}

Let $(P,\mathcal{X})$ be a mixed Shimura datum, $x \in \mathcal{X}$ and $\alpha \colon H_{0, \C } \rightarrow P_{\C}$ a morphism such that $\pi \circ \alpha \colon H_{0, \C} \rightarrow G_\C $ is defined over $\R$. Assume $\alpha \circ h_0 =h_x$. Then, as a representation of $H_{0 }$ via  $ \pi  \circ \alpha$,
\begin{enumerate}
\item $(\Lie U)_\R$ is a direct sum of copies of $\wedge^2M_0$. In particular, the decomposition on $\Lie U$ induced by $\pi \circ \alpha \circ h_\infty$ is of type $\{ (-1,-1) \}$.

\item $(\Lie V)_\R$ is a direct sum of copies of $ M_0$ and $N_0$. In particular, the decomposition on $\Lie V$ induced by $\pi \circ \alpha \circ h_\infty$ is of type $ \{ (0,0), (-1, 0), (0, -1),$ $ (-1,-1) \}$.

\item $(\Lie G)_\R$ is a direct sum of copies of $(N_0)^\vee\otimes N_0$, $(N_0)^\vee \otimes M_0$, $(M_0)^\vee \otimes M_0$. In particular, the decomposition on $\Lie G$ induced by $\pi \circ \alpha \circ h_\infty$ is of type $\{  (1,1), (0,1), (1,0), (-1, 1), (0,0), (1, -1), (-1, 0), (0, -1), (-1,-1)\}$.
\end{enumerate}
\end{subsecproposition}

This shows that, under $\omega \circ h_\infty$, the new Hodge structure becomes more mixed.

The second step in the construction of a rational boundary component involves a certain kind of parabolic subgroups of $P$. To streamline the statement of the result, we briefly digress to introduce some general facts on algebraic groups. The reader may refer to \cite{Milne17} or \cite{Humphreys75}.

Let $P$ be a connected algebraic group  over a field $k$. Then $k$-parabolic subgroups of $P$ correspond bijectively to $k$-parabolic subgroups of $G^{\mathrm{ad}}$. By the structure theorem of semisimple algebraic groups, $ G^{\mathrm{ad}}=\prod_{1\leq i \leq r}G_i$ with each $G_i$ a $k$-simple adjoint group. A $k$-parabolic subgroup $Q \subset P$ is called $k$-admissible if it corresponds to $\prod_{1\leq i \leq r}Q_i$, where $Q_i $ is either maximal proper parabolic or equal to $G_i$ itself. 

Consider a (connected) reductive algebraic group $G$ over a field $k$. Let $Q$ be a $k$-parabolic subgroup of $G$ and $W_Q$ its unipotent radical. Choose a maximal $k$-split torus $T$ and a minimal $k$-parabolic $Q'$ such that $T\subset Q' \subset Q$. Then $\Lie Q$ is a direct sum of root spaces of $T$. Consider the system of simple roots relative to $Q'$.  Let $\beta_1, \ldots, \beta_r\in X^*(T)$ (the character group of $T$) be the simple roots that appear in $\Lie W_Q$ and $\beta_1^*, \ldots, \beta_r^*\in Y_*(T)_\Q$ (the cocharacter group of $T$ tensored with $\Q$) the dual of the $\beta_i (1\leq i\leq r)$. Define $\lambda$ to be the unique generator of  $ \left ( \Q\cdot(\beta^*_1+\cdots +\beta^*_r) \right ) \cap Y_*(T)$ that is a $\Q_{<0}$-multiple of $\beta^*_1+\cdots +\beta^*_r$. It is well-known that $\Lie Q$ is the direct sum of root spaces $(\Lie Q)_\beta$ such that $\beta \circ \lambda \leq 0$. Since all such pairs $(Q', T)$ are $Q(k)$-conjugate to each other, the $Q(k)$-conjugacy class of  $\lambda$ does not depend on the choice of $(Q', T)$. Let $\pi_Q \colon Q \rightarrow Q/W_Q$ be the canonical projection, then $\pi_Q \circ \lambda$ maps into the center of $Q/W_Q$. Thus this cocharacter is uniquely determined by $Q$.

Moreover, if $E\supset k$ is an extension field, then it can be checked that $\lambda_E$ is such a cocharacter constructed for $Q_E$ relative to any $(S, P)$ where $S$ is a maximal $E$-split torus and $P$ a minimal $E$-parabolic with $T_E\subset S\subset P \subset Q'_E\subset Q_E$. Therefore, the field of definition of the $Q(E)$-conjugacy class of any such $\lambda$ is $k$. In particular, if $G$ is defined over $\Q$, then the $Q(\C)$-conjugacy class of any such $\lambda$ defined via a  maximal $\C$-torus and a minimal $\C$-parabolic contained in $Q_\C$ has an element defined over $\Q$. And $\pi_Q \circ \lambda$ is defined over $\Q$.

We can now state the result relating $\Q$-admissible parabolics to $H_0$.

\begin{subsecproposition}[{\cite[4.6]{Pink}}] \label{subsecproposition: rational boundary components, morphism of reference group}

Let $(P, \mathcal{X})$ be a mixed Shimura datum and $Q\subset P$ a $\Q$-parabolic subgroup of $P$. Let $\pi_U \colon P\rightarrow P/U$ be the projection. The following are equivalent.
\begin{enumerate}
\item $Q$ is $\Q$-admissible.

\item There exists an $x \in \mX$ and a homomorphism $\alpha_x \colon H_{0, \C}\rightarrow P_\C$ satisfying the following conditions. 
           \begin{enumerate}[label=(\roman*)]
                 \item $\pi_U \circ \alpha_x \colon H_{0,\C} \rightarrow (P/U)_\C$ is defined over $\R$.
                 \item $h_x= \alpha_x \circ h_0$.
 
                \item $ \pi \circ \alpha_x \circ h_\infty \circ \omega = \mu \cdot \lambda$, where $\mu= \pi \circ h_x \circ \omega$ maps into $Z(G)_\C$ and $\lambda=\pi \circ \alpha_x \circ \lambda_0 $ is a cocharacter defining the parabolic subgroup $(Q/W)_\C$ constructed as above.\ Moreover, $\pi_Q \circ \lambda\colon \mathbb{G}_{m,\C} \rightarrow (Q/W)_\C \rightarrow (Q/W_Q)_\C$ is defined over $\Q$ (and maps into the center of $Q/W_Q$). 
            \end{enumerate}

\item For every $x \in \mathcal{X}$, there exists a unique $\alpha_x \colon H_{0, \C} \rightarrow P_\C$ as above.
\end{enumerate}
\end{subsecproposition}

\begin{subsecremark}~

\begin{enumerate}[label=(\roman*)]
\item Note that $\dps  \pi \circ \alpha_x \circ h_\infty \circ \omega = \pi \circ \alpha_x \circ \left( (h_0\circ \omega) \cdot \lambda_0 \right)= (\pi \circ \alpha_x \circ h_0 \circ \omega) \cdot ( \pi \circ \alpha_x \circ \lambda_0  )= (\pi \circ h_x \circ \omega) \cdot ( \pi \circ \alpha_x \circ \lambda_0  ) $. 

\item Combined with Proposition \ref{subsecproposition: rational boundary components, two hodge structures}, this implies that $\alpha_x \circ h_\infty$ and $\alpha_x \circ h_0 \circ \omega=h_x \circ \omega $ both factor through $Q_\C$. Further, $(\Lie Q)_\C$ is the direct sum of  all non-positive weight spaces under $\alpha_x \circ h_\infty \circ \omega$.
\end{enumerate}
\end{subsecremark}

\begin{subsecexample}\label{subsecexample: rational boundary components, gl2 part 1}
Take $(GL_{2, \Q}, \mathcal{H}_2)$ to be the (pure) Shimura datum defined in Example \ref{subsecexample: mixed shimura data, siegel modular variety}. And let $x\in \mathcal{H}_2$ be the point corresponding to $\pr_{GL_2} \circ h_0$ defined at the beginning of this subsection. Then with respect to the $\Q$-admissible parabolic subgroup $B\subset GL_{2,\Q}$ of upper-triangular matrices, $\alpha_x$ is just $\pr_{GL_2} \colon H_{0, \R} \rightarrow GL_{2, \R}$, which is already defined over $\R$ and consistent with the fact that $GL_{2, \Q}$ is reductive. We leave it to the reader to verify that under $\alpha_x \circ h_\infty \colon (z_1, z_2) \mapsto \begin{bmatrix}z_1z_2 & i(z_1z_2-1)\\0&1\end{bmatrix}$, the Lie algebra of $B  \subset GL_{2,\Q}$ is the direct sum of subspaces of non-positive weights. (Indeed, $\left (\Lie GL_{2,\Q} \right )_\C$ is of type $\{(1, 1), (0, 0), (-1, -1)\}$.) Note that $\alpha_x \circ \lambda_0 \colon t \mapsto \begin{bmatrix} t^{-1} &  i(t^{-1}-t)\\ 0 & t \end{bmatrix}$ maps into $B_\C$ and is defined over $\Q$ modulo its unipotent radical $\left \{ \begin{bmatrix}1 & *\\ 0 & 1 \end{bmatrix} \right \}$. It defines $B_\C$ relative to the maximal $\C$-torus 
\[
\left \{ \left .  \begin{bmatrix} x & i(x-y)\\ 0&y\end{bmatrix}=\begin{bmatrix} 1 & -i\\ 0&1\end{bmatrix} \begin{bmatrix} x& 0\\ 0&y\end{bmatrix} \begin{bmatrix} 1 & -i\\ 0&1\end{bmatrix}^{-1}\right |  x, y\in \C^\times  \right \}.
\]

Further, the smallest $\Q$-normal subgroup $P_1\subset B$ containing the image of $\alpha_x \circ h_\infty$ consists of matrices $\left \{ \begin{bmatrix}* & *\\ 0 & 1 \end{bmatrix} \right \}$. Indeed, this is isomorphic to the group $\mathbb{U}_0 \rtimes \mathbb{G}_{m, \Q}$ as defined in Example \ref{subsecexample: mixed shimura data, simplest example}, where $\mathbb{G}_{m, \Q}$ acts by multiplication  on the first factor.
\end{subsecexample}

Let $(P, \mathcal{X})$ be a mixed Shimura datum and $Q\subset P$ a $\Q$-admissible parabolic subgroup. Each $ h_x $ now corresponds to an $\alpha_x$. If $q \in Q(\R)\cdot U(\C)$, then $\mathrm{int}(q) \circ \alpha_x= \alpha_{q\cdot x}$ by the uniqueness characterization above. Define $P_1 \unlhd Q$ to be the smallest $\Q$-normal subgroup through which $\alpha_x \circ h_\infty$ factors. By conjugacy, it follows that all $\alpha_x \circ h_\infty$ factor through $P_1$. We have seen that the assignment $\mathcal{X} \rightarrow \Hom(\DS_\C, P_{1, \C}), x \mapsto \alpha_x \circ h_\infty$ is $Q(\R)\cdot U(\C)$-equivariant.

\begin{subseclemma}[{\cite[4.7-4.11]{Pink}}] \label{subseclemma: rational boundary components, P1 W1 U1}
Let $W_1$ be the unipotent radical of $P_1$.

\begin{enumerate}
\item Each $\alpha_x \circ h_\infty$ induces a rational mixed Hodge structure on $\Lie Q$ with $ W_1= \exp (W_{-1}(\Lie Q))$. (Here $W_1$ is the unipotent radical of $P_1$.)

\item Define $U_1 \coloneq \exp (W_{-2}(\Lie Q))$. ($U_1$ is normal in $Q$.) The groups $P_1$, $W_1$, $U_1$, and the homomorphisms $\{\alpha_x \circ h_\infty\}_{x\in\mathcal{X}}$ satisfy axioms A(2)-A(8) in Definition \ref{subsecdefinition: mixed shimura data, mixed shimura datum}. This set of data only depends on $(P, \mathcal{X})$ and $Q$.

\item The assignment $\mathcal{X}  \rightarrow \Hom(\DS_\C, P_{1, \C}), \, x \mapsto \alpha_x \circ h_\infty$ is $Q(\R)\cdot U(\C)$-equi\-vari\-ant. In addition, it maps each connected component $\mathcal{X}^0 \subset \mathcal{X}$ to a $P_1(\R)^0\cdot U_1(\C)$-orbit.
\end{enumerate}
\end{subseclemma}

\begin{subsecremark}~
\begin{enumerate}
\item It follows from Lemma \ref{subseclemma: rational boundary components, P1 W1 U1} (1), that $W_Q =W_1\cdot W$, and the decomposition induced on $\Lie W_Q$ by $\alpha_x \circ h_\infty$ is of type $\{(0,0), ({-}1, 0), (0, -1), (-1, -1)\}$.

\item By the proof of {\cite[4.7]{Pink}}, $Q(\R)\cdot U(\C)$ acts transitively on $\mathcal{X}$. Thus, (3) says that the map sends every $Q(\R)^0\cdot U(\C)$-orbit to a $P_1(\R)^0\cdot U_1(\C)$-orbit.

\item We could replace our $h_0\colon \DS \rightarrow H_0$ with any other morphism having the same $\pr_\DS \circ h_0 $ and such that $\pr_{GL_2}\circ h_0$ induces a Hodge structure of type $\{(-1, 0), (0,-1)\}$ on $M_0$. (Recall that $h_\infty$ is uniquely determined by $h_0$.) The resulting $P_1$, $W_1$, and $U_1$ remain the same.

\end{enumerate}
\end{subsecremark}

The map $\mathcal{X} \rightarrow \Hom(\DS_\C, P_{1, \C}), x \mapsto \alpha_x \circ h_\infty$ can map multiple $\mathcal{X}^0$'s to the same $P_1(\R)^0\cdot U_1(\C)$-orbit. We wish to separate their images so that distinct connected components have distinct images. More precisely, we consider the following map (in which we denote the connected component containing $x$ as $[x]$ and let $P_1(\R)\cdot U_1(\C)$ act on $\pi_0(\mathcal{X})$ via $P_1(\R)$).
\begin{IEEEeqnarray*}{cccCl}
\zeta&\colon&\mathcal{X}&\longrightarrow&\pi_0(\mathcal{X}) \times  \Hom(\DS_\C, P_{1,\C})\\
&&x&\longmapsto& ([x],\alpha_x \circ h_\infty)
\end{IEEEeqnarray*}

The image of $\zeta$ is contained in finitely many $P_1(\R)\cdot U_1(\C)$-orbits and $\zeta$ by definition is injective on connected components. (After all, there are finitely many connected components in $\mathcal{X}$.) Let $\mathcal{X}_1$ be one such orbit (containing, say, $\zeta(x)$ for some $x\in \mathcal{X}$) and $\mathcal{X}^+_{(P_1,\mathcal{X}_1)} \coloneq \zeta^{-1}(\mathcal{X}_1) \subset \mathcal{X}$ be its preimage. Then $\mathcal{X}_{(P_1,\mathcal{X}_1)}^+ \rightarrow \mathcal{X}_1$ is bijective on connected components and $ \Stab_{Q(\R)}(\mX_1) \cdot U(\C)$-equivariant. The projection onto the second factor induces a $P_1(\R)\cdot U_1(\C)$-equivariant map $\mathcal{X}_1 \rightarrow P_1(\R)\cdot U_1(\C)\cdot (\alpha_x \circ h_\infty)$ with finite fibers.

\begin{subsecdefinition}[rational boundary components]\label{subsecdefinition: rational boundary components, rational boundary components}

Given the mixed Shimura datum $(P, \mathcal{X})$ and the $\Q$-admissible parabolic subgroup $Q$, a pair $(P_1, \mathcal{X}_1)$ thus constructed is called a \emph{rational boundary component} of $(P, \mathcal{X})$.
\end{subsecdefinition}
 
With fixed $(P, \mathcal{X})$, $Q$ and $P_1$, choosing an $\mathcal{X}_1$ amounts to choosing a $P_1(\R)$-orbit in $\pi_0(\mathcal{X})$: the $P_1(\R)\cdot U_1(\C)$-orbit generated by the $\zeta$-image of the chosen $P_1(\R)$-orbit forms an $\mathcal{X}_1$. 

\begin{subsecremark}
Note that, by construction, a rational boundary component is always irreducible as a mixed Shimura datum.
\end{subsecremark}

Let us collect some of the properties of $(P_1, \mathcal{X}_1)$.

\begin{subsecproposition}[{\cite[4.12, 4.13, 4.19]{Pink}}]\label{subsecproposition: rational boundary components, rational boundary components} ~

\begin{enumerate}
\item $(P_1, \mathcal{X}_1)$ is an irreducible mixed Shimura datum, equipped with a map $\mathcal{X}^+_{(P_1,\mathcal{X}_1)}  \rightarrow \mathcal{X}_1$ that is $ \Stab_{Q(\R)}(\mX_1) \cdot U(\C)$-equivariant and bijective on connected components. 

In particular, $\mathcal{X}^+_{(P_1,\mathcal{X}_1)} \rightarrow \mathcal{X}_1$ is $P_1(\R)\cdot Q(\R)^0\cdot U(\C)$-equivariant.

\item The map $\mathcal{X}_{(P_1,\mathcal{X}_1)}^+ \rightarrow \mathcal{X}_1$ is injective and holomorphic.

\item Given $x\in \mathcal{X}_{(P_1,\mathcal{X}_1)}^+$, let $x_1 \in \mathcal{X}_1$ be its image. Suppose $\rho \colon P \rightarrow GL_V$ is a $\Q$-representation of $P$. Then the Hodge structures on $V$ induced by  $\rho \circ h_x$ and $\rho \circ h_{x_1}=\rho \circ \alpha_x \circ h_\infty$  have the same Hodge filtration.

\item Suppose $(P_2, \mathcal{X}_2)$ is a rational boundary component of $(P_1, \mathcal{X}_1)$ associated with a $\Q$-admissible parabolic subgroup $Q_{12}\subset P_1$. Then there exists a $\Q$-admissible parabolic subgroup $Q_2\subset P$ such that $\nlb Q_{12}= P_1 \cap Q_2$ and $(P_2, \mathcal{X}_2)$ is the rational boundary component of $(P, \mathcal{X})$ associated with $Q_2$.\ Further, $\mathcal{X}^+_{(P_2,\mathcal{X}_2)}$ is contained in $\mathcal{X}^+_{(P_1,\mathcal{X}_1)}$  and mapped under $\nlb \mathcal{X}^+_{(P_1,\mathcal{X}_1)}  \rightarrow \mathcal{X}_1$ into $(\mathcal{X}_1)^+_{(P_2,\mathcal{X}_2)}$.\ The composite map 
\[ 
\zeta_{\mX_1}^{\mX_2} \,\circ \,\bigl ( \zeta_\mX^{\mX_1}|_{\mX^+_{(P_2,\mX_2)}} \bigr )\colon \mathcal{X}^+_{(P_2,\mathcal{X}_2)} {\rightarrow} (\mathcal{X}_1)^+_{(P_2,\mathcal{X}_2)}  {\rightarrow} \mathcal{X}_2
\] is the map $\zeta_{\mX}^{\mX_2}$ defined for $(P_2, \mathcal{X}_2)$ as a rational boundary component of $(P, \mathcal{X})$.

Thus we have the following diagram.

\begin{equation*}
\begin{tikzcd}
\mathcal{X}^+_{(P_1,\mathcal{X}_1)} \arrow[r, rightarrowtail, "{\zeta_{\mX}^{\mX_1}}"] & \mathcal{X}_1&\\
\mathcal{X}^+_{(P_2, \mathcal{X}_2)} \arrow[u, rightarrowtail] \arrow[r, rightarrowtail]&(\mathcal{X}_1)^+_{(P_2,\mathcal{X}_2)}\arrow[u, rightarrowtail] \arrow[r, rightarrowtail,"{\zeta_{\mX_1}^{\mX_2}}"] &\mathcal{X}_2\\
\mathcal{X}^+_{(P_2, \mathcal{X}_2)} \arrow[u, equal]\arrow[rr,"{\zeta_{\mX}^{\mX_2}}"]&&\mX_2 \arrow[u, equal]
\end{tikzcd}
\end{equation*}

\end{enumerate}

\end{subsecproposition}

Loosely speaking, a rational boundary component of a rational boundary component is a rational boundary component. From now on, we will write $(P,\mX) \geq  (P_1, \mX_1) \geq (P_2, \mX_2)$ to mean that $(P_1, \mX_1)$ is a rational boundary component of $(P, \mX)$ and $(P_2, \mX_2)$ is a rational boundary component of $(P_1, \mX_1)$. We call this a chain of rational boundary components. (Note that $(P, \mX)$ itself may not be a genuine one if it is not irreducible.) By definition, we see that $ U \subset U_1 \subset U_2$.\\

\begin{subsecexample}\label{subsecexample: rational boundary components, gl2 part 2}

Continuing Example \ref{subsecexample: rational boundary components, gl2 part 1}, we have seen that for the rational boundary component $(P_1, \mX_1)$ corresponding to the parabolic subgroup $B$, the group $P_1$ consists of matrices of the form $  \begin{bmatrix}* &*\\ 0&1\end{bmatrix} $. For $\mathcal{H}_2 \ni h_x \colon z=a+bi \mapsto \begin{bmatrix} a&-b\\b & a \end{bmatrix}$, its image under
\begin{IEEEeqnarray*}{cccCl}
\mathcal{H}_2 &\rightarrow& \pi_0(\mathcal{H}_2) \times \Hom(\DS_\C, P_{1, \C})&=&\{\pm 1\} \times  \Hom(\DS_\C, P_{1, \C})
\end{IEEEeqnarray*}
is equal to
\[
\mX_1 \ni x_1=  (+1, h_{x_1}=\alpha_x \circ h_\infty \colon z_1 z_2\mapsto \begin{bmatrix} z_1z_2 & i(z_1z_2-1)\\ 0 &1\end{bmatrix}    ).
\]
Consider the adjoint representation of $GL_2$. The two Hodge structures induced by the two morphisms $h_x$ and $h_{x_1}$ both have the Hodge filtration below.
\begin{IEEEeqnarray*}{rCl}
  F^1 \left(\Lie GL_2 \right)&=&\left < \begin{bmatrix}-i & 1\\ 1 & i \end{bmatrix} \right >\\
F^0\left(\Lie GL_2 \right)&=&\left < \begin{bmatrix}-i & 1\\ 1 & i \end{bmatrix}, \begin{bmatrix}1 & i\\ 0 & 0 \end{bmatrix}, \begin{bmatrix}0 & 1\\ 0 & i \end{bmatrix} \right >\\
F^{-1} \left (\Lie GL_2 \right)&=&\left (\Lie GL_2 \right )_\C
\end{IEEEeqnarray*}
We leave the details to the reader.
\end{subsecexample}

 We now turn to the last significant ingredient. Again let $(P, \mX)$ be a mixed Shimura datum. For some $x_0\in \mX$, the corresponding $h_{x_0}$ is already defined over $\R$. Since $U$ is a commutative unipotent group over $\Q$, we may identify it with a $\Q$-vector space. Define $U(\R)(-1)\subset U(\C)$ to be the subspace $(2\pi i)^{-1}U(\R)$. Thus $U(\C)=U(\R) \times U(\R)(-1)$. It follows that $P(\R)\cdot U(\C)= P(\R)\cdot U(\R)(-1)$. Any other point $x \in \mathcal{X}$ is of the form $up\cdot x_0$ for some $up\in U(\R)(-1)\cdot P(\R)$.  Therefore, $u^{-1}\cdot x=p \cdot x_0$ is defined over $\R$. Since $U(\C)$ acts freely on $\mX$, we have shown that for any point $x \in \mX$, there exists a unique $u_x \in U(\R)(-1)$ such that $h_{u^{-1}_x\cdot x}= u^{-1}_x \cdot h_x$ is defined over $\R$. The map 
\begin{IEEEeqnarray*}{cCcCl}
\mathrm{im}_\mX&\colon & \mX &\longrightarrow & U(\R)(-1)\\
                        &  &  x     &  \longmapsto& u_x
\end{IEEEeqnarray*}
is called the map of imaginary part. Let $P(\R)\cdot U(\C)= U(\R)(-1)\rtimes P(\R)$ act on $U(\R)(-1)=\frac{U(\R)(-1) \rtimes P(\R)}{P(\R)}$ via left multiplication. Thus $P(\R)$ acts by conjugation and $U(\R)(-1)$ acts by translation on itself. Then $\mathrm{im}_\mX$ is $ P(\R)\cdot U(\C)$-equivariant.\ It is also compatible with a morphism of mixed Shimura data $(P, \mX) \rightarrow (P', \mX ')$.

If $(P_1, \mX_1)$ is a rational boundary component of $(P, \mX)$ for the parabolic subgroup $Q$, we see that $\mathrm{im}_{\mX_1}$ is $Q(\R)^0\cdot P_1(\R)\cdot U_1(\C)$-equivariant.

The following proposition will be useful in the proof of our main theorem.

\begin{subsecproposition}[{\cite[4.15]{Pink}}]\label{subsecproposition: rational boundary components, pink 4.15} Let $(P_1, \mX_1)$ be a rational boundary component of $(P, \mX)$ with the map $\mX_{(P_1, \mX_1)}^+ \rightarrow \mX_1$ defined above. Let $\mX^0$ be a connected component of $\mX$ and $\mX_1^0$ the corresponding connected component of $\mX_1$. 

\begin{enumerate}
\item $\dps \mX_{(P_1, \mX_1)}^+ \rightarrow \mX_1$ is an open embedding.

\item There exists an open convex cone $C(\mX^0, P_1) \subset U_1(\R)(-1)$ such that the image of $\mX^0 $ in $\mX_1^0$ is the preimage of $C(\mX^0, P_1) $ under $\mathrm{im}_{\mX_1} |_{\mX_1^0}$. Thus we have a pullback diagram.
\begin{equation*}
\begin{tikzcd}
\mX^0 \arrow[r, rightarrowtail]   \arrow[d, twoheadrightarrow] & \mX^0_1 \arrow[d, "{\mathrm{im}_{\mX_1}|_{\mX_1^0}}", twoheadrightarrow]  \\
C(\mX^0, P_1) \arrow[r, rightarrowtail]        & U_1(\R)(-1)
\end{tikzcd}
\end{equation*}

\item $C(\mX^0, P_1)$ is an orbit under conjugation by $Q(\R)^0$ and translation by $U(\R)(-1)$. It is also invariant under translation by $(U_1\cap W)(\R)(-1)$.

\item Under $U_1(\R)(-1) \rightarrow \left ( U_1/(U_1\cap W) \right) (\R)(-1)$, $C(\mX^0, P_1)$ projects onto a non-degenerate homogeneous self-adjoint open cone (in the sense of \cite{AMRT}, Chapter II.) 
\end{enumerate}
\end{subsecproposition}

\begin{subsecproposition}[{\cite[4.21]{Pink}}]\label{subsecproposition: rational boundary components, properties of the cone} Suppose we have a chain of rational boundary components $(P,\mX) \geq (P_1, \mX_1) \geq (P_2, \mX_2)$.\ Under the map $\mathcal{X}^+_{(P_2, \mathcal{X}_2)}\rightarrow (\mathcal{X}_1)^+_{(P_2, \mathcal{X}_2)} \rightarrow (P_2, \mX_2)$, let $\mX^0 \rightarrowtail \mX_1^0 \rightarrowtail \mX_2^0$ be a chain of corresponding connected components. We have $U \subset U_1 \subset U_2$. Then

\begin{enumerate}
\item $C(\mX_1^0, P_2)=C(\mX^0, P_2)+U_1(\R)(-1)$.

\item $C(\mX^0, P_1)$ is contained in the closure of $C(\mX^0, P_2)$.

\item If $(P,\mX) \geq (P'_1, \mX_1') \geq (P_2, \mX_2)$ is another chain of rational boundary components, then viewing $C(\mX^0, P_1)$ and $C(\mX^0, P_1')$ as subsets of $U_2(\R)(-1)$, we have 
\[
C(\mX^0, P_1) \cap C(\mX^0, P_1') \neq \varnothing \iff (P'_1, \mX_1')=(P_1,\mX_1).
\]
\end{enumerate}
\end{subsecproposition}

\begin{subsecremark}
From (3) above, in particular, if $(P_1, \mX_1) \geq (P_2, \mX_2)$ are distinct rational boundary components of $(P, \mX)$, then $C(\mX^0, P_1) \cap C(\mX^0, P_2)=\varnothing$.
\end{subsecremark}

The disjoint $C(\mX^0, P_1)$ can be combined to define another cone. Let $(P_2, \mX_2)$ be a rational boundary component of $(P, \mX)$. Fix $\mX^0\subset \mX_{(P_2, \mX_2)}^+$. Then for any chain $(P,\mX)\geq (P_1, \mX_1) \geq (P_2, \mX_2)$, $\mX^0 \subset \mX_{(P_1, \mX_1)}^+$ and thus $C(\mX^0, P_1)$ makes sense. Define 
\[
C^*(\mX^0, P_2)\coloneq \bigsqcup_{(P_1, \mX_1)} C(\mX^0, P_1) \subset U_2(\R)(-1),
\]
where $(P_1, \mX_1)$ runs through all rational boundary components between $(P, \mX)$ and $(P_2, \mX_2)$. 

\begin{subsecproposition}[{\cite[4.22]{Pink}}] \label{subsecproposition: rational boundary components, properties of the star cone}
Fix a chain of rational boundary components $(P,\mX)\geq (P_1, \mX_1) \geq (P_2, \mX_2)$ and $\mX^0 \subset \mX_{(P_2, \mX_2)}^+$. Let $\mX^0 \rightarrowtail \mX^0_1 \rightarrowtail \mX^0_2$ be a chain of corresponding connected components.
\begin{enumerate}
\item $C^*(\mX^0, P_2)$ is a convex cone. (In general neither open nor closed.)
\item $C^*(\mX^0_1, P_2)\,{=}\,C^*(\mX^0, P_2)+U_1(\R)(-1)$.
\item  $C^*(\mX^0, P_1)\,{=}\,C^*(\mX^0, P_2) \cap   U_1(\R)(-1) \,{=}\,C^*(\mX^0, P_2)\cap \left(  \left (U_2\cap W_1 \right )(\R)(-1) \right)$.
\end{enumerate}
\end{subsecproposition}

\begin{subsecexample}\label{subsecexample: rational boundary components, gl2 part 3}

Resuming Example \ref{subsecexample: rational boundary components, gl2 part 2}, $\mathcal{H}_2$ has two connected components $\mathcal{H}_2^{\pm}$ by definition, containing morphisms $h_x \colon \DS \rightarrow GL_{2,\R}$ inducing positive definite (resp. negative definite) forms with respect to the alternating form $\begin{bmatrix}0 & 1 \\ -1 & 0\end{bmatrix}$. We now compute the cone 
\[
C(\mathcal{H}_2^{\pm}, P_1) \subset U_1(\R)(-1)=\left \{ \left . \begin{bmatrix}1 & r\cdot (2\pi i)^{-1}\\ 0 & 1\end{bmatrix}  \right | r\in \R \right \}.\] 

We know that for the $h_x$ such that $h_x(i)= \begin{bmatrix}0 & -1\\ 1 & 0 \end{bmatrix}  $, its image in $\mX_1$ is $\bigl (1,  \alpha_x \circ h_\infty \colon (z_1, z_2) \mapsto \begin{bmatrix} z_1 z_2& i (z_1z_2-1)\\ 0 & 1 \end{bmatrix} \bigr)$. However,
\[
 \begin{bmatrix} z_1 z_2& i (z_1z_2-1)\\ 0 & 1 \end{bmatrix}= \begin{bmatrix} 1 & -i  \\ 0 &1 \end{bmatrix} \begin{bmatrix} z_1 z_2 & 0  \\ 0 &1 \end{bmatrix} \begin{bmatrix} 1 & -i  \\ 0 &1 \end{bmatrix}^{-1}.
\]
Thus $h_x \mapsto \alpha_x \circ h_\infty \overset{\mathrm{im}}{ \mapsto}   \begin{bmatrix} 1 & -i  \\ 0 &1 \end{bmatrix}$. 

For the parabolic subgroup $B$, $B(\R)^0=\left \{ \left. \begin{bmatrix} a & b \\ 0 &d \end{bmatrix} \right | a, d >0 \right \}$. By the $B(\R)^0$-equivariance of the maps $\zeta^{\mX_1}_{\mathcal{H}_2} \colon (\mathcal{H}_2)_{(P_1, \mX_1)}^+\rightarrow \mathcal{X}_1 $ and $\mathrm{im}_{\mX_1}\colon \mathcal{X}_1 \twoheadrightarrow U_1(\R)(-1)$, we see that if $q=  \begin{bmatrix} a & b \\ 0 &d \end{bmatrix} $, then the image of $q\cdot h_x$ under $\mathrm{im}_{\mX_1} \circ \zeta^{\mX_1}_{\mathcal{H}_2}$ is just 
\[
q  \begin{bmatrix} 1 & -i  \\ 0 &1 \end{bmatrix}   q^{-1}= \begin{bmatrix} 1 & (ad^{-1}) (-i)  \\ 0 &1 \end{bmatrix}= \begin{bmatrix} 1 & (ad^{-1})2\pi (2\pi i)^{-1}  \\ 0 &1 \end{bmatrix}.
\]
 Hence $C(\mathcal{H}_2^+, P_1)$ is just the ``positive" half-axis, while  $C^*(\mathcal{H}_2^+, P_1)$, being the closure of $C(\mathcal{H}_2^+, P_1)$, is the set of non-negative numbers. 

Similarly, for the $h_{x'}$ such that $h_{x'}(i)= \begin{bmatrix} 0 &  1\\ -1 & 0 \end{bmatrix} $, since $h_{x'}=
\mathrm{int}(b)\circ h_x$ where $b=  \begin{bmatrix}1 &0 \\0&-1 \end{bmatrix}  $, its image in $\mX_1$  is $\left (-1,  \alpha_{x'}\circ h_\infty  \right )$, where the second factor is 
\begin{IEEEeqnarray*}{rCl}
(z_1, z_2) &\mapsto& b  \begin{bmatrix} 1 &- i  \\ 0 &1 \end{bmatrix} \begin{bmatrix} z_1 z_2 & 0  \\ 0 &1 \end{bmatrix} \begin{bmatrix} 1 & - i  \\ 0 &1 \end{bmatrix}^{-1}b^{-1}\\
&=& b  \begin{bmatrix} 1 & -i  \\ 0 &1 \end{bmatrix}b^{-1} b \begin{bmatrix} z_1 z_2 & 0  \\ 0 &1 \end{bmatrix} b^{-1}b \begin{bmatrix} 1 & - i  \\ 0 &1 \end{bmatrix}^{-1}b^{-1}\\
&=&   \begin{bmatrix} 1 & i  \\ 0 &1 \end{bmatrix} \begin{bmatrix} z_1 z_2 & 0  \\ 0 &1 \end{bmatrix} \begin{bmatrix} 1 & - i  \\ 0 &1 \end{bmatrix}^{-1}.
\end{IEEEeqnarray*}  
Hence $\mathrm{im}_{\mX_1}(\alpha_{x'}\circ h_\infty)=  \begin{bmatrix} 1 & i  \\ 0 &1 \end{bmatrix}  $ and the cone 
\[ 
C(\mathcal{H}_2^-, P_1)=\left \{ \left .  \begin{bmatrix} 1 & (ad^{-1}) (i)  \\ 0 &1 \end{bmatrix} =\begin{bmatrix} 1 &- (ad^{-1})2\pi (2\pi i)^{-1}  \\ 0 &1 \end{bmatrix}\right| a, d > 0\right \}.
\] This is the ``negative" half-axis and $C^*(\mathcal{H}_2^-, P_1)$ is the set of non-positive numbers.

From the above, we see that the cone $C(\mX^0, P_1)$ does depend on the connected component $\mX^0$.
\end{subsecexample}

\begin{subsecdefinition}\label{subsecdefinition: rational boundary components, conical complex} The \emph{conical complex} of a mixed Shimura datum $(P, \mX)$ is defined as  
\[
\mathcal{C}(P, \mX) \coloneq \bigsqcup_{\mX^0\subset \mX} 
 \Biggl (  
     \left . 
             \Bigl (  \dps \bigsqcup_{  \mX^0\subset \mX^+_{(P_1,\mX_1)}  } 
               \hspace{-14pt}C^*(\mX^0, P_1) \Bigr ) 
      \right /  \sim 
\Biggr ),
\]
where $\mX^0 \subset \mX$ runs over all connected components, $(P_1, \mX_1)$ runs over all rational boundary components such that $ \mX^0\subset \mX^+_{(P_1,\mX_1)} $, and the equivalence relation is generated by all (closed) embeddings $C^*(\mX^0, P_1) \rightarrowtail C^*(\mX^0, P_2)$ coming from chains $(P,\mX) \geq (P_1, \mX_1) \geq (P_2, \mX_2)$.
\end{subsecdefinition}

Set-theoretically, $\mathcal{C}(P, \mX)$ is the disjoint union of all $C(\mX^0, P_1)$. With the quotient topology, the image of each $C(\mX^0, P_1)$ is a locally closed subset with closure the image of $C^*(\mX^0, P_1)$. The canonical map $C^*(\mX^0, P_1) \rightarrow \mathcal{C}(P, \mX)$ is a closed  homeomorphism onto its image.

The definition of a $K_f$-admissible cone decomposition will make use of the product $\mathcal{C}(P, \mX) \times P(\fadele)$. But before that, we still need some background on torus embeddings. This is the topic of the next section.


          \subsection{Torus embeddings} \label{subsec: torus embeddings}

In this subsection, we briefly introduce the notion of torus embeddings and gather some propositions for later use. For a comprehensive reference, the reader may consult, say, \cite{toricvarieties11}. In \cite{Pink}, the author also studies morphisms of torus embeddings in terms of line bundles and an ampleness criterion. We will mostly focus on the definitions and preliminary results.

Let $V$ be a finite-dimensional $\Q$-vector space and $V^\vee$ its dual space. A subset $\sigma$ of $V_\R$ is called a rational convex polyhedral cone if there exists a finite subset $\{l_i\}_{i \in I} \subset V^\vee$ such that $\sigma = \{v \in V_\R \,|\, \langle l_i, v \rangle \geq 0, \forall i \in I \}$. In other words, $\sigma$ is the intersection of finitely many closed half-spaces defined over $\Q$. Equivalently, there exists a finite subset $\{v_i\}_{i\in I}\subset V$ such that $\sigma=\{v \in V_\R |  v= \sum_{i\in I} x_i v_i, x_i \geq 0\}$. 

Given such a $\sigma$ defined by $\{l_i\}_{i\in I}\subset V^\vee$, a face $\tau $ of $\sigma$ is a subset of the form $\sigma \cap \bigcap_{i\in J} V(l_i)$, where $J \subseteq I$ and $V(l_i)=\{ v \in V_\R| \langle l_i, v \rangle=0\}$. Thus a face of $\sigma$ is the zero locus of finitely many $l_i$ on $\sigma$. A face of $\sigma$ is also a rational convex polyhedral cone (by adding $\{-  l_i\}_{i \in J}$ to the defining set). A proper face is a face not equal to $\sigma$. We write $\tau \leq \sigma$ to mean $\tau$ is a face of $\sigma$. The interior $\sigma^0 $ of $\sigma$ is the complement of the union of all proper faces of $\sigma$. It is open in $\R\cdot \sigma$, but not open in $V_\R$ in general. A cone $\sigma$ is \emph{of top dimension} if $\R\cdot \sigma=V_\R$.

The \emph{dual cone} of $\sigma$ is the cone $\check{\sigma} \coloneq \{ l \in V_\R^\vee \,\lvert \, l|_\sigma\geq 0 \}$. The dual cone of a rational convex polyhedral cone is also a rational convex polyhedral cone and $\check{\check{\sigma}}=\sigma$. Moreover, $\tau \subseteq \sigma \iff \check{\tau} \supseteq \check{\sigma}$.

\begin{subsecproposition}
 $\{l_i\}_{i\in I}$ defines $\sigma$ if and only if $\check{\sigma}=\sum_{i\in I}\R_{\geq 0} \cdot l_i$. 
\end{subsecproposition}
\begin{proof}
We may assume $l_i\in \check{\sigma}$. Let $\sigma '=\sum_{i\in I}\R_{\geq 0} \cdot l_i$. Then $\{l_i\}_{i\in I}\subseteq \sigma ' \subseteq \check{\sigma}$ and $  (\{l_i\}_{i\in I})^\vee= \check{\sigma '} \supseteq \sigma $, where $ (\{l_i\}_{i\in I})^\vee$  denotes the cone defined by $\{l_i\}_{i\in I}$. Thus $\{l_i\}_{i\in I}$ defines $\sigma$  $ \iff (\{l_i\}_{i\in I})^\vee=\sigma \iff \check{\sigma ' } = \sigma \iff \sigma '=\check{\sigma}$.
\end{proof}

\begin{subsecproposition}
The following statements are equivalent for a rational convex polyhedral cone $\sigma$.

\begin{enumerate}
\item $\{ 0\} $ is a face of $\sigma$ (in which case $\sigma$ is said to be \emph{strongly convex}).
\item $\sigma$ does not contain a non-trivial linear subspace of $V_\R$.
\item $\check{\sigma}$ contains a basis of $V^\vee$. 
\end{enumerate}
\end{subsecproposition} 

\begin{subsecdefinition}
Let $\Sigma\subseteq V_\R$ be a collection of rational convex polyhedral cones. $\Sigma$ is called a partial convex polyhedral decomposition if the following three conditions hold.

\begin{enumerate}
\item Every $\sigma \in \Sigma$ is strongly convex, i.e., $\{0\}\leq \sigma$.
\item A face of any $\sigma \in \Sigma$ is again a cone in $\Sigma$.
\item The intersection of two cones in $\Sigma$ is a face of both cones.
\end{enumerate} 
\end{subsecdefinition}

(A collection of cones satisfying only (2) and (3) is also called a fan.) We say $\Sigma$ is finite if it contains finitely many cones. Denote by $|\Sigma|$ the union of all cones in $\Sigma$. If $C\subset V_\R$ contains $|\Sigma|$, $\Sigma$ is called a partial rational polyhedral decomposition of $C$. If $C=|\Sigma|$, then $\Sigma$ is a complete rational polyhedral decomposition of $C$.

Next, let $T$ be a split torus over a field of characteristic zero. We use $X^*(T)$ and $Y_*(T)$ to denote its character and cocharacter group, respectively. Then the affine ring of $T$ is $k[X^*(T)]$. Let $\sigma \subseteq Y_*(T)_\R$  be a rational polyhedral convex cone. The \emph{torus embedding} $T_\sigma$ of $T$ with respect to $\sigma$ is defined to be $\Spec k[X^*(T) \cap {\check{\sigma}}]$. (In the rest of this article, we often use $X^*(T)_{\sigma \geq 0}$ to mean $X^*(T)\cap \check{\sigma}$.) The action of $T$ on itself extends to $T_\sigma$. As $X^*(T)\cap \check{\sigma}\rightarrowtail X^*(T)$, we have a canonical $T$-equivariant map $T \rightarrow T_\sigma$. It is an open embedding if and only if $\sigma$ is strongly convex. In this case, $T_\sigma$ is a normal variety. 

\begin{subsecexample}
Let $T=\mathbb{G}_{m,\C}^d$. Identify $Y_*(T)$ with $\Z^d$ and let $\sigma= \sum_{i=1}^d\R_{\geq 0}e_i$ where $e_i$ is the $i$-th standard basis vector. Then $T \rightarrow T_\sigma$ is just $\mathbb{G}_{m,\C}^d \rightarrow \mathbb{A}^d_{\C}$.
\end{subsecexample}

A $T$-variety over $k$ containing $T$ as an open dense subvariety and for which the inclusion is  $T$-equivariant is called a \emph{toric variety}.
\begin{subsectheorem}\label{subsectheorem: torus embeddings, normality} $T_\sigma$ is a normal affine toric variety. Conversely, any normal affine toric variety is of the form $T_\sigma$ for some strongly convex rational polyhedral cone $\sigma \subseteq Y_*(T)_\R$ of top dimension.
\end{subsectheorem}
\begin{proof} 
Let $T'' \subseteq T$ be the subtorus such that $Y_*(T'')_\Q$ is the maximal $\Q$-linear subspace contained in $\sigma$ and let $T\rightarrow T'=T/T''$ be the quotient map. Then $\sigma$ is the preimage under $Y_*(T)_\R \rightarrow Y_*(T')_\R$ of a strongly convex (rational!) polyhedral cone $\bar{\sigma}$ and $T_\sigma=T'_{\bar{\sigma}}$. The canonical map $T \rightarrow T_\sigma$ is just the composite map $T \rightarrow T' \rightarrow T'_{\bar{\sigma}}$. The rest follows from {\cite[1.3.5]{toricvarieties11}}.
\end{proof}

A rational polyhedral convex cone $\sigma \subseteq Y_*(T)_\R$ is smooth with respect to the lattice $Y_*(T)$ if the monoid $Y_*(T)\cap \sigma $ is generated by a subset of a $\Z$-basis of $Y_*(T)$. The variety $T_\sigma$ is smooth over $k$ if and only if $\sigma$ is smooth.  A strongly convex cone is simplicial if its one-dimensional faces are spanned by linearly independent vectors. If $f\colon T_1 \rightarrow T_2$ is a morphism of tori with convex cones $\sigma_1 \subseteq Y_*(T_1)_\R$, $\sigma_2 \subseteq Y_*(T_2)_\R$, then $f$ extends to a (necessarily $T_{1}$-equivariant) diagram below if and only if $f_*(\sigma_1)\subseteq \sigma_2$.

\begin{equation*}
\begin{tikzcd}
T_1 \arrow[r, "f"] \arrow[d] & T_2 \arrow[d]\\
T_{1,\sigma_1} \arrow[r]      &  T_{2, \sigma_2}
\end{tikzcd}
\end{equation*}

The action of $T$ on $T_\sigma$ has a unique closed orbit $T_\sigma^\sim $ defined by the ideal $I_\sigma=k[X^*(T)_{\sigma ^{0}>0}]$. The affine ring of $T_\sigma^\sim$ is also isomorphic to $k[X^*(T)_{\sigma=0}]$. If $\sigma$ is of top dimension, then $T_\sigma^\sim$ is just a point. We have a canonical surjective projection map $T_\sigma \rightarrow T_\sigma^\sim$ which fits into a diagram.

\begin{equation*}
\begin{tikzcd}
T _\sigma^\sim \arrow[r, rightarrowtail] \arrow[dr,"\sim" '] & T_\sigma  \arrow[d, "\pi_\sigma"]\\
                                                                                                                  &  T_\sigma^\sim
\end{tikzcd}
\end{equation*}
The restriction $\pi_\sigma|_T\colon T\rightarrow T_\sigma^\sim$ makes $T_\sigma^\sim$ a quotient torus of $T$.

If $\tau \leq \sigma$ is a face, then $T_\tau \rightarrowtail T_\sigma$ is an open dense subvariety. The $T$-orbits in $T_\sigma$ are just $\{T_\tau^\sim | \tau \leq \sigma \}$. Each $T_\tau^\sim$ is a locally closed subvariety. The minimal face corresponds to $T$ itself. The closure of $T_\tau^\sim$ is the union of   $T_\rho^\sim$ for all  $\rho$ such that $\tau \leq \rho \leq \sigma$.

Now consider a partial rational polyhedral decomposition $\Sigma $ of $Y_*(T)_\R$. Since $\tau \leq \sigma \implies T_\tau \rightarrowtail T_\sigma$, it is not hard to see that the various $T_\sigma$ for $\sigma \in \Sigma$ can be glued to form a variety $T_\Sigma$, equipped with open dense embeddings $T_\sigma \rightarrowtail T_\Sigma$ compatible with $T_\tau \rightarrowtail T_\sigma$. It follows that $T_\Sigma$ is the colimit of the diagram of morphisms $T_\tau \rightarrow T_\sigma$. Also, $T_\Sigma$ is of finite type over $k$ if and only if $\Sigma$ is finite. It is smooth with $T_\Sigma \backslash T$ a union of smooth divisors with normal crossings, if and only if $\Sigma$ consists of smooth cones; in this case we say $\Sigma$ is smooth. $T_\Sigma$ is proper over $k$ if and only if $\Sigma$ is a finite and complete decomposition of $Y_*(T)_\R$. 

\begin{subsecexample}
Let $T=\mathbb{G}_{m,\C}^2$. Identify $Y_*(T)$ with $\Z \times \Z$ and let $\Sigma$ consist of all the cones spanned by $\pm e_1, \pm e_2$ where $e_i$ is the $i$-th standard basis vector. In other words, the two-dimensional cones are the four quadrants. Then $T \rightarrow T_\Sigma$ is just $\mathbb{G}_{m,\C}\times \mathbb{G}_{m,\C} \rightarrow \mathbb{P}^1_{\C}\times \mathbb{P}^1_{\C}$.
\end{subsecexample}

The action of $T$ extends to $T_\Sigma$. Again, the $T$-orbits are in one-to-one correspondence with the cones in $\Sigma$. Thus a $T$-invariant subset   $S\subseteq T_\Sigma$ corresponds to a subset $\Sigma_S \subseteq \Sigma$. $S$ is open in $T_\Sigma$ if and only if every face of a cone in $\Sigma_S $ is again in $\Sigma_S$. 

Suppose that $f\colon T_1 \rightarrow T_2$ is a morphism of algebraic groups and $\Sigma_1$, $\Sigma_2$ are partial rational polyhedral decompositions of $Y_*(T_1)_\R$, $Y_*(T_2)_\R$, respectively. Then $f$ extends to a (necessarily $T_1$-equivariant) morphism $T_{1, \Sigma_1}\rightarrow T_{2, \Sigma_2}$, if and only if for all $ \sigma_1\in \Sigma_1$, $f_*(\sigma_1) \subseteq \sigma_2$ for some $\sigma_2 \in \Sigma_2$. Therefore we have a similar diagram below. A special case of this is that, if $\Gamma$ is a discrete group with a left action on $T$, then the $\Gamma$-action extends to $T_\Sigma$ if and only if $\Sigma$ is $\Gamma$-invariant, i.e., for all $\gamma \in \Gamma$ and for all $\sigma \in \Sigma$, $\gamma_*(\sigma)\in \Sigma$. 

\begin{equation*}
\begin{tikzcd}
T_1 \arrow[r, "f"] \arrow[d] & T_2 \arrow[d]\\
T_{1,\Sigma_1} \arrow[r]      &  T_{2, \Sigma_2}
\end{tikzcd}
\end{equation*}

In the construction of toroidal compactifications, we will encounter the following situation: Let $f\colon T\rightarrow T'$ be an isogeny of split tori over $k$. Then under the isomorphism $Y_*(T)_\R \rightarrow Y_*(T')_\R$, $\Sigma \subseteq Y_*(T)_\R$  is mapped to $\Sigma ' \subseteq Y_*(T')_\R$. The kernel $\Gamma$ of $f$ acts on $T$, and $T'$ is just the quotient of $T$ under this action. Further, the $\Gamma$-action extends to $T_\Sigma $. Under $T_\Sigma \rightarrow T'_{\Sigma '}$, $T'_{\Sigma '}$ becomes the scheme-theoretic quotient of the $\Gamma$-action.

Torus embeddings have relative versions. Let $X\rightarrow Y$ be a $T$-torsor over $k$ and $\Sigma \subseteq Y_*(T)_\R$ a partial rational polyhedral decomposition. As the torsor is Zariski-locally trivial, it is isomorphic to $T\times V \rightarrow V$ for various open subsets $V  \subseteq Y$. This gives a torus embedding $T \times V \rightarrowtail T_\Sigma \times V$ over $V$. Any two sections of $X$ over $V$ differ by a morphism $V \rightarrow T$. From this, it follows that the various torus embeddings can be glued to a global relative torus embedding $X\rightarrow X_\Sigma$ over $Y$, uniquely defined up to canonical isomorphisms. 

Relative torus embeddings have analogous functorial properties: Suppose $X_i \rightarrow Y_i$ is a $T_i$-torsor and $\Sigma_i\subseteq Y_*(T_i)_\R$ is a rational partial polyhedral decomposition, for $i=1, 2$. Given the left diagram below, and a morphism $f \colon T_1 \rightarrow T_2$ such that $f(\Sigma_1) \subseteq \Sigma_2$, the left diagram extends to the right one $T_1$-equivariantly.
\begin{IEEEeqnarray*}{cCc}
\begin{tikzcd}
X_1 \arrow[r] \arrow[d] & X_2 \arrow[d]\\
Y_1 \arrow[r]                      & Y_2 
\end{tikzcd}& \implies 
& \begin{tikzcd}
X_1\arrow[r] \arrow[d,rightarrowtail] & X_2 \arrow[d, rightarrowtail]\\
X_{1, \Sigma_1} \arrow[r] \arrow[d] & X_{2, \Sigma_2} \arrow[d]\\
Y_1 \arrow[r]                      & Y_2 
\end{tikzcd}
\end{IEEEeqnarray*}

Finally, we introduce the ``$\mathrm{ord}$"-map. Let $T$ be a $\C$-torus. Then the map $\mathrm{ord}_T$ is defined as follows.
\begin{IEEEeqnarray*}{cCrcccc}
\mathrm{ord}_T& \colon & T(\C) &\overset{\sim}{\longleftarrow}& Y_*(T)\otimes_\Z \C^\times &\longrightarrow& Y_*(T)\otimes_\Z \R \\
&&\prod_{j=1}^d\lambda_j(t_j) &\longmapsfrom& \sum_{j=1}^d \lambda_j\otimes t_j &\longmapsto& \sum_{j=1}^d \lambda_j\otimes (\ln |t_j|)
\end{IEEEeqnarray*} 
This map may be viewed as the quotient map of $T(\C)$ by its maximal compact subgroup. More generally, suppose $X\rightarrow S$ is an analytic $T(\C)$-torus over a complex space $S$. Let $K\subset T(\C)$ be its maximal compact subgroup. Then $X/K\rightarrow S$ is a $T/K\overset{\sim}{\rightarrow}Y_*(T)_\R$-torsor. Since we can patch together local sections via a partition of unity, it can be smoothly trivialized. Any trivialization yields a map
\begin{IEEEeqnarray*}{cCc}
 \mathrm{ord}_X& \colon & X  \rightarrow  X/K  \overset{\sim}{\rightarrow} Y_*(T)_\R \times S  \rightarrow  Y_*(T)_\R.
\end{IEEEeqnarray*}
Any two trivializations differ by a smooth map $S \rightarrow Y_*(T)_\R$ which can be lifted to $T(\C)$ (e.g., via $\R \rightarrow \C^\times, t \mapsto e^{t}$). Thus $\mathrm{ord}_X$ is defined up to translation by a smooth map $X \rightarrow S \rightarrow T(\C)$.


          \subsection{Toroidal compactifications} \label{subsec: toroidal compactifications}

We are in a position to construct the toroidal compactifications of mixed Shimura varieties. This process bears some resemblance to the construction of the Baily-Borel compactification of pure Shimura varieties, but is more involved. 

Before we dive into the technical details, let us outline the rough idea behind this project. Let $(P, \mX)$ be a mixed Shimura datum and $K_f \subset P(\fadele)$ an open compact subgroup. We will see that the mixed Shimura variety $M^{K_f}(P, \mX)(\C)$ can be realized as the quotient $\mU/\sim$ of a covering $\mU$ obtained from all rational boundary components $(P_1, \mX_1)$ of $(P, \mX)$. This covering embeds as an open dense subset into a larger space: $\mU \rightarrowtail \overbar{\mU}$. The equivalence relation $\sim$ extends to $\overbar{\mU}$ and the quotient $\overbar{\mU}/ \sim$ is the desired toroidal compactification of $M^{K_f}(P, \mX)(\C)$ with respect to a certain cone decomposition. The following diagram summarizes the situation.
\begin{equation*}
\begin{tikzcd}
\mU \arrow[r, rightarrowtail] \arrow[d, twoheadrightarrow] & \overbar{\mU} \arrow[d,twoheadrightarrow]\\
\mU/\sim \arrow[r, rightarrowtail]                                                 &  \overbar{\mU}/\sim
\end{tikzcd}
\end{equation*}

We begin with the definition of $K_f$-admissible partial cone decompositions $\Sigma$ on $\mC(P, \mX) \times P(\fadele)$ and the relevant decompositions associated with $\Sigma$. Then we construct the compactification in the case that $\Sigma$ is concentrated on the unipotent fiber. Next we obtain the covering $\mathcal{U}$ of the mixed Shimura variety associated with all its rational boundary components and  define the equivalence relation $\sim$ on this covering. It is then embedded into the larger covering $\mUbar$ with the equivalence relation extended. Finally, we briefly discuss the smoothness and completeness properties. This subsection contains materials from Chapters 6 and 9 in \cite{Pink}.

\begin{subsecdefinition}\label{subsecdefinition: toroidal compactifications, admissible cone decomposition}

Let $(P, \mX)$ be a mixed Shimura datum and $K_f \subset P(\fadele)$ an open compact subgroup. Let $\Sigma $ be a collection of subsets of $\mC(P, \mX) \times P(\fadele)$. For a rational boundary component $(P_1, \mX_1)$ and $\mX^0 \subset \mX^+_{(P_1, \mX_1)}$, define $\Sigma(\mX^0, P_1, p_f)\coloneq \{\sigma \in \Sigma \, \lvert \, \sigma \subset -C^*(\mX^0, P_1)\times \{p_f\} \}$. Then $\Sigma$ is called a \emph{$K_f$-admissible partial cone decomposition for $(P, \mX)$} if it satisfies the following conditions.

\begin{enumerate}
\item $\dps \Sigma=\hspace{-5pt}\bigcup_{(\mX^0, P_1, p_f)} \hspace{-10pt}\Sigma(\mX^0, P_1, p_f)$. In other words, any $\sigma$ is contained in some $\Sigma(\mX^0, P_1, p_f)$.

\item  Every $\Sigma(\mX^0, P_1, p_f)$  is a partial rational convex polyhedral cone decomposition of $-C^*(\mX^0, P_1)\times \{p_f\}\subset U_1(\R)(-1)\times \{p_f\}$.

\item $\Sigma$ is invariant under right multiplication by $K_f$ on the second factor.

\item $\Sigma$ is invariant under left multiplication by $P(\Q)$ on both factors.

\item For a fixed $(P_1, \mX_1)$ and $\mX^0 \subset \mX^+_{(P_1, \mX_1)}$, the subcollection $\dps \hspace{-14pt}\bigcup_{\phantom{\hspace{12pt}}p_f\in P(\fadele)} \hspace{-16pt}\Sigma(\mX^0, P_1, p_f)$ is invariant under left   multiplication by $P_1(\fadele)$ on the second factor.
\end{enumerate}

Furthermore, 

\begin{enumerate}\setcounter{enumi}{5}
\item $\Sigma$ is \emph{finite} if $P(\Q) \backslash \Sigma /K_f$ is finite.

\item $\Sigma$ is \emph{complete} if it is finite and every $\Sigma(\mX^0, P_1, p_f)$ is a complete rational convex polyhedral cone decomposition of $-C^*(\mX^0, P_1)\times \{p_f\}\subset U_1(\R)(-1)\times \{p_f\}$.

\item $\Sigma$ is \emph{smooth with respect to $K_f$} if for every rational boundary component $(P_1, \mX_1)$ and $\mX^0 \subset \mX^+_{(P_1, \mX_1)}$, $\Sigma(\mX^0, P_1, p_f)$ is smooth with respect to the arithmetic lattice $\Gamma_{U_1}(-1) \subset U_1(\Q)(-1)$, where $\Gamma_{U_1}$ is the image under the projection $Z(P) \times U_1 \rightarrow U_1$ of the group $\left ( \Id_{Z(P)(\Q)}(\mX)\times U_1(\Q) \right )\cap p_f K_f p_f^{-1}$.

\end{enumerate}
\end{subsecdefinition}

One can construct new admissible cone decompositions out of a given one. Let $(P, \mX)$ be a mixed Shimura datum, $K_f\subset P(\fadele)$ an open compact subgroup and $\Sigma$ a $K_f$-admissible partial cone decomposition for $(P, \mX)$.

\begin{enumerate}
\item  For all $p_f \in P(\fadele)$, define
\[
[\cdot p_f]^*\Sigma \coloneq \{ (u, p'_f \cdot p^{-1}_f) |(u,p'_f)\in \sigma \,\text{for some}\, \sigma \in \Sigma \}.
\]
In other words, $[\cdot p_f ]^*\Sigma(\mX^0, P_1, p'_f)=\Sigma(\mX^0, P_1, p'_fp_f)\cdot p_f^{-1} $. This is a $p_f K_f p_f^{-1}$-admissible partial cone decomposition for $(P, \mX)$. It inherits the finiteness or completeness property of $\Sigma$ if applicable.

\item If $\varphi \colon (P, \mX) \rightarrow (P, \mX)$ is an automorphism, then it induces an automorphism of $\mC(P, \mX) \times P(\fadele)$. Define $\varphi^*\Sigma \coloneq \{\varphi^{-1}(\sigma)| \sigma \in \Sigma \}$. This is a $\varphi^{-1}(K_f)$-admissible partial cone decomposition and it inherits the finiteness or completeness property of $\Sigma$ if applicable.

\item For a rational boundary component $(P_1, \mX_1)$ of $(P, \mX)$, we can obtain a $P_1(\fadele)\cap K_f$-admissible cone decomposition $\Sigma|_{(P_1,\mX_1)}$ as follows: For all rational boundary components $(P_2, \mX_2)$ and $\mX^0_1\subset (\mX_1)^+_{(P_2, \mX_2)}$, there corresponds a unique $\mX^0 \subset \mX_{(P_2, \mX_2)}^+ \subset \mX^+_{(P_1,\mX_1)}$, then 
\[
\Sigma|_{ (P_1, \mX_1) }(\mX^0_1 , P_2, p_{1,f})\coloneq \Sigma(\mX^0, P_2, p_{1,f} ).
\]
Note that $\Sigma|_{(P_1, \mX_1)}$ inherits neither finiteness nor completeness property in general.

\item The \emph{restriction of $\Sigma$ to the unipotent fiber} is defined to be $\Sigma^0=\Sigma \cap \left( \cup_{(\mX^0, (P_1, \mX_1))} (-C(\mX^0, P_1)) \times P(\fadele) \right )$ where $(P_1, \mX_1)$ runs over all improper rational boundary components and $\mX^0 $ runs over all connected components in $\mX^+_{(P_1,\mX_1)}=\mX_1\subset \mX$. Thus 
\[
 \bigcup_{(\mX^0, (P_1, \mX_1))} (-C(\mX^0, P_1)) \times P(\fadele)= \bigcup_{(\mX^0, (P_1, \mX_1))} U(\R)(-1) \times P(\fadele),
\]
whence the name. In case $\Sigma=\Sigma^0$, $\Sigma$ is said to be concentrated on the unipotent fiber. 
\end{enumerate}

Let $(P, \mX)$ be a mixed Shimura datum with a neat  open compact subgroup $K_f \subset P(\fadele)$ and a $K_f$-admissible partial cone decomposition $\Sigma$ concentrated on the unipotent fiber. Define $(P', \mX') \coloneq (P, \mX)/U$ and let $K'_f \subset P'(\fadele)$ be the image of $K_f$ under the projection $\pi_U \colon (P,\mX) \rightarrow (P', \mX')$. Fix $[x,p_f]\in M^{K_f}(P, \mX)(\C)$. Then by the results from Subsection \ref{subsec: mixed shimura varieties}, the fiber over $\pi_U([x, p_f])$ is isomorphic to $\Gamma_U \backslash U(\C)$, where $\Gamma_U$ is the image of 
\[
\left( \Id_{Z(P)(\Q)}(\mX)\times U(\Q) \right ) \cap p_f K_f p_f^{-1}
\]
under $Z(P)\times U \rightarrow U$, which is an arithmetic lattice. $\Gamma_U \backslash U(\C)$ is an algebraic torus. Regarding $C^\times $ as $\C / \Z(1)$ via the exponential map, its cocharacter group is 
\[
\Hom(\frac{\C}{\Z(1)} , \Gamma_U \backslash U(\C)) = \Hom(\Z(1), \Gamma_U)=\Gamma_U(-1) \subset \Gamma_U\otimes \R (-1)= U(\R)(-1).
\]

Let $\mX^0$ be the connected component containing $x$ and $(P_1, \mX_1)$ the unique improper boundary component containing $\mX^0$. Since $\Sigma(\mX^0, P_1, p_f) \subset U(\R)(-1)$ is a rational partial polyhedral decomposition of $U(\R)(-1)$, it defines a torus embedding of each fiber $\Gamma_U \backslash U(\C)$. It follows that we have a relative torus embedding  over $M^{K_f'}(P', \mX')(\C)$. Denote this torus embedding by $M^{K_f}(P,\mX, \Sigma)(\C)$. We thus have a diagram.
\begin{equation*}
\begin{tikzcd}[column sep=5pt]
M^{K_f}(P,\mX )(\C) \arrow[rr, rightarrowtail] \arrow[dr, "\pi_U" ', twoheadrightarrow]&&M^{K_f}(P,\mX , \Sigma)(\C) \arrow[dl, twoheadrightarrow]\\
&M^{K_f'}(P', \mX')(\C)&
\end{tikzcd}
\end{equation*}

Neatness of $K_f$ implies that $M^{K_f}(P, \mX)(\C)$ is smooth. In this case, the space $M^{K_f}(P,\mX, \Sigma)(\C)$ is smooth if and only if $\Sigma$ is smooth with respect to $K_f$. If $K_f$ is not neat, we define $M^{K_f}(P,\mX , \Sigma)(\C) $ to be the quotient of $M^{\tilde{K}_f}(P,\mX , \Sigma)(\C) $ under the right action of $K_f / \tilde{K}_f$, where $\tilde{K}_f \unlhd K_f$ is any neat open compact normal subgroup. This makes sense since we can extend the action of $K_f / \tilde{K}_f$ on $M^{\tilde{K}_f}(P , \mX)(\C)$ to  $M^{\tilde{K}_f}(P,\mX , \Sigma)(\C) $ by the $K_f$-admissibility of $\Sigma$. Up to canonical isomorphism, the resulting $M^{K_f}(P,\mX , \Sigma)(\C) \rightarrow M^{K_f'}(P', \mX')(\C)$  does not depend on the choice of $\tilde{K}_f$.

Now we  construct the covering $\mU$ mentioned at the beginning of this section. Again let $(P, \mX)$ be a mixed Shimura datum with open compact subgroup $K_f \subset P(\fadele)$ and $\Sigma$ a $K_f$-admissible partial cone decomposition. Suppose $(P_1, \mX_1)$ is a rational boundary component of $(P, \mX)$ and $p_f \in P(\fadele)$. We know that there is an open embedding $\mX^+_{(P_1, \mX_1)} \rightarrowtail \mX_1$, which is bijective on connected components and $P_1(\R)\cdot Q(\R)^0\cdot U(\C)$-equivariant with $Q$ the $\Q$-admissible parabolic subgroup related to $(P_1, \mX_1)$. Define $K^1_f \coloneq P_1(\fadele) \cap p_f K_fp_f^{-1} $ and  let 
$
\mU(P_1,\mX_1, p_f) \coloneq  P_1(\Q) \backslash \mX^+_{(P_1, \mX_1)}\times P_1(\fadele) /K_f^1 
$.
Since \[ P_1(\Q) \backslash \mX^+_{(P_1, \mX_1)}\times P_1(\fadele) /K_f^1  \rightarrowtail  P_1(\Q) \backslash \mX_1 \times P_1(\fadele)/ K_f^1, \] 
$ \mU(P_1, \mX_1, p_f) \rightarrowtail M^{K_f^1}(P_1, \mX_1)(\C)$ as an open subspace. There  is a map to $ M^{K_f}(P, \mX)(\C)$ as follows.
 \begin{IEEEeqnarray*}{cCccc}
  \beta(P_1,\mX_1, p_f)&\colon &\mU(P_1, \mX_1, p_f) &\rightarrow& M^{K_f}(P, \mX)(\C)\\
                                            &   & [x, p_{1,f}]                      &  \longmapsto     & [x, p_{1,f}\cdot p_f]
\end{IEEEeqnarray*}
We take $ \mU $ to be the disjoint union of $\mU(P_1, \mX_1, p_f)$ for all $(P_1, \mX_1)$ and $p_f$. 
\begin{IEEEeqnarray*}{rCl}
\mU& \coloneq& \bigsqcup_{(P_1, \mX_1, p_f)} \mU(P_1, \mX_1, p_f)
\end{IEEEeqnarray*}

The various $\beta(P_1, \mX_1, p_f)$ then combine to give a surjective map $\beta \colon \mU \twoheadrightarrow M^{K_f}(P, \mX)(\C)$.  As explained in {\cite[6.10]{Pink}}, $\beta$ is indeed a ramified covering and $M^{K_f}(P, \mX)(\C)$ inherits its structure of a complex space from that of $\mU$. Let $\sim$ be the equivalence relation on $\mU$ defined by this map. Since $M^{K_f}(P, \mX)(\C)$ is Hausdorff, $\sim$ is a closed subset of $\mU \times \mU$.

Let $(P'_1, \mX_1', p_f')$ be another such triple. There are several basic morphisms  $\mU(P_1, \mX_1, p_f) \rightarrow \mU(P_1', \mX_1', p_f')$.
\begin{enumerate}
\item If $(P_1', \mX_1')=(P_1, \mX_1)$ and $p'_f = p_f\cdot k_f$ for $k_f \in K_f$, then $K_f^1={K_f^1}'$ and the identity map on $M^{K_f^1}(P_1, \mX_1)(\C)$ induces
\begin{IEEEeqnarray*}{cCcCc}
[\mathrm{id}]& \colon &\mU(P_1, \mX_1, p_f) &\rightarrow& \mU(P_1, \mX_1, p_f  \cdot k_f).
\end{IEEEeqnarray*}

\item If $(P_1', \mX_1')=(P_1, \mX_1)$ and $p'_f = p^{-1}_{1,f}\cdot p_f $ for some $p_{1,f} \in P_1(\fadele)$, then right multiplication by $p_{1,f}$ induces
\begin{IEEEeqnarray*}{cCcCc}
[\cdot p_{1,f}]&\colon &\mU(P_1, \mX_1, p_f) &\rightarrow& \mU(P_1, \mX_1, p_{1,f}^{-1} \cdot p_f ).
\end{IEEEeqnarray*}

\item If $(P_1', \mX_1')=\mathrm{int}(p) (P_1, \mX_1)$ and $p'_f =p \cdot p_f $ for some $p  \in P (\fadele)$, then $\mathrm{int}(p)$ also induces
\begin{IEEEeqnarray*}{cCcCc}
[\mathrm{int}(p)]& \colon &\mU(P_1, \mX_1, p_f) &\rightarrow& \mU(p\cdot P_1,  p\cdot \mX_1, p \cdot p_f ).
\end{IEEEeqnarray*}
(Here $p\cdot P_1$ means $p \, P_1 \, p^{-1}$.)

\item If $(P_1, \mX_1)$ is a rational boundary component of  $(P_1', \mX_1') $ and $p'_f =  p_f $, then $\mX^+_{(P_1,\mX_1)} \subset \mX^+_{(P_1', \mX_1')}$ and this inclusion map induces
\begin{IEEEeqnarray*}{cCcCc}
\beta(P_1', \mX_1', P_1, \mX_1, p_f)&\colon &\mU(P_1, \mX_1, p_f) &\rightarrow& \mU(P'_1, \mX'_1,  p_f )\\
&&   [x, p_{1,f}] &\mapsto& [x, p_{1,f}].
\end{IEEEeqnarray*}
\end{enumerate} 

\begin{subsecremark}
Note that only maps of type (4) can possibly change the rational boundary component involved.
\end{subsecremark}

These maps are subject to certain commutativity laws (e.g., $[\mathrm{id}] \circ [\cdot p_{1,f}] = [\cdot p_{1,f}] \circ [\mathrm{id}]$). We will omit them and leave the verification to the reader. If $f \colon \mU(P_1, \mX_1, p_f) \rightarrow \mU(P_1', \mX_1' ,p_f')$ is one of the above four types of maps, then it is clear that $\beta(P_1', \mX_1', p_f') \circ f = \beta(P_1, \mX_1, p_f)$. Thus the equivalence relation on $\mU \times \mU$ generated by the graphs of these maps is contained in $\sim$. The lemma below states that the two equivalence relations indeed coincide.

\begin{subseclemma}
Let $u\in \mU(P_1, \mX_1, p_f)$ and $u' \in \mU(P_1', \mX_1', p_f')$. Then $u \sim u'$ if and only if there exists a rational boundary component $(P_2, \mX_2)$ of $(P, \mX)$, $p \in P(\Q)$, $p_{2,f}\in P_2(\fadele)$, and $k_f \in K_f$ such that 
\begin{enumerate}
\item $(P_1, \mX_1)$ is a rational boundary component of $(P_2, \mX_2)$;

\item $(P_1' , \mX_1')$ is a rational boundary component of $\mathrm{int}(p) (P_2, \mX_2)$;

\item $p_f'=p \cdot p_{2,f} \cdot p_f \cdot k_f$;

\item $u$ and $u'$ map to the same image under the following diagram.

\begin{equation*}
\begin{tikzcd}
\mU(P_1, \mX_1, p_f)\ni u 
         \arrow[d,"{\beta(P_2, \mX_2, P_1, \mX_1, p_f)}" ']   
& u'\in \mU(P_1', \mX_1', p_f') 
         \arrow[d, "{ \beta(p\cdot P_2, p\cdot \mX_2, P_1', \mX_1', p_f')  }", start anchor={[yshift=2pt]}]\\
\mU(P_2 , \mX_2, p_f)                                               
&  \mU(p\cdot P_2, p\cdot \mX_2, p_f')  \\
\mU(P_2, \mX_2, p_{2,f}\cdot p_f) 
         \arrow[u, "{ [\cdot p_{2,f}] }"  , " \wr" ' , start anchor={[yshift=-2pt]}] 
         \arrow[r, "{ [\mathrm{id}] }", "\sim" ' ]  
& \mU(P_2, \mX_2, p_{2,f} \cdot p_f\cdot k_f) 
          \arrow[u, " { [\mathrm{int}(p)] }" ', "\wr",  start anchor={[yshift=-2pt]}]
\end{tikzcd}
\end{equation*} 

\end{enumerate}
\end{subseclemma}

Given $\mU$ and $\sim$, we would like to obtain a larger covering $\overbar{\mU}$ such that $\sim$ can be extended to a well-behaved equivalence relation on $\mUbar$. For this, consider a rational boundary component $(P_1,\mX_1)$ and $p_f \in P(\fadele)$. Define
\begin{IEEEeqnarray*}{rCl}
\Sigma_{(P_1, \mX_1, p_f)} &\coloneq &\left ( \left ( [\cdot p_f]^*\Sigma \right )|_{ (P_1, \mX_1)} \right)^0.
\end{IEEEeqnarray*}
Thus
\begin{IEEEeqnarray*}{rCl}
 \Sigma_{(P_1, \mX_1, p_f)}(\mX_1^0, P_2, p_{1,f})&=& \Bigl ( \Sigma(\mX^0, P_2, p_{1,f}\cdot p_f)p_f^{-1} \Bigr ) \bigcap (-C^*(\mX^0_1, P_1)) \times \{p_{1,f}\}\\
&=& \Bigl ( \Sigma(\mX^0, P_2, p_{1,f}\cdot p_f)p_f^{-1} \Bigr ) \bigcap U_1(\R)(-1) \times \{p_{1,f}\},
\end{IEEEeqnarray*}
for all rational boundary components $(P_2, \mX_2)$ of $(P_1,\mX_1)$ and $\mX^0 \rightarrowtail \mX^0_1 \subset (\mX_1)^+_{(P_2, \mX_2)}$. $\Sigma_{(P_1, \mX_1, p_f)}$ is a $K_f^1$-admissible cone decomposition for $(P_1, \mX_1)$ concentrated on the unipotent fiber. We have 
\[
M^{K_f^1}(P_1, \mX_1)(\C) \rightarrowtail M^{K_f^1}(P_1, \mX_1, \Sigma_{(P_1, \mX_1, p_f)})(\C)
\] as an open dense subset. Let $\mUbar(P_1, \mX_1, p_f)$ be the interior of the closure of $\mU(P_1,\mX_1, p_f) \subset M^{K_f^1}(P_1, \mX_1, \Sigma_{(P_1, \mX_1, p_f)})(\C)$, which contains $\mU(P_1, \mX_1, p_f)$ as an open dense subset (in the analytic topology). 

Let $\sigma \in  \Sigma_{(P_1, \mX_1, p_f)}(\mX^0_1, P_1, p_{1,f})=\Sigma(\mX^0, P_1, p_{1,f}\cdot p_f) p_f^{-1}$. We denote the $\sigma$-stratum by 
\[
M^{K_f^1}(P_1, \mX_1)(\C)_\sigma \subset  M^{K_f^1}(P_1, \mX_1, \Sigma_{(P_1, \mX_1, p_f)}).
\]
Strictly speaking, distinct connected components of $M^{K_f^1}(P_1, \mX_1)(\C)$ correspond to distinct $p_{1,f}$. However, since $\Sigma_{(P_1, \mX_1, p_f)}$ is concentrated on the unipotent fiber and $(P_1, \mX_1)$ is irreducible, the cone decomposition $\Sigma_{(P_1, \mX_1, p_f)}(\mX^0_1, P_1, p_{1,f})$
does not depend on the second factor.

As claimed in {\cite[6.13]{Pink}}, the subspace $\mUbar(P_1, \mX_1, p_f) \cap M^{K_f^1}(P_1, \mX_1)(\C)_\sigma$ has an explicit description. We will discuss this again in Subsection \ref{subsec: mUbar on each stratum}. 
 
Take 
\[
\dps \mUbar\coloneq \bigsqcup_{(P_1, \mX_1, p_f)} \hspace{-2pt} \mUbar(P_1,\mX_1, p_f)
\] to be the disjoint union of all $\mUbar(P_1, \mX_1, p_f)$. It contains $\mU$ as an open dense subset. The next step is to extend the equivalence relation $\sim$ to $\mUbar$. For this, we need to first extend the above four types of maps $\mU(P_1, \mX_1, p_f) \rightarrow \mU(P_1', \mX'_1, p_f')$. 

Each map of type (1), (2), or (3) is induced by an isomorphism 
\[
M^{K_f^1}(P_1, \mX_1, p_f )\rightarrow M^{{K^1_f}'}(P_1', \mX_1', p_f')
\]
that maps $\Sigma_{(P_1, \mX_1, p_f)}$ into $\Sigma_{(P_1', \mX_1', p_f')}$. Thus this isomorphism extends to an isomorphism 
\[ 
M^{K_f^1}(P_1, \mX_1, \Sigma_{(P_1, \mX_1, p_f)}) \rightarrow M^{{K_f^1}'}(P_1', \mX_1', \Sigma_{(P'_1, \mX'_1, p'_f)})  
\]
and hence induces a canonical map $ \mUbar(P_1, \mX_1, p_f) \rightarrow \mUbar(P_1', \mX_1', p_f')$.

To extend the type (4) map, consider $\left. \Sigma_{(P_1' ,\mX_1', p_f)} \right |_{(P_1, \mX_1)}$. By definition,
\begin{IEEEeqnarray*}{rCl}
&&\left.\Sigma_{(P'_1, \mX'_1, p_f)}\right |_{(P_1, \mX_1)}(\mX^0_1, P_1, p_{1,f}) \\
&=&\left (  
  \Sigma(\mX^0, P_1, p_{1,f}\cdot p_f)p_f^{-1} \right ) \bigcap U'_1(\R)(-1) \times \{p_{1,f}\}\\
&\subset&\left (  
  \Sigma(\mX^0, P_1, p_{1,f}\cdot p_f)p_f^{-1} \right ) \bigcap U_1(\R)(-1) \times \{p_{1,f}\}\\
&=&\Sigma_{(P_1, \mX_1, p_f)}(\mX^0_1, P_1, p_{1,f}).
\end{IEEEeqnarray*}
It follows that $\beta(P_1', \mX_1', P_1, \mX_1, p_f)\colon \mU(P_1, \mX_1, p_f) \rightarrow \mU(P_1', \mX'_1, p_f)$ extends to a map $\overbar{\beta}(P_1', \mX_1', P_1, \mX_1, p_f)$:
\begin{IEEEeqnarray*}{rCl}
\mUbar(P_1, \mX_1, p_f)\cap M^{K_f^1} (P_1, \mX_1, \left. \Sigma_{(P'_1, \mX'_1, p_f)}\right |_{(P_1, \mX_1)}  ) &\rightarrow\,& \mUbar(P_1', \mX'_1, p_f).
\end{IEEEeqnarray*}
The image of this map is a union of connected components of $\mUbar(P_1', \mX'_1, p_f)$ (cf. {\cite[6.15]{Pink}}). We denote $\mUbar(P_1, \mX_1, p_f)\cap M^{K_f^1}(P_1, \mX_1, \left. \Sigma_{(P'_1, \mX'_1, p_f)} \right |_{(P_1, \mX_1)})$ by $\mUbar(P_1, \mX_1, p_f)_{\left. \Sigma(P_1', \mX_1', p_f)\right. |_{(P_1,\mX_1)}}$.

Now with these extended maps, we may define an extended equivalence relation on $\mUbar \times \mUbar$ as follows.

\begin{subseclemma}[{\cite[6.17]{Pink}}]
Let $u\in \mUbar(P_1, \mX_1, p_f)$ and $u' \in \mUbar(P_1', \mX_1', p_f')$. Define $u \overbar{\sim} u'$ if and only if there exists a rational boundary component $(P_2, \mX_2)$ of $(P, \mX)$, $p \in P(\Q)$, $p_{2,f}\in P_2(\fadele)$, and $k_f \in K_f$ such that 
\begin{enumerate}
\item $(P_1, \mX_1)$ is a rational boundary component of $(P_2, \mX_2)$;

\item $(P_1' , \mX_1')$ is a rational boundary component of $\mathrm{int}(p) (P_2, \mX_2)$;

\item $p_f'=p \cdot p_{2,f} \cdot p_f \cdot k_f$;

\item $u$ and $u'$ map to the same image under the following diagram.
\begin{equation*}\hspace{-10pt}\adjustbox{center}{
\begin{tikzcd}[column sep = small]
\mUbar(P_1, \mX_1, p_f)_{\left. \Sigma(P_2, \mX_2, p_f)\right. |_{(P_1,\mX_1)}} \ni u 
\arrow[d,
              "{\overbar{\beta}(P_2, \mX_2, P_1, \mX_1, p_f)}" ']  
&   u'\in\mUbar(P'_1, \mX'_1, p'_f)_{\left. \Sigma(p\cdot P_2, p\cdot \mX_2, p'_f) \right |_{(P'_1,\mX'_1)}}  
\arrow[d,
               "{ \overbar{\beta}(p\cdot P_2, p\cdot \mX_2, P_1', \mX_1', p_f')  }",
              start anchor={[yshift=5pt]} ]\\
\mUbar(P_2 , \mX_2, p_f)                                               
&  \mUbar(p\cdot P_2, p\cdot \mX_2, p_f')  \\
\mUbar(P_2, \mX_2, p_{2,f}\cdot p_f) 
\arrow[u, 
             "{ [\cdot p_{2,f}] }"  , " \wr" ',
               start anchor={[yshift=-2pt]}  ] 
\arrow[r, 
              "{ [\mathrm{id}] }", 
               "\sim" ']
& \mUbar(P_2, \mX_2, p_{2,f} \cdot p_f\cdot k_f) 
\arrow[u, 
               " { [\mathrm{int}(p)] }" ', "\wr",
               start anchor={[yshift=-2pt]}  ]
\end{tikzcd} }
\end{equation*}
\end{enumerate}
 
This is indeed an equivalence relation on $\mUbar \times \mUbar$, containing $\sim~\subset \mU \times \mU$ and contained in the closure of $\sim$ in $\mUbar \times \mUbar$.
\end{subseclemma}

\begin{subsectheorem}[{\cite[6.22-6.24]{Pink}}]\label{subsectheorem: toroidal compactifications, toroidal compactifications}
 Let $(P, \mX)$ be a mixed Shimura datum, $K_f \subset P(\fadele)$ an open compact subgroup and $\Sigma$ a $K_f$-admissible partial cone decomposition for $(P, \mX)$. Let $(P_1, \mX_1)$ be any rational boundary component of $(P, \mX)$ and $Q\subset P$ the $\Q$-admissible parabolic subgroup for $(P_1, \mX_1)$.

\begin{enumerate}
\item The equivalence relation $\overbar{\sim} \subset \mUbar \times \mUbar$ is the closure of $\sim~\subset \mU \times \mU$.  Thus the quotient space $\mUbar/\overbar{\sim}$ is Hausdorff. 

\item $\mUbar/\overbar{\sim}$ is locally isomorphic to an open subset $\Delta_1\backslash \mathcal{W}_1\subset \Delta_1 \backslash \mUbar(P_1, \mX_1, p_f)$, where $\Delta_1$ is an arithmetic subgroup of $(Q/P_1)(\Q)$ that acts properly discontinuously  and analytically. Therefore $\mUbar/\overbar{\sim}$ is a normal complex space containing $M^{K_f}(P,\mX)(\C)=\mU/\sim$ as an open dense subset.

\item If $\Sigma$ is concentrated on the unipotent fiber, it is canonically isomorphic to the $M^{K_f}(P, \mX, \Sigma)(\C)$ defined before.
\end{enumerate}

Define $M^{K_f}(P,\mX, \Sigma)(\C) \coloneq \mUbar/\overbar{\sim}$. This is the \textbf{toroidal compactification of $\bm{M^{K_f}(P,\mX)(\C)}$ with respect to $\bm{\Sigma}$}.

\begin{enumerate}\setcounter{enumi}{3}
\item $M^{K_f}(P, \mX, \Sigma)(\C)$ is compact if $\Sigma$ is a complete cone decomposition with respect to $K_f$.

\item If $K_f$ is neat and $\Sigma$ is smooth with respect to $K_f$, then $M^{K_f}(P, \mX, \Sigma)(\C)$ is smooth.
\end{enumerate}
\end{subsectheorem}

\begin{subsecremark} \label{subsecremark: toroidal compactifications, action of the normalizer}
Let $\delta_Q\coloneq \Stab_{Q(\Q)}(\mX_1) \cap P_1(\fadele) \cdot p_f  \cdot K_f \cdot p_f^{-1}$. As $P_1(\fadele)/K_f^1 \overset{\sim}{\rightarrow} P_1(\fadele)\cdot p_f \cdot K_f/K_f$, the group $\Delta_Q\coloneq \delta_Q /P_1(\Q)$ acts equivariantly on 
\[
\mU(P_1, \mX_1, p_f) \rightarrow M^{K_f^1}(P_1, \mX_1)(\C) 
\rightarrow M^{K_f^1}(P_1, \mX_1 ,\Sigma_{ (P_1, \mX_1, p_f) })(\C)
\]
via left multiplication on both factors. It is not hard to see that $\Delta_Q$ acts through a composition of morphisms of type (1), (2), and (3). This action thus extends to an action on $\mUbar(P_1, \mX_1, p_f)$. It follows that the quotient map $\mUbar(P_1, \mX_1, p_f) \rightarrow M^{K_f}(P, \mX, \Sigma)(\C)$ factors through $\Delta_Q \bs \mUbar(P_1, \mX_1, p_f)$. 
\end{subsecremark}

The toroidal compactification has certain functorial properties. Let $K_f', K_f \subset P(\fadele)$ be open compact subgroups with $K_f' \subset p_f K_f p_f^{-1}$ and let $\Sigma$ (resp. $\Sigma'$) be a $K_f$-admissible (resp. $K_f'$-admissible) cone decomposition. Suppose that for all $\tau \in \Sigma'$, there exists $\sigma \in [\cdot p_f]^*\Sigma$ such that $\tau \subset \sigma$. Then we have a holomorphic map
\begin{IEEEeqnarray*}{cCccccc}
[\cdot p_f]&\colon &M^{K'_f}(P, \mX, \Sigma')(\C)
&\rightarrow
&\bigl ( M^{p_fK_fp_f^{-1}}(P, \mX, [\cdot p_f]^*\Sigma)(\C)
&\overset{\sim}{\rightarrow}  \bigr)
&M^{K_f}(P, \mX, \Sigma)(\C).
\end{IEEEeqnarray*}
If $K_f' =p_f K_f p_f^{-1}$ and $\Sigma' = [\cdot p_f]^*\Sigma$, this is an isomorphism. If $K_f' =p_f K_f p_f^{-1}$ and $\Sigma' \subset [\cdot p_f]^*\Sigma$, this is an open embedding with the complement of the image a closed analytic subset. If $\Sigma' =\Sigma$ and $K_f' \unlhd K_f$ is a normal subgroup, then the finite group $K_f/K_f'$ acts from the right on $M^{K'_f}(P, \mX, \Sigma)(\C)$. The canonical projection $M^{K'_f}(P, \mX, \Sigma)(\C)\rightarrow M^{K_f}(P, \mX, \Sigma)(\C)$ is the quotient map under this action.

Consider a morphism of mixed Shimura data $\varphi \colon (P, \mX) \rightarrow (P', \mX')$ and open compact subgroups $K_f \subset P(\fadele), K_f'\subset P'(\fadele)$ with cone decompositions $\Sigma, \Sigma'$ respectively, such that $\varphi(K_f) \subset K_f'$. Suppose that for all $\sigma \in \Sigma$, there exists a $\sigma' \in \Sigma '$ such that $\varphi(\sigma)\subset \sigma'$. Then we have an induced holomorphic map
\begin{IEEEeqnarray*}{cCccc}
[\varphi]&\colon &M^{K_f}(P, \mX, \Sigma)(\C)&\rightarrow&M^{K'_f}(P', \mX', \Sigma')(\C).
\end{IEEEeqnarray*}
In case $\varphi$ is an automorphism of $(P, \mX)$, $K'_f=\varphi(K_f)$, and $\Sigma =\varphi^{*}(\Sigma')$, $[\varphi]$ is an isomorphism.

The toroidal compactification $M^{K_f}(P, \mX, \Sigma)(\C)$ is related to the Baily-Borel compactification in the following way. According to {\cite[6.2]{Pink}}, if $(P, \mX) $ is a pure Shimura datum with open compact $K_f\subset P(\fadele)$, then the Baily-Borel compactification of $M^{K_f}(P, \mX)(\C)$ can be defined as 
\begin{IEEEeqnarray*}{rCl}
M^{K_f}(P, \mX)^*(\C)&\coloneq& P(\Q)\backslash \mX^*\times P(\fadele) / K_f,
\end{IEEEeqnarray*}
where 
\[
\dps \mX^*=\bigsqcup_{(P_1, \mX_1)} {\mX_1/W_1}.
\]
$\mX^*$ as a set is the disjoint union of $\mX_1/W_1$, as $(P_1, \mX_1)$ runs over all rational boundary components of $(P, \mX)$. It carries the Satake topology generated by 
\[
\dps \bigsqcup_{(P_2, \mX_2)\geq (P_1, \mX_1)}\psi_2(\mathrm{im}^{-1}(\Lambda\cdot D)\cap \psi_1^{-1}(\mathcal{Y})),
\] 
where $\psi_i \colon \mX_i \rightarrow \mX_i /W_i$ is the projection, $\Lambda \in \R_{>0}$, and $D\subset C(\mX^0, P_1)$ is a convex core. The sum runs over all rational boundary components $(P_2,\mX_2) \geq (P_1, \mX_1)$ We will not discuss the topology in detail but the notion of a core will be introduced in Subsection \ref{subsec: cones and cores in euclidean spaces}.

\begin{subsecproposition}\label{subsecproposition: toroidal compactifications, map to baily-borel}
Let $(P, \mX)$, $K_f$, and $\Sigma$ be as above, with $\pi \colon (P, \mX) \rightarrow (P, \mX)/W$ the canonical projection. Then the canonical map 
\begin{IEEEeqnarray*}{c}
[\pi]\colon \mU \rightarrow M^{K_f}(P, \mX)(\C) \rightarrow M^{\pi(K_f)}(P/W,\mX/W)(\C)\end{IEEEeqnarray*}
extends uniquely to a  holomorphic map $[\pi]^*\colon  \mUbar \rightarrow M^{\pi(K_f)}(P/W, \mX/W)^*(\C)$. This map factors through $\mUbar/\overbar{\sim}$. In other words, we have the diagram below.

\begin{equation*}
\begin{tikzcd}
\mU \arrow[r, rightarrowtail] \arrow[d, twoheadrightarrow]   &     \mUbar \arrow[d, twoheadrightarrow]\\
M^{K_f}(P, \mX)(\C)  \arrow[r, rightarrowtail] \arrow[d, twoheadrightarrow, "{[\pi]}" ', start anchor={[yshift=2pt]}]
& M^{K_f}(P, \mX, \Sigma)(\C) \arrow[d, "{[\pi]^*}", start anchor={[yshift=2pt]}]\\
M^{\pi(K_f)}(P/W, \mX/W)(\C)     \arrow[r, rightarrowtail]            &  M^{\pi(K_f)}(P/W, \mX/W)^*(\C)
\end{tikzcd}
\end{equation*}  
\end{subsecproposition}

We have now constructed the toroidal compactification, though we assumed there exists a $K_f$-admissible partial cone decomposition $\Sigma$. Actually, the existence of a suitable $\Sigma$ is a non-trivial question and in Chapters 8 and 9 of \cite{Pink} it takes considerable efforts to obtain one. We conclude this subsection by listing the main theorems therein.

\begin{subsectheorem}[{\cite[9.21]{Pink}}]\label{subsectheorem: toroidal compactifications, projective structure} 
Let $(P, \mX)$ be a mixed Shimura datum and $K_f\subset P(\fadele)$ a neat open compact subgroup. There exists a $K_f$-admissible complete cone decomposition $\Sigma$ for $(P, \mX)$ such that $M^{K_f}(P, \mX, \Sigma)(\C)$ is smooth, projective, and the complement $M^{K_f}(P, \mX, \Sigma)(\C) \backslash M^{K_f}(P, \mX)(\C)$ is a union of smooth divisors with only normal crossings. 

If $\Sigma '$ is any $K_f$-admissible complete cone decomposition for $(P, \mX)$, then $\Sigma$ can be chosen to be a refinement of $\Sigma '$. 
\end{subsectheorem}

\begin{subsectheorem}[{\cite[9.24, 9.25]{Pink}}]\label{subsectheorem: toroidal compactifications, quasi-projective structure} 
Let $(P, \mX)$ be a mixed Shimura datum and $K_f\subset P(\fadele)$ an open compact subgroup. 

\begin{enumerate}
\item The mixed Shimura variety $M^{K_f}(P, \mX)(\C)$ admits a canonical structure of a normal quasi-projective algebraic variety over $\C$, such that all the functorial maps discussed in Subsection \ref{subsec: mixed shimura varieties} become algebraic morphisms. 

\item If $\Sigma$ is a complete $K_f$-admissible cone decomposition such that the space $M^{K_f}(P, \mX, \Sigma)(\C)$ is projective, then the open embedding $M^{K_f}(P, \mX)(\C) \rightarrowtail M^{K_f}(P, \mX, \Sigma)(\C)$ is algebraic. 
\end{enumerate}

We denote by $M^{K_f}_\C(P, \mX)$ the normal complex algebraic variety above underlying $M^{K_f}(P, \mX)(\C)$. Similarly, $M^{K_f}_\C(P, \mX, \Sigma)$ denotes the variety underlying $M^{K_f}(P, \mX, \Sigma)(\C)$ if there exists one. When $(P, \mX)$ is a pure Shimura datum, $M^{K_f}_\C(P, \mX)^*$ denotes the normal projective variety underlying the Baily-Borel compactification $M^{K_f}(P, \mX)^*(\C)$.

\begin{enumerate}\setcounter{enumi}{2}
\item $M^{K_f}_\C(P, \mX)$, $M^{K_f}_\C(P, \mX, \Sigma)$ (if it exists), and $M^{K_f}_\C(P, \mX)^*$ are unique up to canonical isomorphism.
\end{enumerate}
\end{subsectheorem}

\begin{subsectheorem}[{\cite[9.21]{Pink}}] \label{subsectheorem: toroidal compactifications, canonical model} 
Every mixed Shimura datum $(P, \mX)$ admits a canonical model (in the sense of {\cite[11.5]{Pink}}) over its reflex field $E(P, \mX)$.
\end{subsectheorem}

This finishes our quick tour into the theory of mixed Shimura varieties and their toroidal compactifications. We start the proof of our main theorem in the next section, after establishing several technical lemmas.


\section{Proof of the main theorem} \label{sec: proof of main theorem}


A crucial step in the proof of our main theorem is to demonstrate that some neighborhood of a certain stratum in a torus embedding maps into the toroidal compactification. For this purpose, we need some preparation. 

The essential goal of the first two subsections is to establish Lemma \ref{subseclemma: some lemmas, fundamental lemma}. It is a mild generalization of the claim in \cite[6.13]{Pink}, which states that the intersection of $\mUbar(P_1, \mX_1, p_f)$ with each $\sigma$-stratum is equal to $\pi_\sigma(\mUbar(P_1, \mX_1, p_f))$, where $\pi_\sigma$ is the canonical projection onto the $\sigma$-stratum. We provide a proof of this characterization and extend it to subspaces of $\mUbar(P_1, \mX_1, p_f)$ that are preimages of open convex cores $D\subset C(\mX^0, P_1)$ under the map of imaginary part. 

The proof of Theorem \ref{sectheorem: introduction, main theorem} requires the case of $\mUbar(P_1, \mX_1, p_f)$. Readers willing to accept the statement of Lemma \ref{subseclemma: some lemmas, fundamental lemma} may skip its proof and move on to Subsection \ref{subsec: proof of main theorem} directly.

\subsection{Cones and cores in Euclidean spaces}
\label{subsec: cones and cores in euclidean spaces}

Since a toroidal compactification is constructed relative to some $K_f$-admissible cone decomposition, we first study cones and cores in Euclidean spaces. 
The result below may be well known, but we include a proof here for the convenience of the reader.

\begin{subseclemma}\label{subseclemma: some lemmas, projection of open cones}
Let $V$ be a finite-dimensional real vector space and $C\subset V$ an open convex subset. Suppose $\pi \colon V \rightarrow W$ is a surjective linear map of $\R$-vector spaces. Denote by $\overbar{C}$ the closure of $C$ in $V$.

\begin{enumerate}
\item $\pi(C)$ is the interior of $\overbar{\pi(C)}$ in $W$.

\item If $C$ is an open convex cone, then $\pi(\overbar{C})=\overbar{\pi(C)}$.
\end{enumerate}
\end{subseclemma}

\begin{proof}~

\begin{enumerate}
\item $\pi(C)$ is an open convex subset of $W$ as the projection is open and the image of a convex subset under a linear map is also convex. Thus we are reduced to showing that the interior of the closure of an open convex subset $D\subset \R^m$ is $D$ itself.  Let $y$ be a point in the interior of $\overbar{D}$. Take $x\in D$ to be an arbitrary point. We may assume $y\neq x$. There exists an open ball neighborhood of $y$ contained in  $\overbar{D}$. Thus by convexity, for $\epsilon>0$ small enough, $z\coloneq x + (1+\epsilon)(y-x) \in \overbar{D}$. Then $y $ lies on the segment connecting $x$ and $z$. Any small open ball neighborhood $B_\delta \ni z$ contains a point $w$ of $D$. Fix an open ball neighborhood $B_x$ of $x$ contained in $D$. As $\delta$ goes to zero, $w$ approaches $z$.  Since the convex hull of $B_x \cup \{z\}$ contains $y$, so does the convex hull of $B_x\cup \{w\}$ if $w$ is close enough to $z$. But this latter convex hull is contained in $D$ by convexity.

\item Since in this case $\pi(C)$ is an open convex cone, it is equal to the interior of $\overbar{\pi(C)}$ and thus dense in $\overbar{\pi(C)}$. We just have to show that $\pi(\overbar{C})$ is closed. 

Consider the unit balls $B^{n+m}\subset \R^{n+m}$ and $B^m \subset \R^m$ respectively. Any convex cone is determined by its intersection with the unit ball. $\overbar{C} \cap B^{n+m}$ is a closed and hence compact subset. Thus $\pi(\overbar{C}\cap B^{n+m})$ is compact and hence closed. As $\pi(\overbar{C})$ is the cone spanned by $\pi(\overbar{C}\cap B^{n+m})$, it is closed.
\end{enumerate}
\end{proof}

Before the next lemma, we introduce a new notion.

\begin{subsecdefinition}
Let $V$ be a finite-dimensional $\Q$ vector space. Fix a lattice $V(\Z) \subset V$. Let $C\subset V_\R$ be an open convex cone. Define $D_0\subset C$ to be the convex closure of $C \cap V(\Z)$. A subset $D\subset C$ is called a \emph{core} if there exists $\lambda_0, \lambda_0'\in \R_{>0}$ such that $\lambda_0\cdot D_0 \subset D \subset \lambda_0'\cdot D_0$. 
\end{subsecdefinition}

It is not hard to see that $D_0 \supseteq \lambda_0 \cdot D_0 $ for all $\lambda_0 \geq 1$. Thus it can be shown that the above definition does not depend on the choice of the lattice. Further, if $D\subset C$ is a core, then so are the following: the interior of $D$, the closure $\overbar{D}$ of $D$, the convex closure of $D$, $\R_{\geq 1}\cdot D$ and $D+C$. Also, $C=\R_{>0}\cdot D$ (cf. {\cite[6.1]{Pink}}). 

\begin{subsecremark}
Observe that if $D$ is convex, then $D=\R_{\geq 1}\cdot D$. If $D$ is open, then $D+C=D+\overbar{C}$. For any $v\in C$, we have $D+C+v \subset D+C$. In this case, we say the core $D+C$ is $C$-invariant. Hence, when $D$ is open and $C$-invariant, we have $D+C=D+\overbar{C}=D$. In fact, every open convex core $D$ is $C$-invariant. (For all $v_D \in D$ and all $v\in C$, there exists a small enough $1>\varepsilon_1>0$ and a large enough $\varepsilon_2>1$ such that $v_D+\varepsilon_1\cdot v\in D$ and $v_D+\varepsilon_2 \cdot v \in D$. By convexity, $v_D+v \in D$. So $D+C \subset D \subset D+\overbar{C}=D+C$.)
\end{subsecremark}

For the product space $\R^n\times \R^m$, we take  the standard lattice $\Z^n \times \Z^m$ formed by integral points. For a subset $S\subset \R^l$, $\Int(S)$ denotes its interior. Define $\pi_n \colon \R^n \times \R^m \rightarrow \R^n$ and $\pi_m \colon \R^n \times \R^m \rightarrow \R^m$ to be the projections onto the left and the right factor, respectively. For a subset $S\subset \R^n\times \R^m$, a subset $\sigma^0 \subset \R^n$, and a point $z_0 \in \R^m$, denote the level set $S\cap \R^n \times \{z_0\}$ by $S_{z_0}$ while $S_{\sigma, z_0} \coloneq S \cap \sigma^0 \times \{z_0\}$ and $S(\sigma^0)\coloneq S \cap \sigma^0 \times \R^m$.

\begin{subseclemma}\label{subseclemma: some lemmas, analytic geometric lemma}
Let $C\subset \R^n \times \R^m$ be an open convex cone with $D\subset C$ an open convex core. Write $\overbar{C}$ for the closure of $C$ in $\R^n \times \R^m$ and $\sigma^0$ for the interior of the intersection $\overbar{C}\cap \R^n \times \{0\}$ in $\R^n \times \{0\}$.  Suppose for the open convex cone $\sigma^0\subset \R^n\times \{0\}$, either of the following two conditions are satisfied. 
\begin{enumerate}[label=(\Alph*)]
\item $C\cap \R^n \times \{0\}=\varnothing$.
\item $\sigma^0 \subset C$.
\end{enumerate} 
Then
\begin{enumerate}
\item If $C_{z_0} 
 \neq \varnothing$, then the open convex subset $\pi_n(C_{\sigma^0, z_0}) \subset \sigma^0 \times \{0\}$ generates the open cone $\sigma^0$. That is,  if $(v,z_0) \in C$ for some $v$, then for any $v_\sigma \in \sigma^0$, there exists a $\lambda >0$ such that $(\lambda\cdot v_\sigma, z_0) \in C $.

In particular, $ \pi_m(C)=\pi_m( C(\tau^0))$ for any open convex cone $\tau^0 \subset \sigma^0$. 

In Case (B), $\pi_m(C)=\R^m$.

\item For all $(v, z_0) \in C$ and $(v', 0) \in \overbar{C} \cap \R^n \times \{0\}$, $(v, z_0) +\lambda(v', 0)\in C$ for all $\lambda \geq 1$. 

$D$ generates $C$ as an open convex cone, i.e.  for all $(v, z_0)\in C$ there exists $\lambda \geq 1$, such that $\lambda (v, z_0) \in D$.
\end{enumerate}

Let $f \colon C \rightarrow \R\cup\{\pm \infty\}$ be the function $f(v,z)\coloneq \inf\left \{ t \,\lvert \,(v, t\cdot z)\in D \right \}$. ($f(v,z)=\infty$ if $v + \R \cdot z$ does not meet $D$.) We have 
\begin{enumerate}\setcounter{enumi}{2}
\item Let $(v_\sigma, z_0)\in C(\sigma^0)$. For all bounded subset $K\subset \R^n \times \{0\}$, there exists a $\mu>0$ such that, for all $v'\in \mu \cdot v_\sigma+K$ and all $t \geq 1$, $f(t \cdot v', z_0)< \mu$. 

\item If $(v_0, z_0)\in D$, then for all $v_\sigma \in \sigma^0$, there exists a $\lambda>0$, such that $(\lambda v_\sigma, z_0)\in D$. Equivalently, $ D_{z_0} \neq \varnothing$ implies that $\pi_n(D_{\sigma^0, z_0})$ generates the open cone $\sigma^0$. 

In particular, $\pi_m(D)=\pi_m( D(\tau^0) )$ for any open convex cone $\tau^0 \subset \sigma^0$.

In Case (B), $\pi_m(D)=\R^m$.

\item If $(v_\sigma, z_0)\in D(\sigma^0)$, then the function $f_{z_0}(\lambda)\coloneq f(\lambda v_\sigma,  z_0)$ with $\lambda \geq 1 $ is a decreasing convex continuous function. In Case (A), it is positive for $\lambda \geq 1$. In Case (B), either $f_{z_0}(\lambda)=-\infty$  for all $\lambda \geq 1$ or $f_{z_0}(\lambda)\neq -\infty$ (and thus real-valued) for all $\lambda \geq 1$.

In particular, $(\lambda v_\sigma, z_0) \in D$ for all $\lambda \geq 1$.
\end{enumerate}
\end{subseclemma}

\begin{proof} We first prove everything in Case (A).
\begin{enumerate}
\item 
        Recall that $C({\sigma^0})\coloneq C \cap \sigma^0\times \R^m$. 
Given $(v, z_0)\in C$, let $v_\sigma\in \sigma^0$ be any point. Consider the line segment $\zeta(t)=(1-t)(v_\sigma, 0)+t(v, z_0)=((1-t)v_\sigma+ t v, t z_0)$ ($0\leq t \leq 1$) connecting the two points. As $C$ is an open convex cone and $(v_\sigma, 0)$ lies on the boundary, $\zeta(t)\in C$ for all $0<t \leq 1$. If $t>0$ is sufficiently small, as $\sigma^0$ is open in $\R^n\times \{0\}$, $\zeta(t)$ will lie in $C(\sigma^0)$. Then $t^{-1}\zeta(t)\in C_{\sigma^0, z_0}$. Indeed, we have shown that any $v_\sigma$ can be approached by the open (convex) cone $\mathrm{Cone}(\pi_n(C_{\sigma^0, z_0}))=\R_{>0}\cdot \pi_n(\sigma^0 \times \{z_0\}\cap C)$. So $\sigma^0$ is contained in the closure of $\mathrm{Cone}(\pi_n(C_{\sigma^0, z_0}))$ in $\R^n$ and therefore must be actually contained in $\mathrm{Cone}(\pi_n(C_{\sigma^0, z_0}))$ as any open convex cone is the interior of its closure. Since $\mathrm{Cone}(\pi_n(C_{\sigma^0, z_0})) \subset \sigma^0$, we have the equality as desired.

\item As $\overbar{C}$ is a closed convex cone, $\lambda (v', 0) \in \overbar{C}$ and so $\frac{1}{2}(v, z_0) + \frac{1}{2}\lambda( v', 0) \in C$.  The second assertion follows from the fact that, for all $(v,z)\in C$, there exists a $\lambda>1$ such that $\lambda\cdot (v,z) \in D_0$ and hence $\lambda_0\cdot \lambda \cdot (v,z)\in \lambda_0\cdot D_0 \subset D$.

\item Replace $K$ with any closed ball centered at $(0,0)$ containing $K$. Recall that for some fixed $\lambda_0, \lambda_0' > 0$,  we have $\lambda_0 \cdot D_0 \subset D \subset \lambda_0' \cdot D_0$ where $D_0$ is the convex closure of $C \cap \Z^n \times \Z^m$. A subset of the form $v+ \lambda_0 (I^n \times I^m)$ (for $I^n, I^m$ the unit cubes in $ \R^n$, $\R^m$, respectively) will be called a lattice cube if  $v\in \lambda_0 \cdot (\Z^n \times \Z^m)$.

Since $C$  and $\sigma^0$ are open convex cones, we may choose $\mu>0$ large enough so that (1) any lattice cube meeting $(\mu \cdot v_\sigma, \mu \cdot z_0)+K$ is contained in $C$ and (2) any lattice cube in $\R^n \times \{0\}$ meeting $\mu \cdot v_\sigma +K$ is contained in $\sigma^0$. Now take any $v' \in \mu\cdot v_\sigma +K$ and any $t \geq 1$. As $(t \cdot v' , \mu \cdot z_0)= (v',  \mu \cdot z_0)+ ((t-1) v', 0)$, any lattice cube meeting $(t \cdot v' , \mu \cdot z_0)$ will also be contained in $C$ (and thus contained in $D$) by (2) just proved. In particular, $f(t\cdot v', \mu\cdot z_0)<1$. But $f(t \cdot v', \mu \cdot z_0)= \mu^{-1}f(t \cdot v', z_0)$, and so $f(t \cdot v',  z_0)< \mu$ as needed.

\item First note that since $ D\subset \lambda_0' \cdot D_0=\overbar{\lambda_0' \cdot D_0}$ is convex and $C$ does not meet $\R^n \times \{0\}$, $f(v,z)$ is a positive function on $C$ and convex on the level set $C_z$.

Let $(v_0, z_0)\in D$ (which implies $f(v_0, z_0)<1-\varepsilon$  for some $\varepsilon>0$ small enough). Let $v_\sigma \in \sigma^0$. By (1), we may assume (after extending $v_\sigma$) that $(v_\sigma, z_0)\in C$. By (2), and the fact that $\sigma^0$ is an open convex cone, we may further assume that $v_\sigma' \coloneq v_\sigma-v_0 \in \sigma^0$.

Let $K$ be a closed ball in $\R^n\times \{0\}$ centered at $(0,0)$ containing $(v_0,0)$.  By (3), we may take $\lambda \geq 1$ large enough so that for all $v\in \lambda \cdot v_\sigma'+K$ and $t\geq 1$, $f(t \cdot v, z_0) \leq M$ for some fixed $M>0$. Now that $t\lambda v_\sigma'+v_0$ is on the segment connecting $t\lambda v_\sigma'$ and $t(\lambda v_\sigma'+v_0)$, $f(t\lambda v_\sigma'+v_0, z_0)\leq M$ for all $t \geq 1$ as well.

Let $u_{t\lambda}=(1-t\lambda) v_0+t\lambda v_\sigma$. Then $u_{t\lambda}=t\lambda v_\sigma' + v_0$ and $f(u_{t\lambda}, z_0)\leq M$. However, $v_\sigma= \frac{1}{t\lambda}u_{t\lambda}+(1-\frac{1}{t\lambda})v_0$. By convexity, we have, for all $t\geq 1$,
\begin{IEEEeqnarray*}{rCl}
f(v_\sigma, z_0)&=&f(\frac{1}{t\lambda}u_{t\lambda}+(1-\frac{1}{t\lambda})v_0, z_0)\\
                                 &\leq& \frac{1}{t\lambda} f(u_{t\lambda}, z_0) + (1-\frac{1}{t\lambda})f(v_0, z_0)\\
                                 &\leq&  \frac{1}{t\lambda} M + (1-\frac{1}{t\lambda})(1-\varepsilon).
\end{IEEEeqnarray*}
Let $t \rightarrow \infty$ and we obtain $f(v_\sigma, z_0)\leq 1-\epsilon$. Thus $(v_\sigma, \delta z_0)\in D$ for some $1-\epsilon < \delta <1$. Then $(\delta^{-1}v_\sigma, z_0)\in D$.  

\item We have seen above that $f_{z_0}(\lambda)$ is convex positive. Suppose there exist $\lambda_1 <\lambda_2$ such that $f_{z_0}(\lambda_1)< f_{z_0}(\lambda_2)$, then for every $\lambda>\lambda_2$, $(\lambda, f_{z_0}(\lambda))$ must be above the line passing through $(\lambda_1, f_{z_0}(\lambda_1))$ and $(\lambda_2, f_{z_0}(\lambda_2))$, which has positive slope. This contradicts (3). Therefore $f_{z_0}$ is decreasing. A decreasing convex function on $[1,\infty]$ must be continuous.
\end{enumerate}

We now turn to the proof in Case (B). 
\begin{enumerate}
\item 
           Take $(v,z_0)\in C$. For any $(v_\sigma, 0) \in \sigma^0 \subset C$, we can take $\lambda>0$ such that $\lambda (v_\sigma,0)+(0,z_0) \in C$. Then $(\lambda v_\sigma, z_0) \in C$. 
\item Note that (2) does not depend on $\sigma^0$. 
\item The proof of (3) remains the same. 
\item 
             First by (1) take $(\lambda v_\sigma , z_0) \in C$. Choose $\mu >0$ large enough so that $\mu (v_\sigma, 0)\in D$ and then further $\mu (v_\sigma, 0)+(\lambda v_\sigma , z_0)\in D$ (as $D$ is open convex). This means $((\mu+\lambda) v_\sigma, z_0)\in D$. 

\item Finally, notice that on the plane $\R \cdot v_\sigma+ \R \cdot z_0$, the points of $C$ form an open convex cone containing the positive $x$-axis while the points of $D$ form an open convex subset of this cone. The point $(1,1)$ is in $D$. That $f_{z_0}(\lambda)$ is decreasing still follows from (3). If for some $\lambda $ large enough, $f_{z_0}(\lambda)<\infty$. This implies $f_{z_0}(\lambda) <\infty$ for all $\lambda \geq 1$. If $f_{z_0}(\lambda_1)=-\infty$ for some $\lambda_1$, then for all $\lambda \geq  \lambda_1$, $f_{z_0}(\lambda)=-\infty$ since $f_{z_0}$ is decreasing. By convexity and openness of $D$ we deduce $f_{z_0}(\lambda) =-\infty$ for all $\lambda \geq 1$. Thus the inequality $f_{z_0}((1-t)\cdot \lambda_1+t \cdot \lambda_2 )\leq (1-t) f_{z_0}(\lambda_1)+tf_{z_0}(\lambda_2)$ holds for $0 \leq t\leq 1$, $\lambda_1, \lambda_2 \geq 1$ in all cases. We conclude that $f_{z_0}$ is convex and continuous. 
\end{enumerate}
This finishes the proof of the lemma.
\end{proof}

\subsection{$\mUbar(P_1, \mX_1, p_f)$ on each stratum}
\label{subsec: mUbar on each stratum}

Let $(P,\mX)$ be a mixed Shimura datum with $K_f\subset P(\fadele)$ a neat open compact subgroup and $\Sigma $ a $K_f$-admissible partial polyhedral cone decomposition for $(P,\mX)$ (cf. Definition \ref{subsecdefinition: toroidal compactifications, admissible cone decomposition}). Take any rational boundary component $(P_1, \mX_1)$ of $(P, \mX)$ (associated with a $\Q$-admissible parabolic subgroup $Q\subset P$). Let $\mX^+_{(P_1, \mX_1)}\rightarrowtail \mX_1$ be the $P_1(\R)\cdot Q(\R)^0\cdot U(\C)$-equivariant open embedding as in Proposition \ref{subsecproposition: rational boundary components, rational boundary components} and \ref{subsecproposition: rational boundary components, pink 4.15}. Define $K_f^1=P_1(\fadele)\cap p_f K_f p_f^{-1}$ with $p_f\in P(\fadele)$. This is a neat open compact subgroup of $P_1(\fadele)$.

Given any connected component $\mX_1^0\subset \mX_1$, let $\mathrm{im}\colon \mX^0_1 \rightarrow U_1(\R)(-1)$ be the restriction of the map ``imaginary part" as defined in Subsection \ref{subsec: rational boundary components}. For $p_{1,f} \in P_1(\fadele)$, $\Gamma_1 \coloneq \Stab_{P_1(\Q)}(\mX^0_1)\cap p_{1,f} K_f^1 p_{1,f}^{-1} $ is a neat arithmetic subgroup of $P_1(\Q)$ acting on $\mX_1^{0}$. Since $(P_1, \mX_1)$ is irreducible, $P_1$ acts on $U_1$ by scalars and thus $\Gamma_1$ acts trivially on $U_1$. Therefore $\mathrm{im}$ factors through a map $\imbar \colon \Gamma_1 \backslash \mX_1^0 \rightarrow U_1(\R)(-1)$. (Note that $\Gamma_1 \backslash \mX_1^0$ is a connected component of $M^{K^1_f}(P_1, \mX_1)(\C)$, being the image of $\mX_1^0 \times \{p_{1,f}\}$.) Then $\Gamma_1 \bs \mX^0=\imbar^{-1}(C(\mX^0, P_1))$. 

\begin{subsecremark}
If $K_f$ is not neat, $\Gamma_1$ acts on $U_1(\R)(-1)$ through $\{\pm 1\}$. When $(P_1, \mX_1)$ is a proper boundary component, $C(\mX^0, P_1)$ does not contain the origin and is stabilized by $\Gamma_1$. Thus $\Gamma_1$ still acts trivially on $U_1(\R)(-1)$ and we still have a map $\imbar \colon \Gamma_1 \bs \mX_1^0 \rightarrow U_1(\R)(-1)$. When $(P_1, \mX_1)$ is an improper boundary component, $C(\mX^0, P_1)=U_1(\R)(-1)$ and any core $D\subset C(\mX^0, P_1)$ is just $U_1(\R)(-1)$. In summary, the subsets $\im^{-1}(C(\mX^0, P_1))=\mX^0$ and $\im^{-1}(D)$ are always $\Gamma_1$-invariant.  
\end{subsecremark}

Since $\mX^+_{(P_1, \mX_1)} \rightarrowtail \mX_1$ is bijective on connected components, the open embedding $\mU(P_1, \mX_1, p_f)= P_1(\Q) \backslash \mX^+_{(P_1, \mX_1)} \times P_1(\fadele) /K_f^1  \rightarrowtail M^{K^1_f}(P_1, \mX_1)(\C)=P_1(\Q)\backslash \mX_1\times P_1(\fadele)/K_f^1$ (as defined in Subsection \ref{subsec: toroidal compactifications}) is also bijective on connected components. By equivariance, it follows that $\Gamma_1 \backslash \mX^0 \rightarrowtail \Gamma_1 \backslash \mX^0_1$ is the open embedding of the corresponding connected components, where $\mX^+_{(P_1, \mX_1)} \supset  \mX^0 \rightarrowtail \mX_1^0$. 

We also defined in Subsection \ref{subsec: toroidal compactifications} the cone decomposition $\Sigma_{(P_1, \mX_1, p_f)}=  ( \left . ([\cdot p_f]^*\Sigma)\right |_{(P_1, \mX_1)}  )^0$. It is $K_f^1$-admissible and concentrated on the unipotent fiber. As $(P_1, \mX_1)$ is irreducible, $\Sigma_{(P_1, \mX_1, p_f)}(\mX^0_1, P_1, p_{1,f})$ only depends on the connected component $\mX_1^0\subset \mX_1$. Moreover, note that the projection 
\[
M^{K_f^1}(P_1, \mX_1)(\C) \rightarrow M^{\pi_{U_1}(K_f^1)}((P_1, \mX_1)/U_1)(\C)\]
is bijective on connected components as it  has connected fibers (cf. Subsection \ref{subsec: mixed shimura varieties}). We see that the relative torus embedding 
\[ 
M^{K_f^1}(P_1, \mX_1)(\C) \rightarrowtail M^{K_f^1}(P_1, \mX_1, \Sigma_{(P_1, \mX_1, p_f)})(\C)\]
can be viewed as constructed componentwise.
From now on, we will denote $\Sigma_{(P_1, \mX_1, p_f)}$ by $\Sigma_1^0$ and write $\Gamma_1 \backslash \mX^0_1 \rightarrowtail \left ( \Gamma_1 \backslash \mX^0_1\right )_{\Sigma_1^0}$ for the torus embedding on the connected component $\Gamma_1 \bs \mX^0_1$. For $\sigma \in \Sigma^0_1(\mX^0_1, P_1, p_{1,f})$ (which we just abbreviate as $\sigma \in \Sigma^0_1$), $\left ( \Gamma_1 \backslash \mX^0_1\right )_{\sigma}$ is the embedding relative to $\sigma$ and $\left ( \Gamma_1 \backslash \mX^0_1\right )^\sim_{\sigma} \subset \left ( \Gamma_1 \backslash \mX^0_1\right )_{\sigma}$ is the $\sigma$-stratum on $\left ( \Gamma_1 \backslash \mX^0_1\right )_{\sigma}\subset \left ( \Gamma_1 \backslash \mX^0_1\right )_{\Sigma_1^0}$. Use $\pi_\sigma \colon  \left ( \Gamma_1 \backslash \mX^0_1\right )_{\sigma} \rightarrow  \left ( \Gamma_1 \backslash \mX^0_1\right )^\sim_{\sigma}$ for the projection onto the $\sigma$-stratum. It restricts to a surjective projection $\Gamma_1 \backslash \mX_1^0 \rightarrow \left ( \Gamma_1 \backslash \mX^0_1\right )^\sim_{\sigma}$ and to an isomorphism on $\left ( \Gamma_1 \backslash \mX^0_1\right )^\sim_{\sigma} \subset \left ( \Gamma_1 \backslash \mX^0_1\right )_{\sigma}$.  Recall that $\mUbar(P_1, \mX_1, p_f)$ is defined as the interior of the closure of $\mU(P_1, \mX_1, p_f) $ in $M^{K_f^1}(P_1, \mX_1, \Sigma_{(P_1, \mX_1, p_f)})(\C)$. From above we deduce that $\mUbar(P_1, \mX_1, p_f)$ is also constructed componentwise. Its intersection with $ \left ( \Gamma_1 \backslash \mX^0_1\right )_{\Sigma_1^0}$ is just the interior of the closure of $\Gamma_1 \bs \mX^0$ in $\left ( \Gamma_1\bs \mX^0_1 \right)_\sigma$.

Let $[x_1, p_{1,f}] \in M^{K_f^1}(P_1, \mX_1)(\C) $. We know from Subsection \ref{subsec: mixed shimura varieties} that the fiber over $\pi_{U_1}([x_1, p_{1,f}])$ under $ M^{K_f^1}(P_1, \mX_1)(\C)\rightarrow M^{\pi_{U_1}(K_f^1)}((P_1, \mX_1)/U_1)(\C) $ is isomorphic to the algebraic torus $T(\C)=\Omega_{U_1}\backslash U_1(\C)$, where  $\Omega_{U_1} $ is the image under the map $Z(P_1) \times U_1 \rightarrow U_1$ of
the group 
$
\bigl (  \Id_{Z(P_1)(\Q)}(\mX_1) \times U_1(\Q) \bigr) \cap p_{1,f}K_f^1 p_{1,f}^{-1}
$. 
In addition, for every fiber we may choose a suitable $x_1\in \mX^0_1$ such that the composition $\Omega_{U_1} \backslash U_1(\C) \rightarrowtail \Gamma_1 \backslash \mX_1^0 \overset{\imbar}{\rightarrow} U_1(\R)(-1)$ (where the first arrow is $u_1\mapsto u_1\cdot x_1$) is just the projection below.
\begin{IEEEeqnarray*}{ccrCl}
\Omega_{U_1} \backslash U_1(\C) &=& \left ( \Omega_{U_1} \backslash U_1(\R) \right ) \oplus U_1(\R)(-1)& \rightarrow &U_1(\R)(-1)\\
&&(b,a)&\mapsto& a
\end{IEEEeqnarray*}
The above map coincides with the map $\ord_T$ defined  in Subsection \ref{subsec: torus embeddings}. The induced map $\imbar$ is also just one $\ord$-map as defined there. From now on, we write $\imbar$ as $\ord$.

The next result will be crucial to the proof of our main theorem. (Again we will use $\Int(S)$ to denote the interior of a subset $S$.)

\begin{subseclemma}\label{subseclemma: some lemmas, fundamental lemma}
Let $(P,\mX)$ be a mixed Shimura datum with $K_f\subset P(\fadele)$ an open compact subgroup and $\Sigma$ a $K_f$-admissible partial polyhedral decomposition. Let $(P_1, \mX_1) $ be a rational boundary component and let $\mU(P_1, \mX_1, p_f) \rightarrowtail M^{K_f^1}(P_1, \mX_1)(\C) \rightarrowtail M^{K_f^1}(P_1, \mX_1, \Sigma_1^0)(\C)$ be the open embeddings as introduced above. Fix a connected component $\mX^0\subset \mX$ and $p_{1,f}\in P_1(\fadele)$, let $\Gamma_1 \backslash \mX^0 \rightarrowtail \Gamma_1 \backslash \mX^0_1 \rightarrowtail \left ( \Gamma_1 \backslash \mX^0_1 \right )_{\Sigma^0_1}$ be the corresponding open embeddings. Suppose $D\subset C(\mX^0, P_1)$ is any open convex core and $\sigma \in \Sigma_1^0$. (The interior of) the closure of $\ord^{-1}(D)=\Gamma_1 \bs \im^{-1}(D)$ and $\ord^{-1}(C(\mX^0, P_1))=\Gamma_1 \bs\im^{-1}(C(\mX^0, P_1))=\Gamma_1 \backslash \mX^0$ in $\left ( \Gamma_1 \backslash \mX^0_1 \right )_{\Sigma^0_1}$ meets the $\sigma$-stratum $\left (\Gamma_1 \bs \mX^0_1 \right )^\sim_\sigma$ in the following way. (We denote the closure in $\left (\Gamma_1 \bs \mX^0_1 \right )_{\Sigma_1^0}$ by $\Cl(S)$ and the closure in an $\R$-vector space by $\overbar{S}$.)
\begin{enumerate}
\item $\Int \left (\Cl \left (\ord^{-1} \left (C(\mX^0, P_1) \right ) \right ) \right )\cap \left  (\Gamma_1 \backslash \mX^0_1 \right )_\sigma^\sim 
=\pi_\sigma \left (\ord^{-1} \left (C(\mX^0, P_1)  \right ) \right )$. 

That is, $\Int\left (\Cl \left ( \Gamma_1 \bs \mX^0 \right ) \right )\cap \left (\Gamma_1 \backslash \mX^0_1 \right )_\sigma^\sim =\pi_\sigma \left (\Gamma_1 \bs \mX^0  \right )$.

\item $\Int \left ( \Cl \left (\ord^{-1}(D) \right ) \right )\cap (\Gamma_1 \backslash \mX^0_1)_\sigma^\sim =\pi_\sigma \left (\ord^{-1}(D)  \right )$.

\item $\Cl \left ( \Gamma_1 \bs \mX^0  \right )\cap \left (\Gamma_1 \backslash \mX^0_1 \right )_\sigma^\sim =\pi_\sigma \bigl (\ord^{-1}  \bigl (\overbar{ C(\mX^0, P_1) } \bigr  )  \bigr )$.
\end{enumerate}
In particular, if $-\sigma^0 \subset C(\mX^0, P_1)$, then $\left (\Gamma_1 \bs \mX^0_1 \right )^\sim_\sigma \subset \Int \left (\Cl  \left ( \Gamma_1 \bs \mX^0 \right ) \right ) \subset \mUbar(P_1, \mX_1, p_f)$. In this case even  $\left (\Gamma_1 \bs \mX^0_1 \right )^\sim_\sigma \subset \Int \left (\Cl \left (\ord^{-1}(D) \right )\right )$.
\end{subseclemma}

\begin{subsecremark}
The analogue of (3) does not hold for $D$. That is, it is  in general not true that $\Cl \left (  \ord^{-1}(D) \right) \cap \left (\Gamma_1 \backslash \mX^0_1 \right )_\sigma^\sim =\pi_\sigma \left (\ord^{-1}(\overbar{D})  \right )$.
\end{subsecremark}

\begin{subsecremark}
The above results concerning $C(\mX^0, P_1)$ are also claimed in {\cite[6.13]{Pink}}. 
\end{subsecremark}

Before the proof of Lemma \ref{subseclemma: some lemmas, fundamental lemma}, we reduce it to a simpler case and introduce two more notions.

\begin{subseclemma}\label{subseclemma: some lemmas, reduction to open subsets}
Suppose $f \colon X\rightarrow S$ is an analytic $T(\C)$-torsor (for a $\C$-torus $T$) and $V$ is an open subset of X. Let $\overbar{V}$ be the closure of $V$ in $X_\sigma$ where $\sigma \subset Y_*(T)_\R$ is a convex rational polyhedral cone and $X_\sigma$ the relative torus embedding. Suppose $\{U_i\}$ is an open covering of $S$ (e.g., on which the torsor is trivial). In order to show that $\pi_\sigma(\overbar{V}\cap X^\sim_\sigma)=\pi_\sigma (\overbar{V}\cap X)$ and $\pi_\sigma( \Int \left ( \overbar{V}\right)\cap X^\sim_\sigma)=\pi_\sigma (V)$,  it suffices to prove the two properties for every $f^{-1}(U_i) \rightarrow U_i$.
\end{subseclemma}
\begin{proof}
Let $X_i\coloneq f^{-1}(U_i)$; so $(X_i)_\sigma$ and $(X_i)_\sigma^\sim$ are well-defined. Suppose we are already done over each $U_i$, then
\begin{IEEEeqnarray*}{rCl}
\pi_\sigma\left ( \overbar{V} \cap X^\sim_\sigma \right )&=&\bigcup_{i} \pi_\sigma(\overbar{V} \cap (X_i)_\sigma^\sim)\\
&=&\bigcup_{i} \pi_\sigma(\overbar{V}  \cap (X_i)_\sigma \cap (X_i)_\sigma^\sim)\\
&  &(\overbar{V}  \cap (X_i)_\sigma \text{ is the closure of } V\cap X_i \text{ in } (X_i)_\sigma\\
&=&\bigcup_i \pi_\sigma(\overbar{V} \cap X_i)=\pi_\sigma(\overbar{V} \cap X).
\end{IEEEeqnarray*}
As for any subset $S$ and open subset $U$ in a total space $X$, $\Int^X(S)\cap U = \Int^U(S\cap U)$, we obtain
\begin{IEEEeqnarray*}{+rCl+x*}
\pi_\sigma\left ( \Int^{X_\sigma} \left ( \overbar{V} \right) \cap X^\sim_\sigma \right )&=&\bigcup_{i} \pi_\sigma(\Int^{X_\sigma} \left ( \overbar{V} \right)   \cap (X_i)_\sigma \cap (X_i)_\sigma^\sim)\\
&=&\bigcup_{i} \pi_\sigma( \Int^{(X_i)_\sigma} \left ( \overbar{V} \cap (X_i)_\sigma \right ) \cap (X_i)_\sigma^\sim )\\
&=&\bigcup_i \pi_\sigma( V \cap X_i)=\pi_\sigma(V). &\qedhere
\end{IEEEeqnarray*}
\end{proof}

Next, suppose we have a diagram 
\begin{equation*}
\begin{tikzcd}
X\arrow[d,"g"'] \arrow[r, "f"] & X' \arrow[d, "{g'}"]\\
Y \arrow[r, "f ' "]                         &  Y'
\end{tikzcd}
\end{equation*}
in the category of topological spaces. We say $f$ is fiberwise surjective if for all $y \in Y$, $g^{-1}(y) \rightarrow  (g'  )^{-1}(f'(y))$ is surjective. In this situation, it is easy to verify that for any $U' \subset X'$, $g (f^{-1}(U')) = ( f' )^{-1}  (g'(U') )$.

Given a product space $\R^n \times \R^m$, a subset $S\subset \R^n \times \R^m$ and an open convex cone $\sigma^0 \subset \R^n \times \{0\}$, we say that a point $(*, z_0)\in \{*\}\times \R^m$ can be \emph{$\sigma^0$-approached within $S$} if there is a sequence $\{(v_n, z_n)\}^\infty_{n=1} \subset (\sigma \times \R^m)\cap S$ such that $z_n \rightarrow z_0$ and $|v_n|\rightarrow \infty$ as $n \rightarrow \infty$. Basically, the sequence within $S$ converges to $z_0$ in the second coordinate and goes to infinity in the first coordinate along $\sigma^0$. 

\begin{proof}[Proof of Lemma \ref{subseclemma: some lemmas, fundamental lemma}]~

First assume that $K_f$ is neat. Then we have the map $\imbar=\ord \colon \Gamma_1 \bs \mX^0_1 \rightarrow U_1(\R)(-1)$ while $\Gamma_1 \bs\im^{-1} \left (C(\mX^0, P_1) \right ) =\ord^{-1}(C(\mX^0, P_1))$ and $\Gamma_1 \bs\im^{-1} \left (D \right ) =\ord^{-1}(D)$. Let $T(\C)=\Omega_{U_1}\bs U_1(\C)$ be as defined before Lemma \ref{subseclemma: some lemmas, fundamental lemma} and $K=\Omega_{U_1} \bs U_1(\R)$ be its maximal compact subgroup. Consider the analytic $T(\C)$-torsor $\Gamma_1 \bs \mX^0_1 \rightarrow \overbar{\Gamma}_1 \bs (\mX^0_1/U_1)$ which we denote by $\pi \colon X\rightarrow Y$ and where the group $\overbar{\Gamma}_1\coloneq \Stab_{(P_1/U_1)(\Q)}(\mX^0_1/U_1)\cap \pi_{U_1}(p_{1,f}K_f^1p^{-1}_{1,f})$. In fact, this is the canonical projection $M^{K_f^1}(P_1, \mX_1)(\C) \rightarrow M^{\pi_{U_1}(K_f^1)}(P_1/U_1, \mX_1/U_1)(\C)$ restricted to the connected component $\Gamma_1 \bs \mX^0_1$.

By definition, the map $\ord$ is the composition $X \rightarrow X/K \overset{\sim}{\rightarrow}  (T(\C)/K)\times Y  \rightarrow T(\C)/K \overset{\sim}{\rightarrow} Y_*(T)_\R=U_1(\R)(-1)$. Since $X/K$ as a $T/K$-torsor can be globally trivialized over $Y$, we see that it is possible to choose a local section $\mathcal{N} \rightarrow X$ over some open subset $\mathcal{N} \subset Y$ such that under $T(\C)\times X \rightarrow X$, the subset $\ord_T^{-1}(C(\mX^0, P_1)) \times \mathcal{N}$ maps isomorphically onto $(\ord|_{X})^{-1}(C(\mX^0, P_1)) \cap \pi^{-1}(\mathcal{N})$ and similarly for the core $D$.  Hence by Lemma \ref{subseclemma: some lemmas, reduction to open subsets} we just need to prove the parallel statements for $\ord_T^{-1}(C(\mX^0, P_1)) \times \mathcal{N} \subset T(\C) \times \mathcal{N}$ with respect to the torus embedding $T(\C)_\sigma \times \mathcal{N}$ for some convex rational polyhedral cone $\sigma \subset Y_*(T)_\R$. (Recall that $(\Gamma_1 \bs \mX^0_1)_{\Sigma_1^0}$ is also constructed locally.) This reduces us to the case of a single $\ord_T^{-1}(C(\mX^0, P_1)) \subset T(\C) \subset T(\C)_\sigma$.  

Let $T' \subset T$ be the subtorus such that $\R\cdot \sigma =Y_*(T')_\R$. Note that $T_\sigma^\sim \cong \Spec k[X^*(T)_{\sigma=0}]$. Thus we have a split short exact sequence $1 \rightarrow T' \rightarrow T \rightarrow T_\sigma^\sim \rightarrow 1$. After choosing a splitting, $T' \rightarrow T \rightarrow T_\sigma \overset{\pi_\sigma}{\rightarrow} T_\sigma^\sim$ may be identified with $T' \rightarrow T' \times T_\sigma^\sim \rightarrow T'_\sigma \times T_\sigma^\sim \rightarrow T_\sigma^\sim$.

Let $K' \subset T'$ (resp. $K_\sigma^\sim \subset T_\sigma^\sim$) be the maximal compact subgroup. We have a diagram as below, where we write $T$ for $T(\C)$ to save space.
\begin{equation*}
\begin{tikzcd}
T
\arrow[r, "\sim"]
\arrow[d, rightarrowtail]
&
T'\times T_\sigma^\sim 
\arrow[r, twoheadrightarrow, "q", "{\mathrm{f.s.o.}}" ']
\arrow[d, rightarrowtail] 
&
T'/K' \times T_\sigma^\sim / K_\sigma^\sim
\arrow[d, rightarrowtail ]
\arrow[r, "\sim"]
&
 Y_*(T')_\R \times Y_*(T_\sigma^\sim)_\R
\arrow[dd, twoheadrightarrow]
\\
T_\sigma
\arrow[r, "\sim"]
\arrow[d, twoheadrightarrow,"{\pi_\sigma}" '] 
&
T'_\sigma \times T_\sigma^\sim 
\arrow[r, twoheadrightarrow, "{q_\sigma}", "{\mathrm{f.s.o.}}" ']
\arrow[d, twoheadrightarrow]
& 
T'_\sigma /K' \times T_\sigma^\sim / K_\sigma^\sim
\arrow[d, twoheadrightarrow, "j_\sigma"] 
& \\
T_\sigma^\sim 
\arrow[r, "="]
&
T_\sigma^\sim \arrow[r,twoheadrightarrow]
&
T_\sigma^\sim / K_\sigma^\sim
\arrow[r, "\sim"]
&
Y_*(T_\sigma^\sim)_\R
\end{tikzcd}
\end{equation*}
Here ``$\mathrm{f.s.o.}$'' means ``fiberwise surjective and open" with respect to the map $T_\sigma^\sim \rightarrow T_\sigma^\sim/K_\sigma^\sim$.  Under the splitting, $T_\sigma^\sim$ sits in $T_\sigma$ as shown below, where $\{*\}$ is the image of $(T')_\sigma^\sim$. 
\begin{equation*}
\begin{tikzcd}
T_\sigma^\sim
\arrow[r, "\sim"]
\arrow[d, rightarrowtail]
\arrow[dd, "\wr" ', xshift=-10pt]
&
(T')_\sigma^\sim \times T_\sigma^\sim 
\arrow[r, twoheadrightarrow]
\arrow[d, rightarrowtail] 
&
\{*\} \times T_\sigma^\sim / K_\sigma^\sim 
\arrow[d, rightarrowtail]
\arrow[dd, "\wr" , xshift=40pt]
\\
T_\sigma
\arrow[r, "\sim" ]
\arrow[d, ,twoheadrightarrow, "{\pi_\sigma}"]
&T'_\sigma \times T_\sigma^\sim 
\arrow[r, twoheadrightarrow, "{q_\sigma}", "{\mathrm{f.s.o.}}" ']
\arrow[d, twoheadrightarrow]
& 
T'_\sigma /K' \times T_\sigma^\sim / K_\sigma^\sim
\arrow[d, twoheadrightarrow, "j_\sigma"]\\
T_\sigma^\sim \arrow[r, "="]
&
T_\sigma^\sim \arrow[r,twoheadrightarrow]
&
T_\sigma^\sim / K_\sigma^\sim
\end{tikzcd}
\end{equation*}
For an open map $q \colon X\rightarrow Y$, $q^{-1}(\Cl^Y(S))=\Cl^X(q^{-1}(S))$ and $q^{-1}(\Int^Y(S))=\Int^X(q^{-1}(S))$. Let $A\coloneq \{*\} \times T_\sigma^\sim/K_\sigma^\sim$. A little diagram chasing shows that we just have to prove both of the following statements
\begin{IEEEeqnarray*}{rCl}
\Int(\Cl(S)) \cap A
&=& j_\sigma(S),\\
\Cl(S)\cap A  
&=&j_\sigma(\overbar{S}),
\end{IEEEeqnarray*}
for $S=C(\mX^0, P_1)$, and only the first statement for $S=D$ (an open convex core), as a subset of $T'/K' \times T_\sigma^\sim /K_\sigma^\sim \cong  Y_*(T')_\R \times Y_*(T_\sigma^\sim)_\R$. 
We will write $\Cl(S)$ for the closure in $T'_\sigma /K' \times T_\sigma^\sim /K_\sigma^\sim$ and $\overbar{S}$ for the closure in the Euclidean space $T' /K' \times T_\sigma^\sim /K_\sigma^\sim \cong  Y_*(T')_\R \times Y_*(T_\sigma^\sim)_\R$. Note $\overbar{S}=\Cl(S) \cap (T' /K' \times T_\sigma^\sim /K_\sigma^\sim)$.

We consider how points $\{(\overbar{t_n}, z_n)\}^\infty_{n=1} \subset T'/K' \times T_\sigma^\sim /K_\sigma^\sim $ approach a point $(*, z_0) \in \{*\}\times T_\sigma^\sim /K_\sigma^\sim=A$. Certainly, the second coordinate must converge: $z_n \rightarrow z_0$. For the first coordinate, choose $\{t_n\}_{n=1}^\infty \subset T'$ lifting $\overbar{t_n}$. Note that under the open map $T_\sigma '\rightarrow T_\sigma' /K'$, the fiber over $\infty$ is the singleton $(T')_\sigma^\sim$, as every $\sigma$-stratum is a $T'$-orbit and stable under $K'$. Since $K'$ is compact and $T_\sigma'$ is Hausdorff, it follows that any such lifting $\{t_n\}_{n=1}^\infty$ must converge to the point $(T')_\sigma^\sim$. Suppose $\chi_1, \ldots, \chi_r$ is a set of generators for $X^*(T')_{\sigma\geq 0}$ as a monoid. As $T_\sigma' =\Spec k[X^*(T')_{\sigma \geq 0}]$, $(\chi_1, \ldots, \chi_r)$ defines a closed embedding $T_\sigma' \rightarrowtail \mathbb{A}^r_\C$ sending $(T')^\sim_\sigma$ to the origin. Thus $t_n \rightarrow (T')_\sigma^\sim$ if and only if $|\chi_i(t_n)| \rightarrow 0$ for all $1\leq i \leq r$. 

Recall that the map $\ord$ is defined as $T(\C)\cong Y_*(T)\otimes \C^\times \rightarrow Y_*(T)\otimes \R$, $\lambda\otimes z \mapsto \lambda \otimes (\ln|z|)$. Thus in terms of the splitting, it is $T' \times T_\sigma^\sim \rightarrow T'/K' \times T_\sigma^\sim/K_\sigma^\sim \overset{\sim}{\rightarrow} Y_*(T')_\R \times Y_*(T_\sigma^\sim)_\R$. Let $\lambda_1, \ldots, \lambda_l$ be a basis of $Y_*(T')$. Given $t\in T'(\C)$, suppose its image in $Y_*(T) \otimes \C^\times$ is $\sum_{j=1}^l \lambda_j\otimes t_j$, then $\ord(t)=\sum_{j=1}^l \lambda_j\otimes (\ln |t_j|)$. 
\begin{IEEEeqnarray*}{rCl}
\chi_i (\ord(t))&=&\sum_{j=1}^l \langle\chi_i, \lambda_j \rangle \ln|t_j|\\
                            &=& \ln|\prod_{j=1}^l t_j^{\langle \chi_i, \lambda_j \rangle}|=\ln|\chi_i(t)|
\end{IEEEeqnarray*}
Thus $\overbar{t_n} \rightarrow *$ if and only if any lifting $t_n \rightarrow (T')_\sigma^\sim$ if and only if $|\chi_i(t_n)| \rightarrow 0$ if and only if $\chi_i (\ord (t_n) )\rightarrow  -\infty \,(1\leq i\leq r)$. Denote $\ord(t_n)$ by $v_n (\in Y_*(T')_\R)$. This means $\{v_n\}$ goes to $\infty$ along $-\sigma^0$.  In summary, 
\begin{IEEEeqnarray*}{cl}
&\text{points in }\{*\}\times T_\sigma^\sim/K_\sigma^\sim\\
&\text{that can be approached by } S \subset T'/K' \times T_\sigma^\sim/K_\sigma^\sim \\
\iff& \text{points in }\{*\}\times Y_*(T_\sigma^\sim)_\R \\
      &\text{that can be }(-\sigma^0)\text{-approached within } S \subset Y_*(T')_\R \times Y_*(T_\sigma^\sim)_\R.
\end{IEEEeqnarray*}

Our $S$ in $ Y_*(T')_\R \times Y_*(T_\sigma^\sim)_\R$ is $C(\mX^0, P_1)$ or $D$.  But $\sigma \in \Sigma_1^0(\mX_1^0, P_1, p_{1,f})=\Sigma(\mX^0, P_1, p_{1,f}\cdot p_f)p_f^{-1}$, which is a complete cone decomposition of $-C(\mX^0, P_1)$. And so $-\sigma \subset C^*(\mX^0, P_1)\subset \overbar{C(\mX^0, P_1)}$. By Proposition \ref{subsecproposition: rational boundary components, properties of the cone}, there is a unique rational boundary component $(P_2, \mX_2) \geq (P_1, \mX_1)$ such that $-\sigma^0 \subset C(\mX^0, P_2)(\subset U_2(\R)(-1))$ and so $-\sigma\subset U_2(\R)(-1)$. If $(P_2, \mX_2) \neq (P_1, \mX_1)$, we note that $C^*(\mX^0, P_1) \cap U_2(\R)(-1)=C^*(\mX^0, P_2)$ by Proposition \ref{subsecproposition: rational boundary components, properties of the star cone}. Hence 
\begin{IEEEeqnarray*}{rCl}
C(\mX^0, P_1) \cap (-\sigma) &=& C(\mX^0, P_1)\cap C^*(\mX^0, P_1) \cap (-\sigma)\\
&\subset& C(\mX^0, P_1)\cap C^*(\mX^0, P_1) \cap U_2(\R)(-1)\\
&=&C(\mX^0, P_1)\cap C^*(\mX^0, P_2)=\varnothing,
\end{IEEEeqnarray*} and we are in Case (A) of Lemma \ref{subseclemma: some lemmas, analytic geometric lemma} ($-\sigma^0$ contained in the interior of $\overbar{C(\mX^0, P_1)}\cap Y_*(T')_\R$). If $(P_2, \mX_2)=(P_1, \mX_1)$, $-\sigma^0 \subset C(\mX^0, P_1)$ and we are in Case (B) of Lemma \ref{subseclemma: some lemmas, analytic geometric lemma}. 

Let $C\coloneq C(\mX^0, P_1)$. With the above analysis,  we claim that the set  $E\coloneq \Cl(C)\cap A$ (of points in $A=\{*\} \times Y_*(T_\sigma^\sim)_\R$ that can be $(-\sigma^0)$-approached within $C(\mX^0, P_1) \subset Y_*(T')_\R \times Y_*(T_\sigma^\sim)_\R$) is exactly $\overbar{j_\sigma (C)}$. Namely, $E\subset \overbar{j_\sigma(C)}$ by definition of $(-\sigma^0)$-approaching and $E \supset \overbar{j_\sigma(C)}$ by (1) and (2) of Lemma \ref{subseclemma: some lemmas, analytic geometric lemma}. By Lemma \ref{subseclemma: some lemmas, projection of open cones},  $E =\overbar{j_\sigma (C)}=j_\sigma(\overbar{C})$ and we have done (3). 

In particular, $j_\sigma(C) \subset \Cl(C)$. Thus we have shown, for all $\sigma \in \Sigma_1^0$, that $\pi_\sigma(\ord^{-1}(C)) \subset \Cl (\ord^{-1}(C))$.

Points in $F\coloneq  \Int (\Cl(C)) \cap A $ should lie in the interior 
\[
 \Int^A( \Cl(C) \cap A )=\Int^A(j_\sigma(\overbar{C}))=\Int^A(\overbar{j_\sigma(C)})=j_\sigma(C),
\] 
 by Lemma \ref{subseclemma: some lemmas, projection of open cones} and so $F\subset j_\sigma(C)$.

 Let $\tau \leq \sigma$ be any face of $\sigma$. Then $-\tau\subset \overbar{C}$ and so $-\tau+C=-\tau^0+C=C$. By {\cite[5.9]{Pink}},  $\bigcup_{\tau \leq \sigma} \pi_\tau(\ord^{-1}(C))$ is an open subset, contained in $\Cl(\ord^{-1}(C))$. This implies $\pi_\sigma(\ord^{-1}(C)) \subset \Int(\Cl (\ord^{-1}(C)))$, or equivalently, $j_\sigma(C) \subset \Int(\Cl(C))$. Thus $j_\sigma(C)\subset \Int (\Cl(C))\cap A=F$ and $F=j_\sigma(C)$ as desired. We have done (1).

In the case of an open convex core $D$, a similar argument using (4) and (5) of Lemma \ref{subseclemma: some lemmas, analytic geometric lemma} gives $ \Cl(D)\cap A=\overbar{j_\sigma(D)} $ and in particular $j_\sigma(D)\subset \Cl(D)$. By {\cite[5.9]{Pink}} again $j_\sigma(D) \subset \Int(\Cl(D))$. 
Thus 
\begin{IEEEeqnarray*}{rCl}
j_\sigma(D)&\subset& \Int(\Cl(D))\cap A\\
&\subset &\Int^A  \bigl (  \Cl(D)\cap A  \bigr  )\\
&=&\Int^A (\overbar{j_\sigma(D)})=j_\sigma(D),
\end{IEEEeqnarray*}
by Lemma \ref{subseclemma: some lemmas, projection of open cones}. We have done (2).

If $-\sigma^0 \subset C(\mX^0, P_1)$, then $j_\sigma(D)=j_\sigma(C)=T_\sigma^\sim /K_\sigma^\sim$. In other words, 
\[
\Int( \Cl (D) ) \cap T_\sigma^\sim=\Int( \Cl (C) ) \cap T_\sigma^\sim=T_\sigma^\sim.
\]
This proves the last statement for a neat $K_f$.

When $K_f$ is not assumed neat, take a neat open compact normal subgroup $K_{f,n} \unlhd K_f$. By the $K_f$-admissibility of $\Sigma$, the action of the finite group $K^1_f /K^1_{f,n}$ on $M^{K^1_{f,n}}(P_1,\mX_1)(\C)$ 
(over $M^{\pi_{U_1}(K^1_{f,n})}(P_1/U_1,\mX_1/U_1)(\C)$) 
extends to the relative torus embedding $M^{K^1_{f,n}}(P_1,\mX_1, \Sigma_1^0)(\C)$. This action preserves the $\sigma$-stratum and is compatible with $\pi_\sigma$. Therefore all the assertions follow from taking quotients. This finishes the proof of Lemma \ref{subseclemma: some lemmas, fundamental lemma}.
\end{proof}

We will also need the following result from \cite{BrosnanFakhruddin}.

\begin{subseclemma}[{\cite[Lemma 49]{BrosnanFakhruddin}}]\label{subseclemma: some lemmas, polydisc intersection}

Let $U$ be a real vector space with $C\subset U$ an open convex cone and $\Gamma \subset U$ a lattice. Let $p\colon U_\C \rightarrow \Gamma \bs U_\C  \eqcolon T(\C)
$ be the quotient map. ($\Gamma \bs U_\C =(\Gamma \bs U) \oplus (i\cdot U)$.) Suppose $\sigma \subset C$ is a smooth rational polyhedral convex cone of top dimension. Let $T \rightarrow T_\sigma$ be the torus embedding associated with $\sigma$. Given $c\in U$, define $D_c \coloneq U +   i\cdot (c+C) $. Then for sufficiently small polydisc $D$ centered at $T_\sigma^\sim$, $D^\circ\coloneq D \cap p(D_c)$ is a product of punctured discs. In particular, this applies to $D^\circ=D\cap p(U+iC)$.
\end{subseclemma}


         \subsection{Proof of Theorem \ref{sectheorem: introduction, main theorem}} \label{subsec: proof of main theorem}

We now turn to the proof of Theorem \ref{sectheorem: introduction, main theorem}.

Let $(P, \mX)$ be a mixed Shimura datum. Let $\Sigma$ be a $K_f$-admissible complete polyhedral cone decomposition such that $M^{K_f}(P, \mX, \Sigma)(\C)$ is a projective variety as in Theorem \ref{subsectheorem: toroidal compactifications, projective structure}. Consider the rational boundary component $(P_1, \mX_1)$ related to a $\Q$-admissible parabolic subgroup $Q\subset P$ together with the $\Stab_{Q(\R)}(\mX_1)\cdot U(\C)$-equivariant open embedding $\mX^+_{(P_1,\mX_1)}\rightarrow \mX_1$. Suppose $\mX^0_1 \subset \mX_1$ is the connected component corresponding to $\mX^0\subset \mX^+_{(P_1,\mX_1)}$ and $C(\mX^0, P_1) \subset U_1(\R)(-1)$ is the open convex cone as in Proposition \ref{subsecproposition: rational boundary components, pink 4.15}. An element of $P_1(\Q)$ stabilizes $\mX^0$ if and only if it stabilizes $\mX^0_1$. We defined the following groups in the introduction.
\begin{IEEEeqnarray*}{rCl}
\Gamma           &=&\Stab_{P(\Q)}(\mathcal{X}^0)\cap p_f K_fp_f^{-1}\\
\Gamma_Z      &=& \left(  \Id_{Z(P)(\Q)}(\mathcal{X})\right)\cap  p_f K_fp_f^{-1}\\
\Delta                &=&\Gamma/\Gamma_Z\\
\Gamma_{ZU_1}&=&\left ( \Id_{Z(P)(\Q)} (\mathcal{X}) \times U_1(\Q)\right ) \cap p_f K_fp_f^{-1}\\
\Gamma_{U_1}&=&\frac{\Gamma_{ZU_1}}{\Gamma_Z} \subset \frac{\Gamma}{\Gamma_Z}\\
&&\bigl ( \Gamma_{U_1}=\pr_{U_1}(\Gamma_{Z U_1}) \text{ for }\pr_{U_1}\colon Z(P) \times U_1 \rightarrow U_1 \bigr)
\end{IEEEeqnarray*}
We need a few more analogous groups.
\begin{IEEEeqnarray*}{rCl}
\Gamma_1      &\coloneq&\Stab_{P_1(\Q)}(\mX^0)\cap p_fK_fp^{-1}_f \\
                              &  =          & \Gamma \cap P_1(\Q) \subset \Gamma\\
\Lambda_{ZU_1}&\coloneq&\bigl ( (\Id_{Z(P)(\Q)}(\mX)\cap P_1(\Q))\times U_1(\Q) \bigr )\cap p_fK_fp^{-1}_f\\
                            & =&\Gamma_{ZU_1}\cap P_1(\Q)=\Gamma_{ZU_1}\cap \Gamma_1\\
\Lambda_{U_1} &\coloneq& \pr_{U_1}(\Lambda_{ZU_1}) \text{ for }\pr_{U_1}\colon Z(P) \times U_1 \rightarrow U_1\\
\Omega_{Z_1 U_1}&\coloneq&(\Id_{Z(P_1)(\Q)}(\mX_1)\times U_1(\Q) )\cap p_fK_fp_f^{-1}\\
\Omega_{U_1}&\coloneq& \pr_{U_1}(\Omega_{Z_1 U_1}) \text{ for }\pr_{U_1}\colon Z(P_1) \times U_1 \rightarrow U_1
\end{IEEEeqnarray*}
Note that  $\Gamma_{ZU_1} \supset \Lambda_{ZU_1} \subset \Omega_{Z_1U_1}$ as $\Id_{Z(P)(\Q)}(\mX)\cap P_1(\Q)\subset \Id_{Z(P_1)(\Q)}(\mX_1)$ by Proposition \ref{subsecproposition: rational boundary components, rational boundary components}. As a result, we have $\Gamma_{U_1}\supset \Lambda_{U_1} \subset \Omega_{U_1}$.

Let $K^1_f \coloneq P_1(\fadele)\cap p_f K_fp_f^{-1}$. Recall that 
\begin{IEEEeqnarray*}{rCl}
\mU(P_1, \mX_1, p_f)& \coloneq &P_1(\Q)\bs \mX^+_{(P_1, \mX_1)} \times P_1(\fadele) / K^1_f,\\
\Sigma_1^0 &\coloneq& \left (\left( [\cdot p_f]^*\Sigma\right)|_{(P_1, \mX_1)}\right)^0,
\end{IEEEeqnarray*}
 and $ \mUbar(P_1, \mX_1 , p_f)\coloneq \Int(\Cl(\mU(P_1, \mX_1, p_f)))$ is the interior of the closure of $\mU(P_1, \mX_1, p_f)$ in $M^{K_f^1}(P_1, \mX_1 ,\Sigma_1^0)(\C)$.

\begin{proof}[Proof of Theorem \ref{sectheorem: introduction, main theorem}]~

\begin{enumerate}
\item By definition, $\Gamma_Z$ acts trivially on $\mX^0$. Thus $\Gamma \bs \mX^0$ is actually $(\Gamma/\Gamma_Z) \bs \mX^0$. Further, let $\gamma \in \Gamma$ be an element fixing a point $x_0 \in \mX^0$. Then by Lemma \ref{subseclemma: mixed shimura varieties, centralizer lemma}, $\gamma$ lies in $Z(P)(\Q)$ and hence fixes every point $x\in \mX$. That is, $\gamma \in \Id_{Z(P)(\Q)}(\mX)\cap \Gamma =\Gamma_Z$. We have shown that $\Delta=\Gamma /\Gamma_Z$ acts freely on $\mX^0$. By Proposition \ref{subsecproposition: mixed shimura varieties, topology} and Theorem \ref{subsectheorem: toroidal compactifications, quasi-projective structure}, $\Gamma \bs \mX^0$ and $\Gamma' \bs \mX^0$ are (connected) smooth quasi-projective varieties over $\C$. Since $\Delta'=\Gamma' / \Gamma_Z' \unlhd\Gamma / \Gamma_Z=\Delta$, we see that $\Delta' \bs \mX^0 \rightarrow \Delta \bs \mX^0$ is a $\Delta/\Delta'$-torsor. 

\item 
Let $\delta_Q\coloneq \Stab_{Q(\Q)}(\mX_1) \cap \bigl ( P_1(\fadele) \cdot p_f  K_f p_f^{-1} \bigr)$ and $\Delta_Q \coloneq \delta_Q /P_1(\Q)$. By Remark \ref{subsecremark: toroidal compactifications, action of the normalizer}, the map $\mUbar(P_1, \mX_1, p_f) \rightarrow  M^{K_f }(P, \mX,\Sigma )(\C)$ factors through $\Delta_Q \bs \mUbar(P_1, \mX_1, p_f) $ (resp. $\mU(P_1, \mX_1, p_f) \rightarrow  M^{K_f }(P , \mX)(\C)$ factors through $\Delta_Q \bs \mU(P_1, \mX_1, p_f)$). The above objects fit into a diagram. 
\begin{equation*}\adjustbox{center}{%
\begin{tikzcd}[column sep=small]
\mU(P_1, \mX_1, p_f) \arrow[d, rightarrowtail] \arrow[dr,rightarrowtail] \arrow[r, twoheadrightarrow]
& \Delta_Q \bs \mU(P_1, \mX_1, p_f) \arrow[rr, "{[\cdot p_f]}"] \arrow[dr, rightarrowtail]&
&M^{K_f}(P, \mX)(\C) \arrow[d, rightarrowtail]\\
M^{K_f^1}(P_1, \mX_1)(\C) \arrow[d, rightarrowtail] 
& \mUbar(P_1, \mX_1 , p_f)  \arrow[r, twoheadrightarrow] \arrow[dl, rightarrowtail]
&  \Delta_Q \bs \mUbar(P_1, \mX_1, p_f)  \arrow[r]
&M^{K_f}(P, \mX, \Sigma)(\C)\\
M^{K_f^1}(P_1, \mX_1 ,\Sigma_1^0)(\C)
&&&
\end{tikzcd}
}
\end{equation*}

In terms of connected components, the above diagram translates into a diagram as shown below.
\begin{equation*}
\adjustbox{center}{%
\begin{tikzcd}[column sep=small ]
\Gamma_1 \bs \mX^0 \arrow[d, rightarrowtail] \arrow[dr,rightarrowtail] \arrow[r, twoheadrightarrow]
& \Gamma_Q \bs \mX^0= \frac{\Gamma_Q}{\Gamma_1} \left \bs (\Gamma_1 \bs \mX^0) \right. \arrow[rr,twoheadrightarrow] \arrow[dr, rightarrowtail]&
& \Gamma \bs \mX^0 \arrow[d, rightarrowtail]\\
\Gamma_1 \bs \mX^0_1 \arrow[d, rightarrowtail] 
& (\Gamma_1 \bs \mX^0)_{\Sigma_1^0}  \arrow[r, twoheadrightarrow] \arrow[dl, rightarrowtail]
&  \frac{\Gamma_Q}{\Gamma_1} \left \bs (\Gamma_1 \bs \mX^0)_{\Sigma_1^0} \right.  \arrow[r]
&(\Gamma \bs \mX^0)_{\Sigma}\\
(\Gamma_1 \bs \mX_1^0)_{\Sigma^0_1}
&&&
\end{tikzcd}
}
\end{equation*}
Here $\Gamma_Q= \Stab_{Q(\Q)}(\mX^0) \cap p_f K_f p_f^{-1} =Q(\Q) \cap \Gamma$, and $(\Gamma_1 \bs \mX^0)_{\Sigma_1^0} \subset  \mUbar(P_1, \mX_1 , p_f)$ is the interior of the closure of $\Gamma_1 \bs \mX^0 \subset (\Gamma_1 \bs \mX_1^0)_{\Sigma^0_1}$. 

Now choose a point $x_0 \in \mX^0$. The $U_1(\C)$-orbit map $U_1(\C) \rightarrow \mX^0_1$, $u_1 \mapsto u_1 \cdot x_0$ descends to a map $\Omega_{U_1}\bs U_1(\C) \rightarrow \Gamma_1 \bs \mX^0_1$. According to Subsections \ref{subsec: mixed shimura varieties} and \ref{subsec: toroidal compactifications}, this is the inclusion of a closed fiber into the analytic $\Omega_{U_1}\bs U_1(\C)$-torsor $M^{K_f^1}(P_1, \mX_1)(\C) \rightarrow M^{\pi_{U_1}(K_f^1)}(P_1/U_1, \mX_1/U_1)(\C)$, relative to which the torus embedding $\Gamma_1 \bs \mX^0_1 \rightarrowtail (\Gamma_1 \bs \mX^0_1)_{\Sigma_1^0}$ is constructed. 

We claim that under this orbit map, $U_1(\R)+C(\mX^0, P_1)$ is mapped into $\mX^0$. Indeed, take $u_1\cdot c \in U_1(\R)+C(\mX^0, P_1)$ and let $\im \colon \mX^0_1 \rightarrow U_1(\R)(-1)$ be the map of imaginary part. Then $\im(u_1\cdot c\cdot x_0)=(u_1\cdot c)\cdot \im(x_0)= c + \im(x_0)\in C(\mX^0, P_1) + C(\mX^0, P_1) \subset C(\mX^0, P_1)$ as $C(\mX^0, P_1) $ is an open convex cone. Denote $C(\mX^0, P_1)$ by $C$. It follows that we have an induced map $\Omega_{U_1} \bs (U_1(\R)+C) \rightarrowtail \Gamma_1 \bs \mX^0$, compatible with $\Omega_{U_1}\bs U_1(\C) \rightarrowtail \Gamma_1 \bs \mX^0_1$. 

By assumption, $\Sigma_1^0(\mX^0_1, P_1, 1)=\Sigma(\mX^0, P_1, p_f)p_f^{-1}$ is a complete rational polyhedral cone decomposition of $-C^*(\mX^0, P_1)\subset U_1(\R)(-1)$. In particular, there exists a $\sigma \in \Sigma_1^0$ of top dimension such that $-\sigma^0 \subset C(\mX^0, P_1)$. Then $(\Omega_{U_1} \bs U_1(\C))_\sigma \rightarrowtail (\Gamma_1 \bs \mX^0_1)_{\sigma}\subset (\Gamma_1 \bs \mX^0_1)_{\Sigma_1^0}$. 

Define $T\coloneq \Omega_{U_1}\bs U_1(\C)$ and $T'\coloneq \Lambda_{U_1}\bs U_1(\C)$. Note that $T$ and $T'$ are algebraic tori and there are isogenies $T' \rightarrow T$ and $T'\rightarrow \Gamma_{U_1} \bs U_1(\C)$. We have maps $\Lambda_{U_1} \bs (U_1(\R)+C) \rightarrow \Omega_{U_1} \bs (U_1(\R)+C) \rightarrowtail \Gamma_1 \bs \mX^0$.

By the last assertion of Lemma \ref{subseclemma: some lemmas, fundamental lemma}, $(\Gamma_1 \bs \mX^0)_{\Sigma_1^0}$ is an open neighborhood of the entire $\sigma$-stratum $(\Gamma_1 \bs \mX^0_1)^\sim_{\sigma}$. In particular, it is a neighborhood of the point $T_\sigma^\sim$. Therefore, a neighborhood $D \subset T_\sigma$ of this point will map into $(\Gamma_1 \bs \mX^0)_{\Sigma_1^0} $, and then into $ (\Gamma \bs \mX^0)_\Sigma $, as shown above.

We take $\sigma_s \subset \sigma $ to be a rational convex polyhedral cone, smooth with respect to the lattice $\Gamma_{U_1}$. The composite morphism $T'_{\sigma_s}\rightarrow T_{\sigma_s} \rightarrow T_\sigma$ maps $(T')_{\sigma_s}^\sim$ to $T_{\sigma}^\sim$. Thus a neighborhood of $(T')_{\sigma_s}^\sim$ will also map into $(\Gamma \bs \mX^0)_\Sigma$.

However, $\Id_{Z(P)(\Q)}(\mX) \times U_1(\Q) \subset \Stab_{Q(\Q)}(\mX^0)$, so $\Gamma_{ZU_1} \subset \Gamma_Q$ and $\Gamma_{ZU_1}/\Lambda_{Z U_1} \subset  \Gamma_Q /\Gamma_1$. It follows that the maps $(T')_{\sigma_s} \rightarrow  (\Gamma_1 \bs \mX^0_1)_{\Sigma_1^0}$, $T'\rightarrow \Gamma_1 \bs \mX^0_1$, and $\Lambda_{U_1} \bs (U_1(\R)+C) \rightarrow \Gamma_1 \bs \mX^0$ all descend to the corresponding quotients, inducing the following diagram. 

\begin{equation*}\hspace{2.5pt}
\adjustbox{center}{%
\begin{tikzcd}[column sep= small]
D^\circ \arrow[r, rightarrowtail]
&
\Gamma_{U_1} \bs (U_1(\R)+C)
\arrow[r] \arrow[d, rightarrowtail] 
& \frac{\Gamma_Q}{\Gamma_1} \left. \bs (\Gamma_1 \bs \mX^0)  \right.
\arrow[rr]
\arrow[d,rightarrowtail]
\arrow[dr,rightarrowtail]
&& \Gamma \bs \mX^0 \arrow[d, rightarrowtail]\\
&
\Gamma_{U_1} \bs U_1(\C) 
\arrow[r] 
& \frac{\Gamma_Q}{\Gamma_1} \left. \bs (\Gamma_1 \bs \mX_1^0)  \right.
& \frac{\Gamma_Q}{\Gamma_1} \left. \bs (\Gamma_1 \bs \mX^0)_{\Sigma_1^0}  \right.
\arrow[r] \arrow[dl, rightarrowtail]
&
(\Gamma \bs \mX^0)_\Sigma \\
D \arrow[from=uu, rightarrowtail]
\arrow[r, rightarrowtail]
\arrow[rrru, dashed, start anchor={[yshift=3pt]}]
&(\Gamma_{U_1} \bs U_1(\C) )_{\sigma_s}
\arrow[r] \arrow[from=u, crossing over, rightarrowtail] 
&  \frac{\Gamma_Q}{\Gamma_1} \left. \bs (\Gamma_1 \bs \mX_1^0)_{\Sigma^0_1}  \right.
\arrow[from=u, rightarrowtail, crossing over]
&&
\end{tikzcd}
}
\end{equation*}

As $\Gamma_{U_1}/ \Lambda_{U_1}$ is finite, there exists a $\Gamma_{U_1}/ \Lambda_{U_1}$-invariant neighborhood of $T_\sigma^\sim$ that maps into $(\Gamma_1 \bs \mX^0)_{\Sigma_1^0}$. This shows a neighborhood of $(\Gamma_{U_1} \bs U_1(\C) )_{\sigma_s}^\sim$ maps into $\frac{\Gamma_Q}{\Gamma_1} \left. \bs (\Gamma_1 \bs \mX_1^0)_{\Sigma^0_1}  \right.$, viewing $\Gamma_{U_1} \bs U_1(\C)$ and $(\Gamma_{U_1} \bs U_1(\C))_{\sigma_s}$ as quotients of $T'$ and $(T')_{\sigma_s}$, respectively, by $\Gamma_{U_1}/ \Lambda_{U_1}\subset T'$.

Now that $\sigma_s$ is smooth of top dimension with respect to $\Gamma_{U_1}$, by Lemma \ref{subseclemma: some lemmas, polydisc intersection}, a polydisc neighborhood $D$ of $(\Gamma_{U_1} \bs U_1(\C) )_{\sigma_s}^\sim$ intersects $\Gamma_{U_1} \bs (U_1(\R)+C)$ in a product of punctured discs $D^\circ$. We may assume that $D$ is sufficiently small. Thus $D$ maps into $\frac{\Gamma_Q}{\Gamma_1} \bs (\Gamma_1 \bs \mX^0)_{\Sigma_1^0}$ and then into $(\Gamma \bs \mX^0)_\Sigma$, as indicated in the above diagram.

We compute the image of $\pi_1 (D^\circ) \rightarrow  \pi_1(\Gamma \bs \mX^0) \rightarrow \mathrm{Aut}(\Gamma ' \bs \mX^0 \rightarrow \Gamma \bs \mX^0)$. By Proposition \ref{subsecproposition: rational boundary components, pink 4.15}, $C$ is $(W\cap  U_1)(\R)(-1)$-translation invariant and the quotient $C /(W\cap  U_1)(\R)(-1)$ is a self-adjoint homogeneous open cone. Thus $\frac{U_1(\R)}{(W\cap U_1)(\R)}+\frac{C}{(W\cap  U_1)(\R)(-1)}$ is a tube domain, and, in particular, topologically simply connected. Then $U_1(\R)+ C$ is a $(W\cap  U_1)(\C)$-vector bundle over a simply connected base and so simply connected as well. The map $U_1(\R)+C \rightarrow \mX^0$ is $\Gamma_{ZU_1}$-equivariant, where $\Gamma_{ZU_1} $ acts via $\Gamma_{ZU_1} \rightarrowtail \Gamma_{U_1}$. By Proposition \ref{subsecproposition: mixed shimura varieties, topology}, $\Gamma /\Gamma_Z$ acts properly discontinuously (and freely) on $\mX^0$. Thus (cf. \cite{Hatcher02}) the image on $\pi_1$ is just 
\begin{IEEEeqnarray*}{rCl}
&=&\mathrm{Im}\left (\frac{\Gamma_{ZU_1}}{\Gamma_Z} \rightarrowtail
 \frac{\Gamma}{\Gamma_Z} \twoheadrightarrow 
\frac{\Gamma /\Gamma_Z}{\Gamma ' /\Gamma_Z'}=
\frac{\Gamma}{\Gamma_Z \cdot \Gamma '}\right )\\
&=&\frac{\Gamma_{ZU_1}}{\Gamma_Z \cdot \Gamma' _{ZU_1}}=\frac{\Gamma_{ZU_1}/\Gamma_Z}{\Gamma'_{ZU_1}/\Gamma_Z'}=\frac{\Gamma_{U_1}}{\Gamma_{U_1}'} .
\end{IEEEeqnarray*}

Apply Theorem \ref{sectheorem: fixed point method, existence of fixed points} and we finish the proof of the second part.

\item The assertion will be a direct consequence of the second part combined with Theorem \ref{sectheorem: fixed point method, fixed point method} once we show there exists a $\Gamma/\Gamma '$-equivariant smooth compactification $Y \rightarrow (\Gamma \bs \mX^0)_\Sigma$ extending $\Gamma ' \bs \mX^0 \rightarrow \Gamma \bs \mX^0$.

Since $\Sigma$ is a complete $K_f $-admissible cone decomposition, it is automatically a complete $K_f'(\unlhd K_f)$-admissible cone decomposition. Thus (by Theorem \ref{subsectheorem: toroidal compactifications, toroidal compactifications}) $M^{K_f'}(P, \mX, \Sigma)(\C)$ is a compact normal complex space with a finite morphism to $M^{K_f}(P, \mX, \Sigma)(\C)$, the quotient of $M^{K_f'}(P, \mX, \Sigma)(\C)$ by the finite group $K_f /K_f'$. Now the Generalized Riemann Existence Theorem (cf. {\cite[Appendix B, Theorem 3.2]{GTM52}}) implies that $M^{K_f'}(P, \mX, \Sigma)(\C)$ is a (proper) normal scheme, admitting a finite morphism to a projective variety, hence projective itself. An equivariant resolution of singularities then produces a smooth compactification $Y'$ with a $K_f /K_f'$-equivariant proper $p \colon Y' \rightarrow M^{K'_f}(P, \mX, \Sigma)(\C)$. 

By Proposition \ref{subsecproposition: toroidal compactifications, map to baily-borel}, $(\Gamma \bs \mX^0)_\Sigma$ is a smooth projective variety, being a connected component of $M^{K_f}(P, \mX, \Sigma)(\C)$. After taking the preimage $Y\coloneq p^{-1}((\Gamma' \bs \mX^0)_\Sigma ) \subset Y' $, we get a $\Gamma / \Gamma'$-equivariant smooth compactification $Y \rightarrow (\Gamma \bs \mX^0)_\Sigma$. Equivalently, it is $\Delta/\Delta'$-equivariant.

\item We just need to find $K'_f \unlhd K_f$ such that $\Gamma_{U_1}' \subset p\cdot \Gamma_{U_1}$. Define the following groups. 
\begin{IEEEeqnarray*}{rCl}
K_f^{ZU_1}        &\coloneq &\left  (Z(P)(\fadele) \times U_1(\fadele) \right )\cap  p_f K_f p_f^{-1}\\
K_f^{U_1}          &\coloneq &U_1(\fadele)\cap p_f K_f p_f^{-1}\\
K_f^Z                  &\coloneq &Z(P)(\fadele)\cap p_f K_f p_f^{-1}
\end{IEEEeqnarray*}
Let $\gamma_{U_1}\coloneq U_1(\Q)\cap K_f^{U_1}$. Then $\gamma_{U_1} \subset \Gamma_{U_1}$. We may choose a smaller $\widetilde{K}_f^{U_1} \subset K_f^{U_1}$ such that $p\cdot \gamma_{U_1}= U_1(\Q) \cap \widetilde{K}_f^{U_1}$.
But $K_f^Z \times \widetilde{K}_f^{U_1}$ is an open compact subgroup of $Z(P)(\fadele) \times U_1(\fadele)$. 
Since $Z(P)\times U_1 $ is a (closed) subgroup of $P$, we may find neat open compact $K_f' \unlhd K_f$ such that ${K_f^{ZU_1}}' \coloneq p_fK_f'p_f^{-1} \cap  \left ( Z(P)(\fadele) \times U_1(\fadele) \right ) \subset K_f^Z \times \widetilde{K}_f^{U_1}$. 
Then $\Gamma_{ZU_1}' 
\subset  
\bigl ( \Id_{Z(P)(\Q)}(\mX)\times U_1(\Q) \bigr ) \cap \bigl ( K_f^Z \times \widetilde{K}_f^{U_1} \bigr ) =
\Gamma_Z \times (p\cdot \gamma_{U_1})$. 
Thus $\Gamma_{U_1}' \subset  p\cdot \gamma_{U_1} \subset p\cdot \Gamma_{U_1}$. 

The inclusive relations among the groups are as follows.
\vspace{-0pt}\begin{equation*}\adjustbox{center}{%
\begin{tikzcd}[row sep=small]
\Gamma_{U_1} \arrow[from=r, twoheadrightarrow]
& \Gamma_{ZU_1} \arrow[r, rightarrowtail]
& K_f^{ZU_1} \arrow[r, rightarrowtail]
& p_f K_f p_f^{-1}\\
\gamma_{U_1} \arrow[from=r, twoheadrightarrow] \arrow[u, rightarrowtail]
& \Gamma_{Z}\times \gamma_{U_1} \arrow[r, rightarrowtail] \arrow[u, rightarrowtail]
& K_f^{Z} \times K_f^{U_1} \arrow[u, rightarrowtail]
& \\
p\cdot  \gamma_{U_1} \arrow[from=r, twoheadrightarrow] \arrow[u, rightarrowtail]
& \Gamma_{Z}\times  (p\cdot \gamma_{U_1}) \arrow[r, rightarrowtail] \arrow[u, rightarrowtail]
&K_f^{Z} \times \widetilde{K}_f^{U_1} \arrow[u, rightarrowtail]
& \\
\Gamma'_{U_1} \arrow[from=r, twoheadrightarrow] \arrow[u, rightarrowtail]
& \Gamma'_{ZU_1}  \arrow[r, rightarrowtail] \arrow[u, rightarrowtail]
&{K_f^{ZU_1}}' \arrow[u, rightarrowtail] \arrow[r, rightarrowtail]
& p_f K_f'p_f^{-1} \arrow[uuu, rightarrowtail]
\end{tikzcd}
}
\end{equation*}

The surjections $\Gamma_{U_1} \twoheadrightarrow \Gamma_{U_1}/\Gamma_{U_1}' \twoheadrightarrow \Gamma_{U_1} /(p\cdot \Gamma_{U_1})$ imply that
\begin{IEEEeqnarray*}{rcl}
\rank_p(\Gamma_{U_1}/\Gamma_{U_1}' )=\rank_p(\Gamma_{U_1} /(p\cdot \Gamma_{U_1}))=\dim U_1.
\end{IEEEeqnarray*}
This completes the proof of our main theorem. \qedhere
\end{enumerate}

\end{proof}


\section{Examples of lower bounds} \label{sec: examples}

Note that in Theorem \ref{sectheorem: introduction, main
 theorem}, if we take $(P_1, \mX_1)$ to be the (unique) improper boundary
 component with $\mX^0 \subset \mX^+_{(P_1, \mX_1)}$, then $U_1=U$ and 
$\ed_\C(\Gamma' \bs \mX^0 \rightarrow \Gamma \bs \mX^0; p)\geq \rank_p (\Gamma_{U} /\Gamma_{U}')$ where $\Gamma_U$ is the image of 
\[ \left ( \Id_{Z(P)(\Q)}(\mX) \times U(\Q) \right )\cap p_f K_f p_f^{-1}\]  under $Z(P) \times U \rightarrow U$. This gives a lower bound without computing any proper boundary components.

We now turn to a few examples. 

\subsection{The torus case of dimension $1$} \label{subsec: torus case of dimension 1}
 
Recall that in Example \ref{subsecexample: mixed shimura data, simplest example}, we defined the mixed Shimura datum $(P_0, \mX_0)$ as the unipotent extension of $(\mathbb{G}_m, \mathcal{H}_0)$, where $U_0=\mathbb{G}_{a, \Q}$ and $P_0=U_0 \rtimes \mathbb{G}_{m, \Q}$ with $\mathbb{G}_{m, \Q}$ acting via multiplication. Both points of $\mH_0$ map to the norm morphism $N \colon \DS \rightarrow \mathbb{G}_{m, \R}, z \mapsto z \overbar{z}$. Let $\iota \colon \mathbb{G}_{m, \Q} \rightarrowtail P_0$ be the inclusion. Then $\mX_0$ consists of two connected components $\mX_0^{\pm}$, each of which projects isomorphically onto $U_0(\C)\cdot (\iota \circ N)\overset{\sim}{\leftarrow} U_0(\C)$. Take a neat open compact subgroup $K_f^* \subset \bG_m(\hat{\Z})$ and take $K_f^U(d)\coloneq d \cdot U_0(\hat{\Z})$. Then $K_f^P(d)\coloneq K_f^U(d) \rtimes K^*_f \subset P_0(\fadele)$ is a neat open compact subgroup. 

Since $Z(P_0)=1$, for $p_f=(0,g_f) \in \bG_m(\hat{\Z})$, we see that 
$\Gamma=\Gamma_U=U_0(\Q) \cap K^P_f(d) = d \cdot U_0(\Z) \eqcolon \Gamma_U (d)$. The $d \Z/ nd \Z$-torsor $\Gamma_U(nd) \bs U_0(\C) \rightarrow \Gamma_U (d) \bs U_0(\C)$ is just the $\Z / n\Z$-torsor $\bG_{m,\C} \overset{n}{\rightarrow} \bG_{m, \C}$. By our main theorem, we have $\ed_\C(\bG_{m,\C} \overset{n}{\rightarrow} \bG_{m, \C} ; p ) \geq 1$ and thus it is equal to $1$ if and only if $p\, |\, n$.  This torsor $\bG_{m,\C} \overset{n}{\rightarrow} \bG_{m, \C}$ is $p$-incompressible if and only if $p\, |\, n$.

This assertion can also be proved in a more elementary way. If $p \nmid n$, then the pullback along $\bG_{m,\C} \overset{n}{\rightarrow} \bG_{m, \C}$ itself makes the torsor trivial, and thus $\ed_\C(\bG_{m,\C} \overset{n}{\rightarrow} \bG_{m, \C} ; p ) =0$. Denote $\bG_{m,\C} \overset{n}{\rightarrow} \bG_{m, \C}$ by $X' \rightarrow X$. Suppose $f \colon Y \rightarrow X$ is a dominant morphism from an integral variety $Y$ such that $\ed_\C( Y \times_{f, X} X')=0$. Then the pullback along $f$ must trivialize the torsor, and thus $f$ must lift to a dominant morphism $f' \colon Y \rightarrow X'$. It follows that the degree of $f$ must be divisible by $n$, hence also by $p$. We conclude that $\ed_\C(\bG_{m,\C} \overset{n}{\rightarrow} \bG_{m, \C} ; p ) =1$ if $p\, |\, n$.

\subsection{The torus case of dimension $n$} \label{subsec: torus case of dimension n}

We can generalize the previous example to the unipotent extension $(P, \mX) \rightarrow (\bG_{m,\Q}, \mathcal{H}_0)$ with $U=\Q^r$ and $P=U \rtimes \bG_{m,\Q}$. Take $K_f^U(d_1, \ldots, d_r) \subset U(\hat{\Z})=\hat{\Z}^r$ to  be the open compact subgroup $(d_1\hat{\Z})\times \cdots \times (d_r \hat{\Z})$, $d_i\neq 0$. Let $K_f^*\subset \bG_m(\hat{\Z})$ be a neat open compact subgroup and $K_f^P=K_f^U(d_1, \ldots, d_r) \rtimes K_f^*$. For $p_f=(0,g_f) \in \bG_m(\hat{\Z})$, $p_fK_f^Pp_f^{-1} =K_f^P$. Thus $\Gamma=\Gamma_U=U (\Q) \cap K_f^P= (d_1 \Z) \times \cdots \times (d_r \Z)\eqcolon \Gamma_U(d_1, \ldots, d_r)$. 

Again, $\mX$ has two connected components, each of which is isomorphic to $U(\C)$. The $(\Z/n_1\Z) \times \cdots \times (\Z/n_r\Z)$-torsor 
\[
\Gamma_U(n_1d_1, \ldots, n_r d_r) \bs U(\C) \rightarrow \Gamma_U(d_1,\ldots, d_r) \bs U(\C)
\] 
is just $\bG_{m,\C} \times \cdots \times \bG_{m,\C}  \xrightarrow{(n_1, \ldots, n_r)}\bG_{m,\C} \times \cdots \times \bG_{m,\C}$. Denote $(n_1, \ldots, n_r)$ by $(n)$. Theorem \ref{sectheorem: introduction, main theorem} implies that $\ed_\C(\bG_{m,\C}^r \overset{(n)}{\rightarrow} \bG_{m,\C}^r; p) \geq \rank_p ( (\Z /n_1\Z) \times \cdots \times (\Z /n_r\Z) )$.

Note that by Theorem \ref{sectheorem: essential dimension, KarpenkoMerkurjev} and Remark \ref{secremark: essential dimension, essential dimension of product groups}, we can deduce the upper bound $\ed_\C(\bG_{m,\C}^r \overset{(n)}{\rightarrow} \bG_{m,\C}^r; p) \leq \rank_p ( (\Z /n_1\Z) \times \cdots \times (\Z /n_r\Z) )$. Hence equality holds:
\[
\ed_\C(\bG_{m,\C}^r \overset{(n)}{\rightarrow}\bG_{m,\C}^r; p) = \rank_p (\frac{\Z}{n_1\Z} \times \cdots \times \frac{\Z}{n_r \Z}).
\]
Thus this torsor is $p$-incompressible if and only if  $ \rank_p (\frac{\Z}{n_1\Z} \times \cdots \times \frac{\Z}{n_r \Z})=r$, or equivalently $p \,| \,n_i$ for all $ 1\leq i \leq r$. 

\begin{secremark}
We may apply the fixed point method directly to the torsor $ (n) \colon \bG_{m,\C}^r \rightarrow  \bG_{m,\C}^r$ since $\bG_{m,\C}^r$ admits the partial compactification $\mathbb{A}^r_\C$. This torsor extends to a finite morphism $(n) \colon \mathbb{A}_\C^r \rightarrow  \mathbb{A}_\C^r$ and the origin is already a smooth fixed point. Thus we obtain the lower bound by Theorem~\ref{sectheorem: fixed point method, fixed point method}.
\end{secremark}

\subsection{Siegel modular varieties}\label{subsec: siegel modular varieties}
 
Let $(CSp_{2n}, \mH_{2n})$ be the pure Shimura datum defined as in Example \ref{subsecexample: mixed shimura data, siegel modular variety}, for $V=\Q^{2n}$ and $\psi \colon V \times V \rightarrow \Q$ the alternating form. (The standard basis $e_{\pm i} (1{\leq} i {\leq} n)$ is a symplectic basis for $\psi$. Thus $CSp_{2n}$ is naturally defined over $\Z$.) In order to get a better lower bound, we need to know more about the rational boundary components of $(CSp_{2n}, \mH_{2n})$.

Let $V_0 \subset V$ be a totally isotropic subspace and let $Q\coloneq \Stab_{CSp_{2n}}(V_0)$. Then $Q$ is a $\Q$-admissible parabolic subgroup of $CSp_{2n}$ and this construction gives a one-to-one correspondence between totally isotropic subspaces and $\Q$-admissible parabolic subgroups (cf. {\cite[4.25]{Pink}}). (We may assume without loss of generality that $V_0=\langle e_1, \ldots, e_r \rangle$ for $0\leq r\leq n$.) 

Fix such a $V_0$ and $Q$. Then the rational boundary component $(P_1, \mX_1)$ of $(P, \mX)$ associated with $Q$ can be described as follows. Let $x\in \mX^+_{(P_1, \mX_1)}$ and $x_1 \in \mX_1$ be its image. Then $h_{x_1} \colon \DS_\C \rightarrow P_{1,\C}\subset Q_\C \subset CSp_{2n, \C}$ induces a Hodge structure on $V$ with weight filtration as follows.
\begin{IEEEeqnarray*}{+rCl+x*}
W^{h_{x_1}}_{-2}V&=&V_0\\
W^{h_{x_1}}_{-1}V&=&V_0^\perp\\
W^{h_{x_1}}_{0}V&=&V
\end{IEEEeqnarray*}
(Note that this is consistent with the fact that $Q$ should be the subgroup of $CSp_{2n}$ preserving the above weight filtration.) Then we have
\begin{IEEEeqnarray*}{+rCl+x*}
P_1&=&\{g \in Q \, \lvert \, g|_{V/V_0^\perp}=\mathrm{id} \},   \\
W_1&=&\left \{g \in Q \,  \left \lvert \,  
\begin{array}{l} 
g|_{V/V_0^\perp}=\mathrm{id}\\
g|_{V_0^\perp/V_0}=\mathrm{id}\\
g|_{V_0}=\mathrm{id}
\end{array}
 \right. \right \},    \\
U_1&=&\left \{g \in Q \,  \left \lvert \,  
\begin{array}{l} 
g|_{V/V_0}=\mathrm{id}\\
g|_{V_0^\perp}=\mathrm{id}
\end{array}
 \right. \right \}.
\end{IEEEeqnarray*}

Let $G\coloneq CSp_{2n}$. Define $K_f(d)\coloneq \{g \in G(\hat{\Z}) \, \lvert \, g \equiv 1 \,\mathrm{mod}\, d\}  (d \geq 3)$. By {\cite[0.1]{Pink}}, $K_f(d)\unlhd G(\hat{\Z})$ is neat. It is immediate that $\Id_{Z(G)(\Q)}(\mH_{2n})=Z(G)(\Q)=\bG_m(\Q) \subset GL_{2n}(\Q)$. For $g_f \in G(\hat{\Z})$, we have $(\bG_m(\Q) \times U_1(\Q) )\cap g_f K_f(d)g_f^{-1}=d\cdot U_1(\Z)\eqcolon \Gamma_{U_1}(d)$. Take $\mH^0\coloneq \mH_{2n}^+$ to be the connected component that induces positive definite forms. Then $\Stab_{G(\Q)}(\mH^0) \cap K_f(d)= Sp_{2n}(\Q)\cap K_f(d)=Sp_{2n}(\Z)\cap K_f(d)\eqcolon \Gamma_{Sp}(d)$. Theorem \ref{sectheorem: introduction, main theorem} implies that 
\begin{IEEEeqnarray*}{+rCl+x*}
\ed_{\C}(\Gamma_{Sp}(md) \bs \mH^0 \rightarrow \Gamma_{Sp}(d) \bs \mH^0;p)&\geq& \rank_p \frac{d\cdot U_1(\Z)}{md\cdot U_1(\Z)}=\rank_p \frac{ U_1(\Z)}{m\cdot U_1(\Z)}.
\end{IEEEeqnarray*}

An inspection of the Lie algebra of $U_1$ reveals that $\dim U_1 = \frac{r(r+1)}{2}$. We find $\ed_{\C}(\Gamma_{Sp}(md) \bs \mH^0 \rightarrow \Gamma_{Sp}(d) \bs \mH^0;p) \geq \frac{r(r+1)}{2}$ if $p \,|\, m$. Now take $V_0=\langle e_1 , \ldots, e_n \rangle$ to be a maximal totally isotropic subspace. This gives 
\[
\ed_{\C}(\Gamma_{Sp}(md) \bs \mH^0 \rightarrow \Gamma_{Sp}(d) \bs \mH^0 ; p ) \geq \frac{n(n+1)}{2},~ p\,|\,m.
\]
It is well known that $\dim_\C \mH_{2n}=\frac{n(n+1)}{2}$. Therefore,  for $p \, | \, m$ we obtain the $p$-incompressibility of $\Gamma_{Sp}(md) \bs \mH^0 \rightarrow \Gamma_{Sp}(d) \bs \mH^0$. This example illustrates that a larger $U_1$ produces a better lower bound for the essential dimension.

Finally, let $M(d)\coloneq M^{K_f(d)}(G, \mH_{2n})(\C)$, $d \geq 3$. Then $M(d)$ is the moduli space of principally polarized abelian varieties of dimension $n$ with level-$d$ structure.  We have shown the  $p$-incompressibility of $M(pd) \rightarrow M(d)$ and thus have recovered \cite[41]{BrosnanFakhruddin} and \cite[3.2.9]{FarbKisinWolfson21} for $d \geq 3$.

\begin{subsecremark}
Note that when  $\gcd(m,d)=1$, $\dps  \Gamma_{Sp}(d)/\Gamma_{Sp}(md)$ is canonically isomorphic to $Sp_{2n}(\frac{\Z}{m\cdot \Z})$. Take $m=p$ to be an odd prime. The $p$-incompressibility of  $\Gamma_{Sp}(pd) \bs \mH^0 \rightarrow \Gamma_{Sp}(d) \bs \mH^0$ implies that $\ed_\C(Sp_{2n}(\mathbb{F}_p); p) \geq \frac{n(n+1)}{2}$. However, as computed in \cite[79]{BrosnanFakhruddin}, $\ed_\C(Sp_{2n}(\mathbb{F}_p); p)=p^{n-1}$. We see that there is an exponential difference. The reason for this is that we essentially used a subgroup $\Gamma_{U_1}(p)/\Gamma_{U_1}(pd) \subset \Gamma_{Sp}(d)/\Gamma_{Sp}(pd)=Sp_{2n}(\mathbb{F}_p)$ to deduce the lower bound. This subgroup may be relatively small compared to $Sp_{2n}(\mathbb{F}_p)$.
\end{subsecremark}

\subsection{Universal families over Siegel modular varieties} \label{subsec: universal families over siegel modular varieties}

Given $(V ,\psi)$ and $G=CSp_{2n}$ as in the previous example, let $U\coloneq \Q$ and $W\coloneq U\times V$ with multiplication defined by $(u,v)\cdot (u',v') \coloneq (u+u'+\frac{1}{2}\psi(v,v') , v+v')$. $G$ acts on $W$ via $g\cdot (u,v)\coloneq (g(u), g(v))= (\mu(g)\cdot u, g(v))$, where $\mu \colon G \rightarrow \bG_{m,\Z}$ is the multiplier morphism. We have a short exact sequence $1\rightarrow U \rightarrow W \rightarrow V \rightarrow 1$ leading to unipotent extensions
$P_{2n}\coloneq W \rtimes G $ and $V\rtimes G$. This gives unipotent extensions of mixed Shimura data
\[
(P_{2n}, \mX_{2n}) \rightarrow (V \rtimes G, \mY_{2n})  \rightarrow (G, \mH_{2n})
\]
as in Example \ref{subsecexample: mixed shimura data, extension of symplectic group}. 

For $d\geq 4$ even, let $K_f(d) \unlhd G(\hat{\Z})$ be as in Subsection \ref{subsec: siegel modular varieties}. Define $K_f^W(d)\coloneq K_f^U(d)\times K_f^V(d)= (d\cdot U(\hat{\Z})) \times (d\cdot V(\hat{\Z}))$ and $K_f^P(d)\coloneq K_f^W(d) \rtimes K_f(d)$. (The assumption $d\geq 4$ ensures that $K_f^W(d)$ is a subgroup normalized by $K_f(d)$.)  Consider the following mixed Shimura varieties.
\begin{IEEEeqnarray*}{rCl}
M(d)&\coloneq&M_\C^{K_f(d)}(G, \mH_{2n})\\
M_V(d)&\coloneq&M_\C^{K_f^V(d) \rtimes K_f(d)}(V \rtimes G, \mH'_{2n})\\
M_W(d)&\coloneq&M_\C^{K_f^P(d)}(P_{2n}, \mX_{2n})
\end{IEEEeqnarray*}

\begin{subsecremark}
According to \cite[10.9]{Pink}, $M_W(d) \rightarrow M_V(d) \rightarrow M(d)$ is the universal family of morphisms $X\rightarrow A \rightarrow S$, where $A\rightarrow S$ is an  abelian scheme   over $S$  of relative dimension $n$ with a symplectic level $d$-structure and $X\rightarrow A$ is a normalized totally symmetric $\bG_m$-torsor over $A$.
\end{subsecremark}

We have described the rational boundary components of $(G, \mH_{2n})$ in Subsection \ref{subsec: siegel modular varieties}. Indeed, this enables us to obtain the rational boundary components of $(P_{2n} , \mX_{2n})$ conveniently. Let $V_0 \subset V$ and $Q$ be as in Subsection \ref{subsec: siegel modular varieties}. Then $W \rtimes Q \subset  P $ is a $\Q$-admissible parabolic subgroup.  Recall that, by Subsection \ref{subsec: rational boundary components}, the rational boundary component $(P_{2n,1}, \mX_{2n,1})$ associated with $W \rtimes Q$ should involve the smallest $\Q$-normal subgroup $P_{2n,1} \unlhd W \rtimes Q$ containing the image of $\iota \circ h_{x_1}$. 
Here $\iota \colon G \rightarrowtail P_{2n}=W \rtimes G$ is the inclusion, $(P_1, \mX_1)$ is the rational boundary component of $(G, \mH_{2n})$ associated with $Q\subset G$, and $x_1$ is the image of $x$  under $\mX^+_{(P_1, \mX_1)} \rightarrowtail \mX_1$.

\begin{subsecproposition} \label{subsecproposition: universal families over siegel modular varieties, rational boundary components of unipotent extensions}
Let the rational boundary component  associated with the admissible $\Q$-parabolic subgroup $W \rtimes Q\subset W \rtimes G$ be $(P_{2n,1}, \mX_{2n,1})$. Then 
\begin{enumerate}
\item  $U_{2n,1}= (U\times V_0)\times U_1$.
\item $P_{2n,1}= (U\times V_0^\perp)\rtimes P_1$.
\end{enumerate}
\end{subsecproposition}

\begin{proof}~

\begin{enumerate}
\item $U_{2n,1}= (U \times V_0)\times U_1$ by considering the weight, as described in Subsection \ref{subsec: siegel modular varieties}.

\item 
In order to compute $P_{2n,1}$, first observe that by the description of the weight filtration in Subsection \ref{subsec: siegel modular varieties},  the unipotent radical of $P_{2n,1}$ is $ W_{2n,1}= (U\times V_0^\perp)\rtimes W_1$.  ($W_1$ is the unipotent radical of $P_1$.)

Next, we claim that $(U\times V_0^\perp)\rtimes P_1$ is a $\Q$-normal subgroup of $W \rtimes Q$. In fact, it is clearly normalized by $Q$. Take $(u, v) \in W(\Q)$. Since $P_1$ acts trivially on $V /V_0^\perp$, we have 
\begin{IEEEeqnarray*}{rCl}
(u,v,1)(0,0, p_1)(-u,-v , 1)&=& \bm{\bigl (}u-p_1(u)-\frac{1}{2}\psi(v, -p_1(v) ) \bm{,} v-p_1(v) \bm{,} p_1 \bm{ \bigr )}\\
                                                &\in& \bigl (  (U \times V_0^\perp)\rtimes P_1 \bigr ) (\Q).
\end{IEEEeqnarray*} 
Thus $ (U\times V_0^\perp)\rtimes W_1 \subset P_{2n,1} \subset  (U \times V_0^\perp)\rtimes P_1$. However, $P_{2n,1}$ must project onto $P_{1}$ as its image under the projection contains $h_{x_1}(\DS_\C)$ (and $P_1$ is the smallest such $\Q$-normal subgroup in $Q$).  Hence $P_{2n,1}= (U \times V_0^\perp)\rtimes P_1$. \qedhere
\end{enumerate}

\end{proof}

Let $K_f^{P_{2n}}(d) \coloneq  K^W_f(d) \rtimes K_f(d) $  and $\mX_{2n}^+$ the connected component that maps onto $\mH^{+}_{2n}$. We have
\begin{IEEEeqnarray*}{rcl}
\Stab_{P_{2n}(\Q)}(\mX^+_{2n}) \cap K^{P_{2n}}_f(d) \,&=&\,
\bigl ( d{\cdot} U(\Z) {\times} d{\cdot} V(\Z) \bigr ) {\rtimes} \Gamma_{Sp}(d)  
{\eqcolon}  \Gamma_{P_{2n}}(d);\\
\bigl (\Id_{Z(P_{2n})(\Q)}(\mX_{2n}) {\times} U_{2n,1}(\Q) \bigr ) \cap K_f^{P_{2n}}(d)\,&=&\,
\bigl ( d{\cdot} U(\Z) {\times} d{\cdot} V_0(\Z) \bigr) {\times} (d{\cdot} U_1(\Z)), \\
&&\text{as }\Id_{Z(P_{2n})(\Q)}(\mX_{2n})=1.
\end{IEEEeqnarray*} 
If $p \, | \, m$, Theorem \ref{sectheorem: introduction, main theorem} concludes that 
\[
\ed_\C(\Gamma_{P_{2n}}(md) \bs \mX^+_{2n} \rightarrow \Gamma_{P_{2n}}(d) \bs \mX^+_{2n}; p ) \geq 1+r +\frac{r(r+1)}{2}.
\]

A maximal totally isotropic subspace $V_0$ of dimension $r=n$ gives the lower bound $1+n+\frac{n(n+1)}{2}$. Since $\dim_\C \mX_{2n} =1 +n +\frac{n(n+1)}{2}$, we obtain

\begin{subsectheorem} \label{subsectheorem: universal families over siegel modular varieties, incompressibility}
Let $d\geq 4$ be an even integer.\ The above $\Gamma_{P_{2n}}(d)/\Gamma_{P_{2n}}(md)$-torsor 
$\Gamma_{P_{2n}}(md) \bs \mX^+_{2n} \rightarrow \Gamma_{P_{2n}}(d) \bs \mX^+_{2n}$ is $p$-incompressible for $p \, | \, m$.
\end{subsectheorem}

\subsection{Kuga varieties}\label{subsec: kuga varieties}

A \emph{Kuga datum} $(P, \mX)$ is a mixed Shimura datum with $U=1$. That is, there is no subspace in $\Lie P$ of weight $-2$. A \emph{Kuga variety} is a mixed Shimura variety $M^{K_f}(P, \mX)(\C)$ from a Kuga datum. For more properties of Kuga data, the reader can refer to \cite{KeChen13}.

Here we consider the Kuga datum $(V\rtimes G, \mY_{2n})$ as defined in Subsection \ref{subsec: universal families over siegel modular varieties}.
Similar to Proposition \ref{subsecproposition: universal families over siegel modular varieties, rational boundary components of unipotent extensions}, we have the following

\begin{subsecproposition}\label{subsecproposition: kuga varieties, rational boundary components of kuga extensions}
Let $(P'_{2n,1}, \mY_{2n,1})$ be the rational boundary component associated with the $\Q$-admissible parabolic $V \rtimes Q\subset V \rtimes G=P_{2n}'$ . Then 
\begin{enumerate}
\item $P'_{2n,1}= V_0^\perp \rtimes P_1$.

\item  $U'_{2n,1}= V_0\times U_1$.
\end{enumerate}
\end{subsecproposition}

Let $\mY^+_{2n}$ be the connected component over $\mH^+_{2n}$. Let $K'_f(d) \coloneq K^V_f(d) \rtimes K_f(d)$.
\begin{IEEEeqnarray*}{rCl}
\Stab_{P_{2n}' (\Q)}(\mY^+_{2n}) \cap K'_f(d) &=&
\left (  d{\cdot} V(\Z) \right ) \rtimes \Gamma_{Sp}(d)  
\eqcolon  \Gamma' (d);\\
\left (\Id_{Z(P_{2n}')(\Q)}(\mY_{2n}) \times U'_{2n,1}(\Q) \right ) \cap K_f'(d)&=&
\left (  d{\cdot} V_0(\Z) \right) \times (d{\cdot} U_1(\Z)), \\
&&\text{as }\Id_{Z(P_{2n}')(\Q)}(\mY_{2n})=1.
\end{IEEEeqnarray*} 
 Again, a maximal totally isotropic subspace $V_0$ of dimension $r=n$ gives the lower bound $ n+\frac{n(n+1)}{2}$. Since $\dim_\C \mY_{2n} = n +\frac{n(n+1)}{2}$, we get

\begin{subsectheorem} \label{subsectheorem: kuga varieties, incompressibility}
Let $d\geq 3$ be any integer.\ The above $\Gamma'(d)/\Gamma'(md)$-torsor 
$\Gamma'(md) \bs \mY^+_{2n} \rightarrow \Gamma'(d) \bs \mY^+_{2n}$ is $p$-incompressible for $p \, | \, m$.
\end{subsectheorem}

\begin{subsecremark}
As $V$ is already a group, we don't need to assume that $d$ is even as in Subsection \ref{subsec: universal families over siegel modular varieties}.
\end{subsecremark}

Let us examine the structure of the underlying Kuga varieties. First, the splitting $G \rightarrowtail V\rtimes G$ yields an inclusion $\mH_{2n} \rightarrowtail \mY_{2n}$. Since $V(\R) \rtimes G(\R)$ acts transitively on $\mY_{2n}$ with $V(\R)$ acting freely, we can deduce that the map
\begin{IEEEeqnarray*}{rCl}
 V(\R)\times \mH_{2n}  &\rightarrow& \mY_{2n} \\
                 (v, x)   &\mapsto& v\cdot x
\end{IEEEeqnarray*}
is a $V(\R) \rtimes G(\R)$-equivariant analytic isomorphism. Here $V(\R) \rtimes G(\R)$ acts on $V(\R)\times \mH_{2n}$  via the formula $(v,g)\cdot (w,x)\coloneq (v+g(w), g\cdot x)$.

Consider the action of the groups $\Gamma_V(l) \rtimes \Gamma_{Sp}(l) \,\, (l=pd, d)$, we have the following diagram.

\begin{equation*}\adjustbox{center}{
\begin{tikzcd}
\Gamma_V(pd) \bs V(\R) \times \{x\} 
\arrow[r, twoheadrightarrow]
\arrow[d, rightarrowtail, "\text{fiber}"'] &
\Gamma_V(d) \bs V(\R) \times \{x\} 
\arrow[r, twoheadrightarrow, "\sim"]
\arrow[d, rightarrowtail, "\text{fiber}"'] &
\Gamma_V(d) \bs V(\R) \times \{x\} 
\arrow[d, rightarrowtail, "\text{fiber}"'] \\
\Gamma_V(pd) \rtimes \Gamma_{Sp}(pd) \bs V(\R)\times \mH_{2n}^+
\arrow[r, twoheadrightarrow]
\arrow[dr, twoheadrightarrow] &
\Gamma_V(d) \rtimes \Gamma_{Sp}(pd) \bs V(\R)\times \mH_{2n}^+
\arrow[r, twoheadrightarrow]
\arrow[d, twoheadrightarrow] &
\Gamma_V(d) \rtimes \Gamma_{Sp}(d) \bs V(\R)\times \mH_{2n}^+
\arrow[d, twoheadrightarrow] \\
&
\Gamma_{Sp}(pd) \bs \mH_{2n}^+ \arrow[r, twoheadrightarrow]&
\Gamma_{Sp}(d) \bs \mH_{2n}^+
\end{tikzcd}
}
\end{equation*}
It is not hard to see that the bottom right square is a pullback. Thus the congruence cover $M_V(pd) \rightarrow M_V(d)$ is the pullback of $M(pd) \rightarrow M(d)$ along $M_V(d) \rightarrow M(d)$ composed with the map of multiplication by $p$ on each fiber of $M_V(d)$. 

\begin{subsecremark}\label{subsecremark: examples, brosnan's conjecture}
Based on the above analysis, the incompressibility result in Theorem \ref{subsectheorem: kuga varieties, incompressibility} can be viewed as an analogue of Brosnan's conjecture that the map of multiplication by $p$ on an abelian variety (over a field of characteristic $0$) is $p$-incompressible. In \cite[6.4]{FakhruddinSaini22}, this is confirmed for abelian varieties of dimension at most three and a positive-density set of primes. \cite[2.3.12]{FarbKisinWolfson24} proves the conjecture for all sufficiently large primes $p$. With techniques from birational geometry, \cite{KollarZhuang25} shows that $[n] \colon A\rightarrow A$ is incompressible for all integers $n\geq 2$.
\end{subsecremark}

Let $\mY_{n,d}\coloneq M_V(d)$ and $\mA_{n,d} \coloneq M(d)$. The above diagram then becomes the one in the introduction.
 
\begin{equation*}\adjustbox{center}{
\begin{tikzcd}
 \mY_{n,pd} \arrow[r, "p"] \arrow[dr] &
\mY_{n,pd} \arrow[r]  \arrow[d] &
\mY_{n,d} \arrow[d]\\
 &
\mA_{n, pd} \arrow[r] &
\mA_{n,d}
\end{tikzcd}
}
\end{equation*}


\section{Examples of fixed points} \label{sec: examples of fixed points}

To further illustrate our main theorem, in this section we describe two explicit fixed points for certain toroidal compactifications of the mixed Shimura data $(GL_{2,\Q}, \mH_2)$ and $(V\rtimes GL_{2, \Q}, \mY_2)$, defined in Subsections \ref{subsec: siegel modular varieties} and \ref{subsec: kuga varieties}. 

In order not to get lost in the details, we first outline the strategy. Given a mixed Shimura datum $(P,\mX)$ with a fixed $\Q$-admissible parabolic $Q$ and associated rational boundary component $(P_1, \mX_1)$, suppose we have neat open compact subgroups $K_f \unlhd K_f' \subset P(\fadele)$ and a $K_f'$-admissible polyhedral decomposition $\Sigma$. For a fixed $\mX^0 \subset \mX^+_{(P_1, \mX_1)}$, let $\mX^0_1\subset \mX_1$ be the corresponding connected component. 
Recall from Subsection \ref{subsec: toroidal compactifications} that we have a canonical map $\mU(P_1, \mX_1, 1)  \rightarrow M^{K_f}(P, \mX)(\C)$, which is just $P_1(\Q) \bs \mX^+_{(P_1, \mX_1)}\times P_1(\fadele)/K_f^1 \rightarrow P(\Q) \bs \mX \times P(\fadele)/K_f$, where $K_f^1= P_1(\fadele) \cap K_f$. Moreover, 
\[
\mU(P_1, \mX_1, 1)  \rightarrowtail  M^{K^1_f}(P_1, \mX_1)(\C) \rightarrowtail  M^{K^1_f}(P_1, \mX_1, \Sigma_1^0)(\C),
\]
where $M^{K^1_f}(P_1, \mX_1, \Sigma_1^0)(\C)$ is the torus embedding relative to the $K^1_f$-admissible cone decomposition $\Sigma_1^0=(\Sigma|_{(P_1,\mX_1)} )^0$. Again, $\mUbar(P_1, \mX_1, 1)$ denotes the interior of the closure of $\mU(P_1, \mX_1, 1)$ in $M^{K^1_f}(P_1, \mX_1, \Sigma_1^0)(\C)$. 

To simplify notations, we will consider the connected component $\Gamma \bs \mX^0 \subset M^{K_f}(P, \mX)(\C)$ obtained from $\mX^0 \times\{1\}$. The corresponding connected component in the toroidal compactification $M^{K_f}(P, \mX, \Sigma)(\C)$ is denoted by $(\Gamma \bs \mX^0)_\Sigma$, and the corresponding connected component in $\mU(P_1, \mX_1, 1)$ is $\Gamma_1\bs \mX^0$. It embeds into $\Gamma_1 \bs \mX^0_1$, whose torus embedding relative to $\Sigma_1^0$ is denoted $(\Gamma_1 \bs \mX^0_1)_{\Sigma_1^0}$. The interior of the closure of $\Gamma_1 \bs \mX^0 \subset (\Gamma_1 \bs \mX^0_1)_{\Sigma_1^0}$ is denoted $ (\Gamma_1 \bs \mX^0)_{\Sigma_1^0}$, which maps into the toroidal compactification $ M^{K_f}(P, \mX, \Sigma)(\C)$. We obtain the following diagram.
\begin{equation*}
\begin{tikzcd}
     \Gamma_1 \bs \mX^0 \arrow[r, rightarrowtail]  \arrow[d, twoheadrightarrow, start anchor={[yshift=1pt]}]
&   (\Gamma_1 \bs \mX^0)_{\Sigma_0^1} \arrow[d, start anchor={[yshift=2.5pt]}] \\
\Gamma \bs \mX^0 \arrow[r, rightarrowtail] & (\Gamma \bs \mX^0)_\Sigma
\end{tikzcd}
\end{equation*}
We will find a fixed point of $\Gamma_{U_1}'/ \Gamma_{U_1}$ in $ (\Gamma_1 \bs \mX^0)_{\Sigma_0^1}$. Since all the maps are $\Gamma_{U_1}'/ \Gamma_{U_1}$-equivariant, this fixed point descends to a fixed point in $(\Gamma \bs \mX^0)_\Sigma$.

Note that the fibers of $ M^{K^1_f}(P_1, \mX_1)(\C) \rightarrow M^{\pi_{U_1}(K^1_f)}(P_1/U_1, \mX_1/U_1)(\C)$ are isomorphic to $\Omega_{U_1}\bs U_1(\C)$ and the torus embedding $M^{K^1_f}(P_1, \mX_1, \Sigma_1^0)(\C)$ is constructed with respect to this torus and $\Sigma_1^0$. The image of $\Gamma_1 \bs \mX^0$ in $\Gamma_1 \bs \mX^0_1$ is the inverse image of $C(\mX^0, P_1) \subset U_1(\R)(-1)$ under the map $\mathrm{im} \colon \Gamma_1 \bs \mX^0_1 \rightarrow U_1(\R)(-1)$. All the involved groups are defined as follows. 
\begin{IEEEeqnarray*}{rCl}
\Gamma&=&\Stab_{P(\Q)}(\mX^0)\cap K_f\\
\Gamma_1&=&\Stab_{P_1(\Q)}(\mX^0) \cap K_f(=\Stab_{P_1(\Q)}(\mX_1^0) \cap K_f^1)\\
\Gamma_{U_1}&=&\pr_{U_1}\Bigl( \bigl( \Id_{Z(P)(\Q)}(\mX) \times U_1(\Q) \bigr )\cap K_f \Bigr)\\
  \Omega_{U_1}&=&\pr_{U_1}\Bigl(  \bigl ( \Id_{Z(P_1)(\Q)}(\mX_1) \times U_1(\Q) \bigr) \cap K_f  \Bigr)
\end{IEEEeqnarray*}
There are also groups $\Gamma'$, $\Gamma_1'$, $\Gamma_{U_1}'$ and $\Omega_{U_1}'$, defined in a similar way.

First we do everything for $(P,\mX)=(GL_{2,\Q}, \mH_2)$ and then for $(\hP, \hX)=(V\rtimes GL_{2,\Q}, \mY_2)$. Part of the calculation presented below is also available in (and motivated by) \cite[10.19]{Pink}.

\subsection{The Case of $(P, \mX)=(GL_{2,\Q}, \mH_2)$} \label{subsec: the case of gl2}

Let $(P, \mX)=(GL_{2,\Q}, \mH_2)$ and $Q \subset P$ the parabolic subgroup of upper-triangular matrices. Let $h_x \colon \bS \rightarrow GL_{2,\R}$ induce the standard complex structure on $V=\R^2$, sending $i$ to $\begin{bmatrix}0 & -1 \\ 1 & 0 \end{bmatrix}$. Then $h_x \in \mX^0 \coloneq \mH_2^{+}$, the connected component inducing positive definite forms. In Examples \ref{subsecexample: rational boundary components, gl2 part 1} and \ref{subsecexample: rational boundary components, gl2 part 3}, we already know that the rational boundary component  associated with $Q$ is $(P_1, \mX_1)$, where
\begin{IEEEeqnarray*}{c}
P_1=\left\{  \begin{bmatrix}  * &* \\ 0 & 1\end{bmatrix}\right \} 
\supset U_1= \left \{ \begin{bmatrix}  1 &* \\ 0 & 1\end{bmatrix}\right \},
\end{IEEEeqnarray*}
and $\mX_1$ is the $P_1(\R)\cdot U_1(\C)$-conjugacy class of the morphism 
\[
\bS_\C \rightarrow P_{1,\C} \colon (z_1 ,z_2) \mapsto \begin{bmatrix}  z_1z_2 & 0 \\ 0 & 1\end{bmatrix}. \]

As $V_{h_x}^{0, -1}=  \langle \begin{bmatrix} -i \\ 1\end{bmatrix} \rangle$, under the open embedding $\mX^0 \rightarrowtail \mX_{1}^0$, the image of $h_x$ is $h_{x_1}\colon \bS_\C \rightarrow P_{1,\C}$, given by 
\begin{IEEEeqnarray*}{cccCc}
&& (z_1 , z_2) &\longmapsto& \begin{bmatrix}  1 &-i \\ 0 & 1\end{bmatrix}\begin{bmatrix}  z_1z_2 & 0 \\ 0 & 1\end{bmatrix} \begin{bmatrix}  1 & i \\ 0 & 1\end{bmatrix}.
\end{IEEEeqnarray*}
 For $q = \begin{bmatrix}  a & b \\ 0 & c\end{bmatrix}\in Q(\R)^0 ~(ac > 0)$, we have
\begin{IEEEeqnarray*}{ccrCl}
 &&h_{x_1}&\overset{\mathrm{im}}{\longmapsto}& \begin{bmatrix}  1 &-i \\ 0 & 1\end{bmatrix},\\
 &&q\cdot h_{x_1}&\overset{\mathrm{im}}{\longmapsto}& q\begin{bmatrix}  1 &-i \\ 0 & 1\end{bmatrix}q^{-1}= \begin{bmatrix}  1 &-ac^{-1}i \\ 0 & 1\end{bmatrix}.
\end{IEEEeqnarray*}
Since $Q(\R)^0$ acts transitively on $\mX^0$, the cone $C(\mX^0, P_1)\subset U_1(\R)(-1)$ is the positive axis   $(2\pi i)^{-1}\cdot \R_{>0}$ and $C^*(\mX^0, P_1)=(2\pi i)^{-1}\cdot \R_{\geq 0}$ . 

Let $K^P_f(d) \coloneq \{g \in P(\hat{\Z})\,\lvert\,g \equiv 1\,\mathrm{mod}\,d\}$. Then one can compute that 
\begin{IEEEeqnarray*}{rCl}
\Stab_{Q(\Q)}(\mX^0) \cap P_1(\fadele)\cdot K_f^P(d)
&=& \left \{ 
\begin{array}{l} 
  \left \{ \left . \pm \begin{bmatrix} a & b\\ 0 &1\end{bmatrix}  \right | \begin{array}{ll}a >0\\ b\in \Q\end{array}   \right \}, \,1\leq d \leq 2; \\
\vspace{-7pt}\\
\left \{ \left . \begin{bmatrix} a & b\\ 0 &1\end{bmatrix} \right | \begin{array}{ll}a >0\\ b\in \Q\end{array} \right \}, \,d \geq 3.
 \end{array} \right.
\end{IEEEeqnarray*}
As $q \begin{bmatrix} 1 & \lambda\\ 0 & 1\end{bmatrix} q^{-1}=\begin{bmatrix} 1 & ac^{-1}\lambda\\ 0 & 1\end{bmatrix}$ and $P(\fadele)=Q(\Q)\cdot K^P_f(1)$, we see that 
\[
(2\pi i)^{-1}\cdot (\{0\}\cup \R_{\leq 0}) \subset U_1(\R)(-1)
\]
determines a $K^P_f(1)$-admissible polyhedral decomposition $\Sigma$ for $(P,\mX)$. Define $\sigma \coloneq (2\pi i)^{-1}\cdot \R_{\leq 0}$.

The corresponding groups are computed below.
\begin{IEEEeqnarray*}{rCl}
 \Gamma_1(d)&\coloneq&\Stab_{P_1(\Q)}(\mX^0_{1})\cap K_f^P(d)=d\cdot U_1(\Z)\\
&=&\bigl ( \Id_{Z(P_1)(\Q)}(\mX_1)\times U_1(\Q) \bigr ) \cap K^{P}_f(d)\\
\Gamma_{U_1}(d)&\coloneq& \pr_{U_1} \Bigl ( \bigl ( \Id_{Z(P)(\Q)}(\mX)\times U_1(\Q) \bigr ) \cap K^{P}_f(d)  \Bigr)=d\cdot U_1(\Z)
\end{IEEEeqnarray*}
Thus, in this case $\Gamma_1(d)=\Gamma_{U_1}(d) =\Omega_{U_1}(d)=d\cdot U_1(\Z)$. We shall take $K_f^P(d) \unlhd K_f^P(1)$ as our open compact subgroups  and find a fixed point for $\Gamma_{U_1}(1)/ \Gamma_{U_1}(d)=U_1(\Z) / d\cdot U_1(\Z)$. (Although $K^P_f(1)$ is not neat, this will produce a fixed point under the action of $\Gamma_{U_1}(d')/\Gamma_{U_1}(d)$ for any $K^P_f(d) \subset K^P_f(d') \subset K^P_f(1)$ with $3\leq d' \, |\, d$.)

The morphism
\begin{IEEEeqnarray*}{cCrCl}
\varphi& \colon &  \DS_\C &\longrightarrow& P_{1,\C}\\
              &     & (z_1, z_2)& \longmapsto& \begin{bmatrix} z_1 z_2 &0 \\ 0 & 1\end{bmatrix}
\end{IEEEeqnarray*}
is defined over $\R$. The $\varphi$-orbit map $U_1(\C)\rightarrow \mX^0_1$ defined as $ u_1 \mapsto u_1 \cdot \varphi$ is an isomorphism and induces another isomorphism $\Gamma_{U_1}(d)  \bs U_1(\C) \overset{\sim}{\rightarrow} \Gamma_1(d) \bs \mX^0_1$. Let $T$ be the torus 
 $\Gamma_{U_1}(d)  \bs U_1(\C)=(d \cdot \Z) \bs \C$. The monoid $X^*(T)_{\sigma \geq 0}$ is generated by the following character.
\begin{IEEEeqnarray*}{cCrCl}
 -\chi & \colon & (d\cdot \Z) \bs \C& \overset{~}{\longrightarrow}& \Z(1)\bs \C\\
         &  &          z &\longmapsto& -\frac{2\pi i}{d}\cdot z
\end{IEEEeqnarray*}
The torus embedding $T\rightarrow T_\sigma$ can be viewed as
 \begin{equation*}
\begin{tikzcd}[column sep=21pt]
  (d\cdot \Z) \bs \C \arrow[r,"{\dps -\chi}" yshift=1pt, "\sim"']
&Z(1) \bs \C \arrow[r, "\mathrm{exp}" yshift=1pt, "\sim"']
& \C^\times \arrow[r, rightarrowtail]
&\C.
\end{tikzcd}
\end{equation*}
Note that $\Gamma_1(d) \bs \mX^0 \rightarrowtail \Gamma_1(d) \bs \mX^0_1$ is just 
\[
(d\cdot \Z) \bs (\R+ C(\mX^0, P_1))=(d\cdot \Z) \bs (\R+(2\pi i)^{-1}\cdot \R_{>0}).
\]
 Thus under the above map $z \mapsto e^{-\frac{2\pi i}{d}z}$, the image of $\Gamma_1(d) \bs \mX^0$ is $D_1^0\coloneq D_1 -\{0\}$, the open unit disk minus the origin. The interior of its closure is just the open unit disk $D_1$. The $\sigma$-stratum is just the origin, fixed by the entire $T=(d\cdot \Z) \bs \C$ and in particular by the subgroup $(d\cdot \Z) \bs \Z$. Thus we are done. The following diagram summarizes the situation.
\begin{equation*}
\begin{tikzcd}
\Gamma_1(d) \bs \mX^0_1 
&   \dps (d\cdot \Z) \bs \C 
              \arrow[l, "{\sim}"]  \arrow[r, "{\dps z \mapsto  e^{-\frac{2\pi i}{d}z}}" yshift=2pt, "{\sim}"']
& \C^\times 
             \arrow [r, rightarrowtail]
& \C\\
     \Gamma_{1}(d) \bs \mX^0  
              \arrow[u, rightarrowtail] 
&  (d\cdot \Z) \bs (\R+(2\pi i)^{-1}\cdot \R_{>0})
               \arrow[l, "\sim"]
              \arrow[u, rightarrowtail]
              \arrow[r,"{\sim}"]
& D_1^0  \arrow[u, rightarrowtail] \arrow[r, rightarrowtail]
& D_1 \arrow[u,rightarrowtail]
\end{tikzcd}
\end{equation*}

\subsection{The Case of $(\hP, \hX)=(V\rtimes GL_{2,\Q}, \mY_2)$ }\label{subsec: the case of kuga}
Let $V=\Q^{2}$ with the standard action of $GL_{2,\Q}$. The datum $(\hP, \hX)\coloneq (V\rtimes GL_{2,\Q}, \mY_2)$ is constructed as a unipotent extension of $(GL_{2,\Q}, \mH_2)$. Let $\hQ \coloneq V\rtimes Q$ be the preimage of $Q$ under $V\rtimes GL_{2,\Q} \rightarrow GL_{2, \Q}$, where $Q$ is the parabolic that  stabilizes $V_0=\langle \begin{bsmallmatrix}*\\0\end{bsmallmatrix}\rangle$. The rational boundary component associated with $\hQ$ is the unipotent extension $(\hP_1, \hX_1)$ above $(P_1, \mX_1)$, where $\hP_1=V_0\rtimes P_1$ and $\hU_1=V_0\times U_1$. 

$\hQ(\R)=V(\R)\rtimes Q(\R)$ acts on $\hU_1(\R)=V_0(\R)\times U_1(\R)$ by conjugation: $(v,q)\cdot (v_0, u_1)=((I-q u_1q^{-1})v+qv_0, qu_1q^{-1})$. More precisely, 
\begin{IEEEeqnarray*}{rCl}
\left (\begin{bmatrix} v_x \\ v_y \end{bmatrix}, \begin{bmatrix}a & b \\ 0& c\end{bmatrix}  \right ) \cdot \left (\begin{bmatrix} v_{0x} \\ 0 \end{bmatrix}, \begin{bmatrix}1 & \lambda \\ 0& 1\end{bmatrix}  \right ) &=& \left (\begin{bmatrix} ac^{-1}v_y \lambda+av_{0x}  \\ 0 \end{bmatrix}, \begin{bmatrix}1 & ac^{-1}\lambda  \\ 0& 1\end{bmatrix}  \right ).
\end{IEEEeqnarray*} 

Let  $\iota \colon GL_{2, \Q}\rightarrowtail V\rtimes GL_{2, \Q}$ be the inclusion and  $\hX^0$ the connected component above $\mX^0=\mH_2^+$. Then $\iota \circ h_x \in \hX^0$ and its image under the open embedding $\hX^0 \rightarrowtail \hX_1^0$ is just $\iota \circ h_{x_1}$. Thus 
\begin{IEEEeqnarray*}{rCl}
\him \bigl ( (v,q)\cdot (\iota \circ h_{x_1})  \bigr)&=&(v,q)\cdot \him(\iota \circ h_{x_1})\\
&=&(v,q)\cdot \left (0,  \begin{bmatrix} 1 & -i \\ 0& 1 \end{bmatrix} \right )\\
&=&\left(   \begin{bmatrix} -iac^{-1}v_y \\ 0  \end{bmatrix}, \begin{bmatrix} 1 & -iac^{-1} \\ 0& 1 \end{bmatrix}\right).
\end{IEEEeqnarray*} 

As $v_y\in \R$ and $ac>0$ for $\Stab_{\hQ(\R)}(\hX^0)$, we see that the cone $C(\hX^0, \hP_1)=(\R_{>0} \times \R)(-1) \subset (U_1\times V_0)(\R)(-1)$. The cone $C^*(\hX^0, \hP_1)=(\{0\}\cup \R_{>0} \times \R )(-1)$. (Note that here we place $U_1$ in the front.)

Consider the cone decomposition $\{0\}\cup \bigcup_{n\in \Z} \sigma_n$, where 
\[
\sigma_n\coloneq  (2\pi i)^{-1} \bigl ( \R_{\geq 0}\cdot (1, n)+ \R_{\geq 0}\cdot (1, (n+1))\bigr ) \subset U_1(\R)(-1)\times V_0(\R)(-1).
\]
Let $K_f^\hP \coloneq K_f^V(d)\rtimes K_f^P(d)$ where $K_f^V(d) \coloneq d \cdot V(\hat{\Z})$. We have
\begin{IEEEeqnarray*}{rCl}
 &&\Stab_{\hQ(\Q)}(\hX^0)\cap \hP_1(\fadele)\cdot K_f^\hP(d)\\
&=& \left \{ 
\begin{array}{l}
\begin{bmatrix}\Q \\ d\cdot \Z \end{bmatrix} \rtimes \left \{ \pm \left. \begin{bmatrix} a & q\\ 0 &1\end{bmatrix}  \right | \begin{array}{l}a\in \Q_{>0}\\ q\in \Q \end{array}\right \}, \,1\leq d \leq 2;\\
\vspace{-7pt}\\
\begin{bmatrix}\Q \\ d\cdot \Z \end{bmatrix} \rtimes \left \{   \left. \begin{bmatrix} a & q\\ 0 &1\end{bmatrix}  \right | \begin{array}{l}a\in \Q_{>0}\\ q\in \Q \end{array}\right \}, \,3\leq d.
\end{array} \right.
\end{IEEEeqnarray*} 
Since
\begin{IEEEeqnarray*}{rCl}
\left (\begin{bmatrix} v_x \\ d\cdot v_y \end{bmatrix}, \pm\begin{bmatrix}a & q \\ 0& 1\end{bmatrix}  \right ) \cdot \left (\begin{bmatrix} v_{0x} \\ 0 \end{bmatrix}, \begin{bmatrix}1 & \lambda \\ 0& 1\end{bmatrix}  \right ) &=& \left (\begin{bmatrix} ad v_y \lambda\pm av_{0x}  \\ 0 \end{bmatrix}, \begin{bmatrix}1 & a\lambda  \\ 0& 1\end{bmatrix}  \right ),
\end{IEEEeqnarray*} 
on $U_1\times V_0$ this is the same as $(\lambda, v_{0x})\mapsto a\cdot(\lambda, \pm v_{0x}+dv_y\lambda)$ with $a>0$ and $v_y\in \Z$. This action leaves invariant the above cones and thus $\{0\}\cup \bigcup_{n\in \Z} -\sigma_n$ determines a $K_f^{\hP}(1)$-admissible cone decomposition $\hSigma$ for $(\hP, \hX)$. The analogous groups $\hGamma_1(d)$ and $\hGamma_{U_1}(d)$ are computed as follows.
\begin{IEEEeqnarray*}{rCl}
 \hGamma_1(d)&=&\Stab_{\hP_1(\Q)}(\hX^0_1)\cap K_f^\hP(d)=(d\cdot V_0(\Z)) \times (d\cdot U_1(\Z)) \\
                              &=& \bigl (  \Id_{Z(\hP_1)(\Q)}(\hX_1)\times \hU_1(\Q) \bigr ) \cap K_f^\hP(d)\\
\hGamma_{U_1}(d) &=&\pr_{\hU_{1}}\Bigl (  \bigl ( \Id_{Z(\hP)(\Q)}(\hX)\times \hU_1(\Q) \bigr)\cap K_f^\hP(d) \Bigr ) = (d\cdot V_0(\Z)) \times (d\cdot U_1(\Z))
\end{IEEEeqnarray*} 
Therefore, again we have $\hGamma_{1}(d)=\widetilde{\Omega}_{U_1}(d)=\hGamma_{U_1}(d)=(d\cdot V_0(\Z)) \times (d\cdot U_1(\Z))$. We will find a fixed point for $\hGamma_{U_1}(1)/\hGamma_{U_1}(d)$. (Although $K^\hP_f(1)$ is not neat, this will produce fixed points for neat subgroups contained in it.)

Note that the $\iota \circ \varphi$-orbit map $V_0(\C)\times U_1(\C) \rightarrow \hX^0_1$ is again an isomorphism and it induces an isomorphism 
$
(d\cdot V_0(\Z)) \times (d\cdot U_1(\Z)) \bs V_0(\C) \times U_1(\C) \overset{\sim}{\rightarrow} \hGamma_1(d) \bs \hX^0_1
$. From now on, we will deal with $U_1 \times V_0$, i.e., the $U_1$-component corresponds to the $x$-axis. Let $\hT$ be the torus $(d\cdot \Z) \bs \C \times (d\cdot \Z) \bs \C$. Then $X^*(\hT)=\mathrm{Hom}((d\cdot \Z)\times (d\cdot \Z), \Z(1))$ and $X^*(\hT)_{\sigma_n\geq 0}$ is generated as a monoid by $(n+1, -1)$ and $(-n, 1)$. Since we use $-\sigma_{n}$ to construct the torus embedding, the corresponding characters should be $(-(n+1), 1)$ and $(n, -1)$. More precisely, they are given as follows.
\begin{IEEEeqnarray*}{rCl}
\dps \frac{\C}{d\cdot \Z} \times \frac{\C}{d\cdot \Z}&\longrightarrow& \frac{\C}{\Z(1)}\\
\chi_1 \colon  (z_1, z_2)&\longmapsto&\frac{2\pi i}{d}(-(n+1)z_1 +z_2)\\
\chi_2 \colon  (z_1, z_2)&\longmapsto&\frac{2\pi i}{d}( n z_1 -z_2)
\end{IEEEeqnarray*}

Thus the torus embedding $\hT\rightarrowtail \hT_{\sigma_n}$ may be viewed as
\begin{equation*}
\begin{tikzcd}[column sep=28pt]
 \dps \frac{\C}{d\cdot \Z} \times \frac{\C}{d\cdot \Z} 
          \arrow[rr, " {    \resizebox{75pt}{!}{ $\dps \frac{2\pi i}{d}  \begin{bmatrix}\dps -(n+1)& \dps 1\\ \dps n &\dps -1\end{bmatrix}  $ }  }" yshift =3pt, "\sim" ' yshift=-0.5pt]
&&\dps  \frac{\C}{\Z(1)} \times \frac{\C}{\Z(1)}
           \arrow[r, "{ \resizebox{31pt}{!}{$\dps \exp \times \exp$} }" yshift=2pt, "\sim"' yshift=-0.5]
&\dps  \C^\times \times \C^\times
           \arrow[r, rightarrowtail]
&\dps  \C  \times \C  .
\end{tikzcd}
\end{equation*}
To describe the image of $\hGamma_1(d) \bs \hX^0 \rightarrowtail \hGamma_1(d) \bs \hX^0_1$, observe that
\begin{equation*}
\begin{tikzcd}[column sep=21pt]
\hGamma_1(d) \bs \hX^0 \arrow[r, equal]
&\dps \frac{\R \times \R +C(\hX^0, \hP_1)}{(d\cdot \Z)\times (d\cdot \Z)}
  \arrow[r, "{ \frac{2\pi i}{d}}" yshift=1pt]
&  \dps  \frac{\R(1) \times \R(1) +\frac{2\pi i}{d}C(\hX^0, \hP_1)}{ \Z(1)\times \Z(1)}
\end{tikzcd}
\end{equation*}
and that $\frac{2\pi i}{d}C(\hX^0, \hP_1)=\R_{>0} \times \R$. Hence under $\begin{bmatrix}-(n+1) & 1\\ n &-1\end{bmatrix}$, it maps to the region 
\[
\R_{>0}\cdot  \begin{bmatrix}-(n+1)  \\ n  \end{bmatrix}+ \R \cdot  \begin{bmatrix}1 \\ -1  \end{bmatrix}.
\]
The image of this region under the exponential map $(e^x, e^y)$ covers $(0,1)\times (0,1)$. As the imaginary part is allowed to vary arbitrarily, we see that the image of $\hGamma_1(d) \bs \hX^0$ contains $D_1^0 \times D_1^0$. The interior of the  closure of this image contains $D_1 \times D_1$. The origin $(0,0)$ (which corresponds to the $\sigma_n$-stratum) is fixed by the entire $\hT=(d\cdot \Z)  \times (d\cdot \Z) \bs \C\times \C$ and in particular by the subgroup $(d\cdot \Z)  \times (d\cdot \Z) \bs \Z \times \Z$. Thus we are done. 

\begin{subsecremark}
The action of $\hGamma(1)/\hGamma(d)$ on $(\hGamma_1(d)\bs \hX^0_1)_{\hSigma_1^0}$ is interesting.  As $\hSigma $ projects onto $\Sigma$ under $U_1\times V_0 \rightarrow U_1$, combining our discussion above, we see that the two torus embeddings are related in the following way.

\begin{equation*}
\begin{tikzcd}
\hGamma_1(d)\bs \hX^0_1 
           \arrow[r, rightarrowtail]
           \arrow[d ]
&  (\hGamma_1(d)\bs \hX^0_1 )_{\hSigma_1^0} 
           \arrow[d, start anchor={[yshift=3pt]}]
&\arrow[r, Leftrightarrow, start anchor={[xshift=-20pt]}, end anchor= {[xshift=-15pt]}, yshift=-20pt]
& \hT=\bG_m\times \bG_m 
             \arrow[r, rightarrowtail]
              \arrow[d,  "{\pr_1}"] 
&     \hT_{\hSigma}
              \arrow[d]\\
\Gamma_1(d)\bs \mX^0_1 
         \arrow[r, rightarrowtail]
&  (\Gamma_1(d)\bs \mX^0_1 )_{\Sigma_1^0} 
&& T =\mathbb{G}_m
            \arrow[r, rightarrowtail]
&   T_\Sigma
\end{tikzcd}
\end{equation*}
The special fiber of $\hT_\hSigma \rightarrow T_\Sigma$ is a union of infinitely many $\mathbb{P}^1_\C$'s with each $\infty$  attached to the next $0$. This is the Tate curve (cf.~{\cite[{10.17}]{Pink}}). Our $\hGamma(1)/\hGamma(d)=\Z /(d\cdot \Z)\times \Z / (d\cdot \Z)$ acts on this fiber, fixing all singularities and rotating each piece of $\C^\times$ via multiplication by a $d$-th root of unity. This follows from the fact that, under the above torus embedding, the two pieces of $\C^\times$ are just the two coordinate axes in $\C \times \C$, while $(1,0) \mapsto (\zeta_d^{-(n+1)}, \zeta_d^n)$ and $(0,1) \mapsto (\zeta_d, \zeta_d^{-1})$ with $\zeta_d =e^{\frac{2\pi i}{d}}$.
\end{subsecremark}


\newpage
\vspace{10pt}
\printbibliography[heading=bibintoc, title={References}]

\end{document}